\documentclass[a4paper,reqno, 10pt]{amsart}   	%11pt
\usepackage[DIV=12, oneside]{typearea}			%DIV=12
\usepackage[utf8]{inputenc}
\usepackage[T1]{fontenc}
\usepackage[english]{babel}
\usepackage[centertags]{amsmath}
\usepackage{amstext,amssymb,amsopn,amsthm}
\usepackage{mathrsfs}
\usepackage{dsfont}
\usepackage{bbm}
\usepackage{thmtools}
\usepackage{graphicx}
\usepackage[titletoc,title]{appendix}
\usepackage{etexcmds}

\usepackage{multirow}

\usepackage{tikz}

\usepackage{pgfplots}% lädt auch tikz
\pgfplotsset{compat=newest}
\usetikzlibrary{calc, patterns,angles,quotes}
\usepgfplotslibrary{fillbetween}
\definecolor{c595959}{RGB}{89,89,89}
\definecolor{c5a5a5a}{RGB}{90,90,90}

\usepackage[colorlinks=true, linkcolor=black, citecolor=black]{hyperref}
\usepackage{enumitem}
\setlist[enumerate]{itemsep=0mm}

\addto\extrasenglish{}
\addto\extrasenglish{}
\addto\extrasenglish{}
\theoremstyle{plain}
\declaretheorem[title=Theorem, parent=section]{theorem}
\declaretheorem[title=Lemma,sibling=theorem]{lemma}
\declaretheorem[title=Proposition,sibling=theorem]{proposition}
\declaretheorem[title=Corollary,sibling=theorem]{corollary}

\theoremstyle{definition}
\declaretheorem[title=Definition,sibling=theorem]{definition}
\declaretheorem[title=Remark,sibling=theorem]{remark}
\declaretheorem[title=Remark, numbered=no]{remark*}

\declaretheorem[title=Assumption, numbered=no]{assumption*}

\numberwithin{equation}{section}
\newcommand{\N}{\mathds{N}}
\newcommand{\R}{\mathds{R}}
\newcommand{\C}{\mathds{C}}

\newcommand{\sF}{\mathscr{F}}

\def\hmath$#1${\texorpdfstring{{\rmfamily\textit{#1}}}{#1}}

\newcommand{\cO}{\mathcal{O}}

\newcommand{\cQ}{\mathcal{Q}}

\newcommand{\eps}{\varepsilon}

\newcommand{\BIGOP}[1]
{
	\mathop{\mathchoice%
		{\raise-0.22em\hbox{\huge $#1$}}%
		{\raise-0.05em\hbox{\Large $#1$}}{\hbox{\large $#1$}}{#1}}}

\def\XXint#1#2#3{{\setbox0=\hbox{$#1{#2#3}{\int}$}
		\vcenter{\hbox{$#2#3$}}\kern-.5\wd0}}

\newcommand{\BIGboxplus}{\mathop{\mathchoice%
		{\raise-0.35em\hbox{\huge $\boxplus$}}%
		{\raise-0.15em\hbox{\Large $\boxplus$}}{\hbox{\large $\boxplus$}}{\boxplus}}}

\newcommand{\I}[1]{(\text{I}_{#1})}
\newcommand{\II} [1] {(\text{II}_{#1})}
\newcommand{\III}[1]{(\text{III}_{#1})}
\newcommand{\IV}[1]{(\text{IV}_{#1})}
\newcommand{\V}[1]{(\text{V}_{#1})} 
\newcommand{\VI}[1]{(\text{VI}_{#1})}

\DeclareMathOperator{\sgn}{sgn}

\DeclareMathOperator{\pv}{p.v.}

\renewcommand{\d}{\textnormal{d}}

\newcommand{\1}{{\mathbbm{1}}}

\usepackage{esint}
\newcommand{\norm}[1]{\left\lVert#1\right\rVert} % norm 
\newcommand{\abs}[1]{\ensuremath{\left\vert#1\right\vert}} % absolute value
\newcommand{\vphi}{\varphi}

\newcommand{\meijer}[5]{%
	G^{#1}_{#2}\Big(\,\begin{matrix}#3\\#4\end{matrix}\,\Big\vert\,#5\Big)%
}

\newcommand{\pfqnr}[5]{%
	{_{#1}F_{#2}}\Big(\,\begin{matrix}#3\\#4\end{matrix}\,\Big\vert\,#5\Big)%
}

\begin{document}
	\allowdisplaybreaks
	\title[Fundamental solution to the fractional Kolmogorov equation]{Pointwise estimates of the fundamental solution to the fractional Kolmogorov equation}
	
	\author{Florian Grube}
			
	\address{Fakult{\"a}t f{\"u}r Mathematik, Universit{\"a}t Bielefeld, Postfach 10 01 31, 33501 Bielefeld, Germany}
	\email{fgrube@math.uni-bielefeld.de}
	
	\makeatletter
	\@namedef{subjclassname@2020}{%
		\textup{2020} Mathematics Subject Classification}
	\makeatother
	
	

	\begin{abstract} 
		We prove sharp two-sided estimates of the fundamental solution to the fractional Kolmogorov equation in $\R\times \R$ using Fourier methods. Additionally, we provide an explicit form of the fundamental solution in case of the square root of the Laplacian.
	\end{abstract}

	\maketitle
	
	\section{Introduction}\label{sec:intro}
	The fractional Kolmogorov equation in one spatial dimension reads 
	\begin{align}\label{eq:frac_kol}
		\begin{split}
			\partial_t u+ v\cdot \nabla_x u+ (-\Delta_v)^s u&=0  \,\,\,\text{ in }(0,\infty)\times \R\times \R, \\
			u(0,\cdot, \cdot)&= u_0  \text{ on } \R\times \R
		\end{split}
	\end{align}
	for the order of the diffusion $s\in (0,1)$. The function $u$ depends on times $t\in \R_+$, the spatial variable $x\in \R$, and the velocity $v\in \R$. The leading order term in the partial differential equation \eqref{eq:frac_kol} is the fractional Laplacian with respect to the velocity variable $v$. It is defined as an integro-differential operator 
	\begin{equation*}
		(-\Delta_v)^su(t,x,v):= c_{s}\, \pv \int\limits_{\R} \frac{u(t,x,v)-u(t,x,w)}{\abs{v-w}^{1+2s}}\d w.
	\end{equation*}
	Here, the positive constant $c_s$ is chosen such that the translation invariant operator $(-\Delta_v)^s$ admits the Fourier symbol $\abs{\phi}^{2s}$ where $\phi$ is the velocity frequency, see e.g.\ \cite{Buc16}. \smallskip 
	
	The aim of this article is to prove sharp bounds on the fundamental solution $p_s$ to the fractional Kolmogorov equation \eqref{eq:frac_kol}. \smallskip
	
	The fractional Kolmogorov equation can be viewed as a linearized model for the Boltzmann equation without cutoff, i.e.\
	\begin{equation*}
		\partial_t u + v\cdot \nabla_x u + \cQ u=0.
	\end{equation*}
	Here, $\cQ$ is the Boltzmann collision operator. It is a nonlinear nonlocal integro-differential operator, see \cite{Sil23} for a detailed survey. The Boltzmann equation enjoys much interest in mathematical research in particular due to its importance in physics, see e.g.\ \cite{Vil02}. As shown in the seminal series of works \cite{Sil16}, \cite{IMS20}, \cite{ImSi20}, \cite{ImSi21}, and \cite{ImSi22}, regularity theory for fractional order operators can be applied in order to study the Boltzmann equation. In these articles, a priori regularity estimates of solutions conditional to macroscopic bounds were established, see the survey \cite{ImSi20b}.\smallskip 
	
	Let us state the main result of this article. 	
	\begin{theorem}\label{th:bounds_time_1}
		Let $s\in (0,1)$. There exists a constant $C=C(s)\ge 1$ such that the fundamental solution $p_s$ to \eqref{eq:frac_kol} at the time $t=1$ satisfies the two-sided bound
		\begin{equation}\label{eq:bounds_time_1}
			p_s(1,x+\frac{v}{2},v)\asymp_C  \frac{1}{\big(1+\abs{x}+\abs{v}\big)^{2+2s}} \frac{1}{\big(1+(2\abs{x}-\abs{v})_+\big)^{2s}} 
		\end{equation}
		for any $x,v\in \R$.
	\end{theorem}
	The scaling properties that $p_s$ enjoys, see \eqref{eq:scaling}, yield the following two-sided bound for all times. 
	\begin{corollary}\label{cor:bounds_time_t}
		For any $s\in (0,1)$, there exists a constant $C=C(s)$ such that the fundamental solution to \eqref{eq:frac_kol} satisfies
		\begin{equation}\label{eq:bounds_time_t}
			p_s(t,x,v)\asymp_C \Big(\frac{1}{t^{1+1/s}} \wedge \frac{t^{2+2s}}{\abs{(x-\frac{tv}{2},tv)}^{2+2s}}\Big)\,\Big(1\wedge \frac{t^{1+2s}}{(\abs{x-\frac{tv}{2}}-\abs{\frac{tv}{2}})_+^{2s}}\Big).
		\end{equation}
		for any $t>0$ and $x,v\in \R$.
	\end{corollary}
	
	\begin{remark}
		While this work was being completed, the article \cite{HoZh24} was uploaded to arXiv. Therein, the authors prove sharp two-sided bounds for the fractional Kolmogorov equation \eqref{eq:frac_kol} in all dimensions using the underlying stochastic jump processes. More precisely, the authors apply the L{\'e}vy-It{\^{o}} decomposition and split the process into small jumps and large jumps. They notice that the behavior of the fundamental solution is governed by the large jumps. In the special case $d=1$, the bound for the fundamental solution $p_s(1,x,v)$ in \cite[Theorem 1.1, Remark 1.4]{HoZh24} reads
		\begin{equation*}
			\frac{1}{(1+\sqrt{x^2+v^2})^{2+2s}}\times\begin{cases}
				(1+\abs{x})^{-2s} & xv<0,\\
				1 & 0\le xv\le v^2,\\
				(1+\abs{x-v})^{-2s} & xv>v^2.
			\end{cases}
		\end{equation*}
		Note that this bound coincides with the bound obtained in \autoref{th:bounds_time_1}. This can be seen using 
		\begin{align*}
			\abs{x-v/2}+\abs{v}&\asymp \sqrt{x^2+v^2},\\
			\big(\abs{x-v/2}-\abs{v/2}\big)_+ &= \begin{cases}
				\abs{x} & \text{ if $xv<0$},\\
				0 &\text{ if $xv>0$ and $\abs{v}\ge \abs{x}$},\\
				\abs{x-v} &\text{ if $xv>0$ and $\abs{v}< \abs{x}$}.
			\end{cases}
		\end{align*}		
		We want to stress that the proof in \cite{HoZh24} is completely different from ours and that our article is independent of \cite{HoZh24}. While their proof uses tools from stochastic analysis, our proof relies only on Fourier analysis. Moreover, we provide an explicit representation of the fundamental solution $p_s(t,x,v)$ in the case $s=1/2$ in \autoref{sec:half_laplace}, see \autoref{th:fsol_d1_s_1/2}.
	\end{remark}
	In 1934, Kolmogorov studied the equation $\partial_t u+ v\cdot \nabla_x u-\Delta_v u=0$. He observed that the equation admits smooth solutions for initial data in $L^1(\R\times\R)+L^\infty(\R\times\R)$. This was immediate after computing the explicit fundamental solution to this equation. In \cite{Kol34}, he found that it equals a constant multiple of
	\begin{equation*}
		t^{-2}\exp\Big( -\frac{v^2}{4t}-3\frac{\abs{x-\frac{tv}{2}}^2}{t^3} \Big).
	\end{equation*}
	Thus, in the case $s=1$, the fundamental solution to \eqref{eq:frac_kol} admits Gaussian bounds which is in contrast to the polynomial bounds in \eqref{eq:bounds_time_1}. 
	
	As mentioned in the survey on nonlocal kinetic equations \cite[Section 5]{Sil23}, sharp two-sided bounds for the fundamental solution to \eqref{eq:frac_kol} were an open problem. In contrast, the parabolic counterpart, i.e.\ the fractional heat equation, is much better understood than its kinetic counterpart. In one dimension it reads
	\begin{equation}\label{eq:frac_heat}
		\begin{split}
			\partial_t u +(-\Delta)^su &=0 \,\,\,\text{ in } \R_+\times \R,\\
			u(0,\cdot)&=u_0 \text{ on }\R.
		\end{split}
	\end{equation}
	 In the early contribution \cite{Pol23}, the decay behavior of the fundamental solution $q_s(t,x)$ to \eqref{eq:frac_heat} was first analyzed. The author found that 
	 \begin{equation*}
	 	q_s(1,x)\asymp \frac{1}{\big(1+\abs{x}\big)^{1+2s}}.
	 \end{equation*}
 	The key idea in the proof in \cite{Pol23} is to analyze the behavior of the term $\abs{x}^{1+2s}q_s(1,x)$ at $x=+\infty$. Since the fundamental solution $q_s$ can easily be written as an inverse Fourier transformation, it suffices to consider
 	\begin{equation}\label{eq:help_frac_heat_behavior_infinity}
 		x^{1+2s}\int\limits_{0}^\infty \cos(x\xi)e^{-\xi^{2s}}\d \xi. 
 	\end{equation}
 	Cauchy's integral theorem allows to write this as a path integral in the complex plane. After appropriate manipulations of integral the limit $x\to +\infty$ of the term $\abs{x}^{1+2s}q_s(1,x)$ can be calculated. The same arguments are applied in \cite{BlGe60} to the fractional heat equation in arbitrary dimension $d\in \N$. The key observation is the radial symmetry of the problem. This enables one to reduce the problem to a one-dimensional problem and to calculate a single limit to deduce the result. \smallskip
 	
 	The proof of \autoref{th:bounds_time_1} borrows the ideas of \cite{Pol23}. The difficulty arises particularly from the missing radial symmetry of the problem \eqref{eq:frac_kol}. Even more so, the problem \eqref{eq:frac_kol} is highly degenerate in contrast to heat equations involving elliptic operators. Due to this lack of symmetry, the analysis of $p_s(1,x,v)$ at infinity is much more involved. In contrast to the one-dimensional problem as in \eqref{eq:help_frac_heat_behavior_infinity}, we need to consider all possible paths $\{(x(t),v(t))\mid t>0 \}\subset \R^2$ that tend to infinity, see \autoref{sec:analysis_at_infinity}. \smallskip
 	
 	It is known that the fractional heat kernel $q_s$ admits an explicit representation in terms of elementary functions in the case of the square root of the Laplacian, i.e.\ $s=1/2$. It coincides with the Cauchy kernel for the Laplace equation in the upper half space, i.e.\ 
 	\begin{equation}\label{eq:explicit_half_heat_equation}
 		q_{1/2}(x)= c\, \frac{t}{\big(\abs{x}^2+t^2\big)^{\frac{d+1}{2}}}.
 	\end{equation}
 	In \autoref{th:fsol_d1_s_1/2}, we provide a representation of the fundamental solution to the fractional Kolmogorov equation \eqref{eq:frac_kol} in the case $s=1/2$ in terms of elementary functions. In contrast to \eqref{eq:explicit_half_heat_equation}, the formula for $p_{1/2}(1,x,v)$ is much more complex. \smallskip
 	
 	The analysis of properties of fundamental solutions to Cauchy problems like \eqref{eq:frac_kol} have overreaching applications, e.g.\ in the study of existence and regularity of solutions.\medskip
 	
 	We give a short overview on the literature related to nonlocal kinetic equations. For a more detailed discussion, we refer to the survey \cite{Sil23}. The construction of fundamental solutions for Cauchy problems involving nonlocal hypoelliptic operators was treated in \cite{AIN24a}, \cite{AIN24b}. Results on $L^p$ maximal regularity was established by a series of authors \cite{ChZh18}, \cite{HMP19}, \cite{NiZa21}, \cite{NiZa22}, and \cite{Nie22}. In \cite{Sto19}, the author proves H{\"o}lder regularity for weak solutions to equations like \eqref{eq:frac_kol} using the De Giorgi method and the theory of averaging lemmas. A nonlocal kinetic Schauder theory was developed in \cite{HWZ20}, \cite{ImSi21}, and \cite{Loh23}. In \cite{Loh24}, the author establishes a nonlinear version of the Harnack inequality for weak solutions to equations like \eqref{eq:frac_kol}. Surprisingly, the strong Harnack inequality for nonlocal kinetic equations fails, see \cite{KaWe24}. Due to the close connection between the fractional Laplacian and stable processes, equations like \eqref{eq:frac_kol} can be studied using stochastic analysis. Such approaches to nonlocal kinetic and hypoelliptic equations via degenerate SDE's driven by stable process were applied in \cite{HM16}, \cite{ChZh18}, \cite{HWZ20}, and \cite{HPZ21}. In \cite{FRW24}, the authors study the regularity of solutions to the Boltzmann equation. In contrast to previous results they prove the boundedness of solutions under lower a priori assumptions on the solution.
\subsection{Outline}
In the \autoref{sec:prelims}, we fix notation and provide basic properties. We derive the fundamental solution via Fourier transformation in \autoref{sec:derive_fund_sol}. In \autoref{sec:invariances}, we simplify to problem and provide two representations of $p_s(1,x+v/2,v)$ using path integrals in the complex plane. These representations are the starting point of our analysis. We analyze $p_s(1,x+v/2,v)$ at infinity along specific paths $\{(x(t),v(t))\mid t>0\}\subset\R^2$ in \autoref{sec:analysis_at_infinity}. \autoref{sec:proof_main_th} is dedicated to the proof of \autoref{th:bounds_time_1}. In the last section, i.e.\ \autoref{sec:half_laplace}, we derive an explicit representation of the fundamental solution to \eqref{eq:frac_kol} in the case of the square root of the Laplacian, i.e.\ $s=1/2$. 

\subsection{Acknowledgements}
I am very grateful to Mateusz Kwa{\'s}nicki from whom I learned how to prove \autoref{lem:fasymp_d1_s_x_s_ge_12_positivity}. I would like to thank Moritz Kassmann for valuable comments on the manuscript. Financial support by the German Research Foundation (GRK 2235 - 282638148) is gratefully acknowledged. 

\section{Preliminaries}\label{sec:prelims}

We fix notation used throughout this work. We write $f(x)\asymp g(x)$ if there exists a constant $C\ge 1$ independent of $x$ such that $C^{-1}f(x)\le g(x)\le Cf(x)$. If the comparability constant $C$ is fixed beforehand, we write $f(x)\asymp_C g(x)$. The symbols $\sF f$, $\hat{f}$ refer to the space-velocity Fourier transformation of the function $f(x,v)$, $\sF^{-1}$ is its inverse. The Fourier transformation in $d\in \N$ dimensions and its inverse are given as
\begin{align*}
	\sF g(\xi)= \int\limits_{\R^d} e^{-i x\cdot \xi}g(x)\d x,\quad	\sF^{-1}g(x)= \frac{1}{(2\pi)^d}\int\limits_{\R^d} e^{i x \cdot \xi} g(\xi)\d \xi.
\end{align*}
The function $p_{s}(t,x,v)$ is the fundamental solution to the fractional Kolmogorov equation, see \eqref{eq:frac_kol}, and $q_{s}(t,x)$ is the fundamental solution to the fractional heat equation \eqref{eq:frac_heat}. We set $q_{s}(x):= q_{s}(1,x)$ and $k_s(x,v):= p_s(1,x+v/2,v)$. We denote Meijer's G-function by $\meijer{p,q}{m,n}{a}{b}{x}$ and $\pfqnr{p}{q}{a}{b}{x}$ is the generalized hypergeometric function, see e.g. \cite{Luk69i} or \cite{PBM90}. 

\begin{lemma}[{\cite[Theorem 1.3]{PrZa09},\cite[Proposition 2.1]{ImSi20}}]\label{lem:basic_properties}
	The function $p_{s}(1,\cdot, \cdot)\in C_b^\infty(\R\times \R)$ is nonnegative.
\end{lemma}
\begin{remark}
	The regularity of $p_s(1,\cdot, \cdot)$ in the previous lemma follows easily from the observation that the space-velocity Fourier transform of $p_s(1,\cdot, \cdot)$, i.e.\ $\exp(-\vphi_s(\xi,\phi))$, is integrable with any polynomial $\xi^\alpha\,\phi^\beta$, Fourier analysis methods, and Sobolev embeddings. 
\end{remark}

\subsection{Derivation of the fundamental solution in Fourier-space via the method of characteristics}\label{sec:derive_fund_sol}
After an application of the space-velocity Fourier transform, the equation \eqref{eq:frac_kol} reads
\begin{equation*}
	\partial_t \widehat{p}_s(t,\xi, \phi)=\xi \cdot \nabla_{\phi} \widehat{p}(t,\xi, \phi) - \abs{\phi}^{2s} \widehat{p}_s(t,\xi, \phi)\quad \text{ in } (0,\infty)\times \R\times \R,
\end{equation*}
with the initial condition $\widehat{p}(0,\cdot, \cdot)= 1$ on $\R \times \R$. Recall that $\widehat{p}$ refers to the space-velocity Fourier transform and $\xi$ is the frequency in the spatial variable $x$, $\phi$ is the velocity frequency related to $v$. \smallskip

This first-order differential equation is easily solved using the method of characteristics for fixed $\xi \in \R$. Fix some initial time $t_0\in \R_+$ and a velocity-frequency $\phi_0\in \R$. Along the path $\phi(t)= -\xi (t-t_0)+\phi_0$ the solution $\widehat{p}_s(t,\xi, \phi(t) )$ solves 
\begin{align*}
	\partial_t \widehat{p}_s(t,\xi, \phi(t) ) = -\abs{\phi(t)}^{2s}\, \widehat{p}_s(t,\xi, \phi(t)) \quad \text{ for any } t\in(t_0,\infty). 
\end{align*}
This ordinary differential equation is readily solved by $\widehat{p}_s(t,\xi, \phi(t))= \exp \big(-\int_{0}^t  \abs{\phi(u)}^{2s} \d u\big)$. Plugging in $t=t_0$ yields
\begin{align*}
	\widehat{p}_s(t_0, \xi, \phi_0)= \widehat{p}_s(t_0,\xi, \phi(t_0))=\exp \Big(-\int\limits_{0}^{t_0}  \abs{ -\xi(u-t_0)+\phi_0 }^{2s} \d u\Big)= \exp \Big(-\int\limits_{0}^{t_0}  \abs{ \phi_0 +u \xi}^{2s} \d u\Big)
\end{align*}
Therefore, the fundamental solution $p_s$ to \eqref{eq:frac_kol} is given by
\begin{align*}
	p_s(t,x,v)= \sF_{\xi, \phi}^{-1} \Big[ \exp \Big(-\int\limits_{0}^{t}  \abs{ \phi_0 +u \xi}^{2s} \d u\Big)  \Big] (x,v)
\end{align*}
As a direct consequence of this representation, the fundamental solution $p_s$ enjoys the following scaling property
\begin{equation}\label{eq:scaling}
	p_s(t,x,v)= \frac{1}{t^{1+1/s}} p_s\Big(1,\frac{x}{t^{1+1/(2s)}}, \frac{v}{t^{1/(2s)}} \Big).
\end{equation}

\subsection{Invariances and reduction}\label{sec:invariances}
We define the kernel $k_s(x,v)= p_s(1,x+v/2,v)$. Recall that this kernel is given by the inverse space-velocity Fourier transform:
\begin{align*}
	k_s(x,v):= \sF^{-1}_{\xi, \phi}\Big[ e^{-\int_{0}^{1} \abs{\phi+u\xi }^{2s}\d u } \Big](x+v/2,v).
\end{align*}
The function $k_s$ exhibits easier symmetries than $p_s$. More precisely, it satisfies for any $x,v\in \R$
\begin{equation}\label{eq:k_symmetry}
	k_s(-x,v)=k_s(x,v), \qquad k_s(x,v)=k_s(-x,-v).
\end{equation}
If we define
\begin{equation}\label{eq:phi_s}
	\vphi_s(\xi, \nu):=\int\limits_{0}^{1} \abs{\nu + (u-1/2)\xi}^{2s}\d u= \frac{(\nu+\xi/2)\abs{\nu+\xi/2}^{2s}- (\nu-\xi/2) \abs{\nu-\xi/2}^{2s}}{(2s+1)\xi},
\end{equation}
then 
\begin{equation}\label{eq:k_s}
	k_s(x,v)= \sF_{\xi,\nu}^{-1}\big[\exp\big( -\vphi_s(\xi,\nu)\big)\big](x,v).
\end{equation}

We will need the derivative of the exponent $\vphi_s$ with respect to both the spatial $\xi$ and the new velocity frequency $\nu$ in many of the following calculations. Thus, we provide it here. For any $\xi,\nu\in \R$ we have
\begin{align}
		\partial_\nu \vphi_{s}(\xi,\nu)&= \frac{\abs{\nu+\xi/2}^{2s}- \abs{\nu-\xi/2}^{2s}}{\xi},\label{eq:d_nu_phi}\\
		\partial_\nu^2 \vphi_{s}(\xi,\nu)&= 2s\frac{(\nu+\xi/2)\abs{\nu+\xi/2}^{2s-2}- (\nu-\xi/2)\abs{\nu-\xi/2}^{2s-2}}{\xi},\label{eq:d2_nu_phi}\\
		\partial_\xi \vphi_s(\xi,\nu)&=- \frac{(\nu+\xi/2)\abs{\nu+\xi/2}^{2s}- (\nu-\xi/2) \abs{\nu-\xi/2}^{2s}}{(2s+1)\xi^2}+\frac{\abs{\nu+\xi/2}^{2s}+\abs{\nu-\xi/2}^{2s}}{2\xi}.\label{eq:d_xi_phi}
\end{align} 
Furthermore, the function $(\xi,\nu)\mapsto \vphi_s(\xi,\nu)$ is continuously-differentiable. For $\xi$ near zero and $\nu \ne 0$, we expand $\vphi_s(2\eps, \nu)$ with respect to $\eps$ at $\eps=0$. This yields
\begin{equation}\label{eq:expansion_vphi_s_in_xi}
	\vphi_s(2\eps,\nu)= \nu^{2s}+2s(2s-1)\nu^{2s-2}\frac{\eps^2}{6} + \cO_\nu(\eps^4) \text{ as }\eps \to 0.
\end{equation}
From this expansion we derive
\begin{equation}\label{eq:vphi_s_0}
	\vphi_s(0,\nu)= \nu^{2s},\quad [\partial_{\xi}\vphi_s](0,\nu)=0, \text{ and }[\partial_{\xi}^2 \vphi_s](0,\nu)= \frac{2s(2s-1)}{12}\nu^{2s-2}.
\end{equation}
Furthermore, a minor calculation reveals $[\partial_\xi^4\vphi_s](0,1)= s(2s-1)(s-1)(2s-3)/20$.

For big $\nu$ we expand
\begin{equation}\label{eq:expansion_vphi_s_in_nu_infinity}
	\vphi_s(2,\nu)= \nu^{2s}+\frac{2s(2s-1)}{6}\nu^{2s-2}+ \cO(\nu^{2s-4})\text{ as }\nu \to \infty.
\end{equation}

Just as in \cite{Pol23}, we rewrite \eqref{eq:k_s} using path integrals in the complex plane, see \autoref{lem:representation_k_s_v} and \autoref{lem:representation_k_s_x}. We need the following curves in the complex plane.
\begin{definition}\label{def:paths}
	We define a few paths in $\C$: for $\theta\in (0,1)$, $n\in \N$
	\begin{align*}
		\gamma_\theta:(0,\infty)&\to \C, \qquad \gamma_\theta(t):= e^{i\pi \theta}t,\\
		\gamma_{\theta, n}^{(1)}:(1/n,n)&\to \C, \qquad \gamma_{\theta, n}^{(1)}(t):= e^{i\pi \theta}t,\\
		\gamma_{\theta, n}^{(2)}:(0,1)&\to \C, \qquad \gamma_{\theta, n}^{(2)}(t):= e^{i\pi \theta t}\frac{1}{n},\\
		\gamma_{\theta, n}^{(3)}:(1/n,n)&\to \C, \qquad \gamma_{\theta, n}^{(3)}(t):= t,\\
		\gamma_{\theta, n}^{(4)}:(0,1)&\to \C, \qquad \gamma_{\theta, n}^{(4)}(t):= e^{i\pi \theta(1-t)}n,
	\end{align*}
	and $\gamma_{\theta, n}$ is the concatenation of the paths $\gamma_{\theta, n}^{(1)}, \dots, \gamma_{\theta, n}^{(4)}$. Note that $\gamma_{\theta, n}$ is a closed path in $\C\setminus (-\infty,0]$.
\end{definition}

The following lemma provides a representation of $\abs{v}^{2+2s}\,k_s(x,v)$ which allows for a thorough analysis of the term in the limit $v\to \infty$. This is the starting point for most results in \autoref{sec:analysis_at_infinity}.
\begin{lemma}\label{lem:representation_k_s_v}
	Let $s\in (0,1)$. For any $x,v\in \R$, $v\ne 0$ the fundamental solution $k_s$ can be written as 	
	\begin{equation}\label{eq:v_limit_final_form_s}
		\pi^2 \abs{v}^{2+2s}k_s(x,v)= \frac{1}{2s}\Im\int\limits_{\gamma_\theta}  e^{ir^{1/2s}} r^{1/2s} \int\limits_{\R}\sgn(t)\frac{[\partial_\nu\vphi_s](2, t+\frac{2x}{v})}{\abs{t}^{1+2s}}e^{-\frac{r}{(v\abs{t})^{2s}}\vphi_s(2, t+\frac{2x}{v})}\d t\d r
	\end{equation}
	for any nonnegative $\theta$ satisfying $\theta <\min\{1/2,2s\}$.
\end{lemma}
\begin{proof}
	By the symmetry of $\vphi_s$, see \eqref{eq:k_symmetry}, we may assume without loss of generality that $x\ge 0$, $v>0$. We divide the proof into two steps. In the first step, we provide the necessary manipulations of the involved integrals for the representation \eqref{eq:v_limit_final_form_s}. Secondly, we rotate the path of integration to a line in the upper complex plane with a small angle to the positive real axis. \smallskip
	
	\textbf{Step 1:} We recall the representation of $k_s$ in \eqref{eq:k_s}. The symmetry of $\vphi_{s}$ yields
	\begin{align*}
		\pi^2 v^{2+2s}k_s(x,v)&=v^{2+2s}\int\limits_{0}^\infty\int\limits_{0}^\infty \cos(x\xi)\cos(v \nu) e^{-\vphi_s(\xi,\nu)}\d \xi \d \nu.
	\end{align*}
	Now, we integrate the previous integral with respect to $\nu$ by parts. To argue that the boundary terms vanish, we note that for any $r,t\ge 0$
	\begin{equation}\label{eq:lower_bound_vphi_s_d_1}
		(2s+1)\,\vphi_s(r,t)\ge \frac{\abs{t+\frac{r}{2}}^{2s} }{2}\ge \frac{\max\{ t^{2s},r^{2s} \}}{2^{1+2s}}\ge \frac{t^{2s}+r^{2s}}{4^{1+s}}.
	\end{equation}
	Using this, the upcoming boundary terms will vanish since
	\begin{align*}
		\big| \sin(v\nu) \int\limits_{0}^\infty e^{-\vphi_s(\xi,\nu)}\d \xi\big| &\le v\nu e^{-\frac{\nu^{2s}}{(2s+1)4^{1+s}}}\,\int\limits_{0}^\infty e^{-\frac{\xi^{2s}}{(2s+1)4^{1+s}}}\d \xi \to 0 
	\end{align*}
	as $\nu \to 0$ and as $\nu \to \infty$. Thus, integration by parts yields
	\begin{align*}
		\pi^2 v^{2+2s}k_s(\frac{v}{2},v)&=v^{1+2s}\int\limits_{0}^\infty\int\limits_{0}^\infty \sin(v \nu)\cos(x\xi) \partial_\nu\vphi_s(\xi,\nu)e^{-\vphi_s(\xi,\nu)}\d \xi \d \nu.
	\end{align*}
	We split the integral with respect to $\xi$ into two parts. \medskip 
	
	\textbf{Part 1: ($2\nu>\xi$)} We use the change of variables $\xi = 2\nu /t$ to write
	\begin{align*}
		\I{}&:=v^{1+2s}\int\limits_{0}^\infty\int\limits_{0}^{2\nu} \sin(v \nu)\cos(x\xi) \partial_\nu\vphi_s(\xi,\nu)e^{-\vphi_s(\xi,\nu)}\d \xi \d \nu\\
		&=v^{1+2s}\int\limits_{0}^\infty\int\limits_{1}^{\infty} \frac{2\nu}{t^2}\sin(v \nu)\cos(\frac{2x\nu}{t} ) [\partial_\nu\vphi_s](\frac{2\nu}{t},\nu)e^{-\vphi_s(2\nu /t,\nu)}\d t \d \nu\\
		&=\frac{1}{s}\int\limits_{0}^\infty\int\limits_{1}^{\infty} \sin(v \nu)\cos(\frac{2x\nu}{t} )\frac{2s\nu^{2s-1}v^{2s}}{t^{2s}} \frac{v\nu}{t} [\partial_\nu\vphi_s](2,t)e^{-\frac{\nu^{2s}}{t^{2s}}\vphi_s(2,t)}\d t \d \nu.\\
		\intertext{After the change of variables $(\nu v)^{2s}=t^{2s}r$, this term equals}
		&= \frac{1}{s}\int\limits_{0}^\infty\int\limits_{1}^{\infty} \sin(t r^{1/2s})\cos(\frac{2x}{v}r^{1/2s} ) r^{1/2s} [\partial_\nu\vphi_s](2,t)e^{-\frac{r}{v^{2s}}\vphi_s(2,t)}\d t \d r.
	\end{align*}
	
	\textbf{Part 2: ($2\nu\le \xi$)} We scale the variable $\nu$ by $\xi/2$ and write
	\begin{align*}
		\II{}&:= v^{1+2s}\int\limits_{0}^\infty\int\limits_{0}^{\xi/2} \sin(v \nu)\cos(x\xi) \partial_\nu\vphi_s(\xi,\nu)e^{-\vphi_s(\xi,\nu)}\d \nu\d \xi \\
		&=v^{1+2s}\int\limits_{0}^\infty\int\limits_{0}^{1} \frac{\xi}{2} \sin(\frac{\xi v }{2}t)\cos(x\xi) [\partial_\nu\vphi_s](\xi,\xi t/2)e^{-\vphi_s(\xi,\xi t/2)}\d t\d \xi \\
		&=\frac{1}{s}\int\limits_{0}^\infty\int\limits_{0}^{1}  \sin(\frac{\xi v }{2}t)\cos(x\xi) \frac{2s\xi^{2s-1}v^{2s}}{2^{2s}} \frac{\xi v}{2} [\partial_\nu\vphi_s](2, t)e^{-\frac{\xi^{2s}}{2^{2s}}\vphi_s(2, t)}\d t\d \xi.\\
		\intertext{Changing the variables $(\xi v)^{2s}= 2^{2s}r$ yields}
		\II{}&= \frac{1}{s}\int\limits_{0}^\infty\int\limits_{0}^{1}  \sin(tr^{1/2s})\cos(\frac{2x}{v}r^{1/2s}) r^{1/2s} [\partial_\nu\vphi_s](2, t)e^{-\frac{r}{v^{2s}}\vphi_s(2, t)}\d t\d r.
	\end{align*}
	Now, we combine $\I{}$ and $\II{}$ and find using the trigonometric identities that $\pi^2 v^{2+2s}k_s(x,v)$ equals
	\begin{align*}
		\I{}+\II{}&= \frac{1}{2s}\int\limits_{0}^\infty\int\limits_{0}^{\infty}  \Big(\sin\big((t+\frac{2x}{v})r^{1/2s}\big)+\sin\big((t-\frac{2x}{v})r^{1/2s}\big)\Big) r^{1/2s} [\partial_\nu\vphi_s](2, t)e^{-\frac{r}{v^{2s}}\vphi_s(2, t)}\d t\d r\\
		&= \frac{1}{2s}\int\limits_{0}^\infty\int\limits_{\R}  \sin\big((t+\frac{2x}{v})r^{1/2s}\big) r^{1/2s} [\partial_\nu\vphi_s](2, t)e^{-\frac{r}{v^{2s}}\vphi_s(2, t)}\d t\d r.
	\end{align*}
	In the previous equality, we used the observation that $[\partial_\nu\vphi_s](2, -t)= -[\partial_\nu\vphi_s](2, t)$. A translation of $t$ by $-2x/v$ yields
	\begin{equation*}
		\I{}+\II{}= \frac{1}{2s}\int\limits_{0}^\infty\int\limits_{\R}  \sin\big(tr^{1/2s}\big) r^{1/2s} [\partial_\nu\vphi_s](2, t+2x/v)e^{-\frac{r}{v^{2s}}\vphi_s(2, t+2x/v)}\d t\d r.
	\end{equation*}
	After scaling $r$ by $\abs{t}^{-2s}$ this is equal to
	\begin{align*}
		\frac{1}{2s}&\int\limits_{0}^\infty\int\limits_{\R}  \sgn(t)\sin\big(r^{1/2s}\big) \frac{r^{1/2s}}{\abs{t}^{1+2s}} [\partial_\nu\vphi_s](2, t+2x/v)e^{-\frac{r}{\abs{t}^{2s}v^{2s}}\vphi_s(2, t+2x/v)}\d t\d r\\
		&= \Im\frac{1}{2s}\int\limits_{0}^\infty\int\limits_{\R}  \sgn(t)e^{ir^{1/2s}} \frac{r^{1/2s}}{\abs{t}^{1+2s}} [\partial_\nu\vphi_s](2, t+2x/v)e^{-\frac{r}{\abs{t}^{2s}v^{2s}}\vphi_s(2, t+2x/v)}\d t\d r. 
	\end{align*}
	\textbf{Step 2:} Now, we rotate the integral with respect to $r$ from the positive real line to a line in the upper complex plane by a sufficiently small angle. Recall the definition of the paths in \autoref{def:paths}. After changing the order of integration, it suffices to prove that for any $c>0$ and any $0<\theta<\min\{ 1/2,s \}$ the equality
	\begin{equation*}
		\int\limits_{0}^\infty e^{ir^{1/2s}} r^{1/2s}e^{-cr}\d r= \int\limits_{\gamma_\theta} e^{iz^{1/2s}} z^{1/2s}e^{-cz}\d z
	\end{equation*}
	holds. \smallskip
	
	 Note that the function $\C\ni z \mapsto e^{iz^{1/2s}}z^{1/2s}e^{-c\, z}$ is holomorphic in $\C\setminus (-\infty, 0]$. By the Cauchy integral theorem, see e.g.\ \cite{Wal33}, the path integral over the closed curve $\gamma_{\theta, n}$ for any $0<\theta<1/2$ and $n\in \N$
	\begin{equation*}
		\int\limits_{\gamma_{\theta, n}} e^{iz^{1/2s}}z^{1/2s}e^{-cz} \d z = 0.
	\end{equation*}
	Since $\int_{\gamma_{\theta, n}^{(1)}}\to \int_{\gamma_\theta}$ and $\int_{\gamma_{\theta, n}^{(3)}} \to \int_{[0,\infty)}$, it remains to show that the curve integrals over $\gamma_{\theta, n}^{(2)}$ and $\gamma_{\theta, n}^{(4)}$ converge to zero. On the path $\gamma_{\theta, n}^{(2)}$, we estimate under the assumption $\theta< \min\{ s,1/2 \}$
	\begin{equation*}
		\big| \int\limits_{\gamma_{\theta, n}^{(2)}}  e^{iz^{1/2s}}z^{1/2s}e^{-cz} \d z \big|\le \frac{\pi \theta}{n^{1+1/2s}}\int\limits_{0}^1  \exp\big( -\frac{\sin(\pi\theta r/(2s))}{n^{1/2s}}  -\frac{c\cos(\pi r \theta )}{n}\big) \d r\le \frac{\pi}{2n^{1+1/2s}}\to 0 
	\end{equation*}
	as $n\to \infty$. The integral over the path $\gamma_{\theta, n}^{(4)}$ becomes
	\begin{equation*}
		\big| \int\limits_{\gamma_{\theta, n}^{(4)}}  e^{iz^{1/2s}}z^{1/2s}e^{-cz} \d z \big|\le \pi \theta n^{1+1/2s} \int\limits_{0}^1 \exp\Big(-\sin(\frac{\pi \theta t}{2s})n^{1/2s}-c\cos(\pi\theta t) n\Big) \d r=:\III{}.
	\end{equation*}
	If $\theta\le 2s$, then $\sin(\pi\theta t/(2s))\ge 0$ for any $0<t<1$. If $\theta <1/2$, then $\cos(\pi \theta t)\ge \cos(\pi \theta )>0$ for any $0<t<1$. Thus, we estimate $\III{}$ by 
	\begin{align*}
		\pi \theta n^{1+1/2s}\exp\Big(-c n \cos(\pi \theta )\Big) \to 0
	\end{align*}
	as $n\to \infty$. Therefore, we can move the path of integration from the positive real line to a path in the positive half space of the complex number with a slight angle $\theta$ to the real axis.
	
\end{proof}

\begin{lemma}\label{lem:representation_k_s_x}
	Let $s\in (0,1)$. For any $x,v\in \R$, $x\ne 0$ the fundamental solution $k_s$ can be written as 	
	\begin{equation}\label{eq:representation_k_s_x}
		\pi^2 \abs{x}^{2+2s}k_s(x,v)=\frac{1}{2s}\Im\int\limits_{\gamma_{\theta}} e^{i2r^{1/2s}}r^{1/2s}\int\limits_{\R} \sgn(t) \frac{[\partial_\xi \vphi_s](2t+\frac{v}{x},1)}{\abs{t}^{1+2s}} e^{-\frac{r}{(x\abs{t})^{2s}}\vphi_s(2t+\frac{v}{x},1)}\d t\d r
	\end{equation}
	for any nonnegative $\theta$ satisfying $\theta <\min\{1/2,2s\}$.
\end{lemma}
\begin{proof}
	Due to \eqref{eq:k_symmetry}, we assume without loss of generality both $x$ and $v$ to be nonnegative. By \eqref{eq:k_s}, we write
	\begin{align*}
		\pi^2 x^{2+2s}k_s(x,v)&= x^{2+2s}\int\limits_{0}^\infty\int\limits_{0}^\infty \cos(x\xi)\cos(v \nu) e^{-\vphi_s(\xi,\nu)}\d \xi \d \nu,\\
		\intertext{which equals after integrating by parts}
		&= x^{1+2s}\int\limits_{0}^\infty  \Bigg(\,\int\limits_{0}^{2\nu}+\int\limits_{2\nu}^\infty\Bigg) \sin(x\xi)\cos(v \nu)[\partial_\xi \vphi_s](\xi,\nu) e^{-\vphi_s(\xi,\nu)}\d \xi \d \nu= \I{}+\II{}.
	\end{align*}
	Next, we will prove a representation of $x^{2+4s}k_s(x,v)$ similar to that proven in \autoref{lem:representation_k_s_v}. Note that here $v$ could be zero. \smallskip
	
	\textbf{Part 1: ($2\nu>\xi$)} We use the change of variables $\xi=2\nu/t$ to write
	\begin{align*}
		\I{}&:= x^{1+2s}\int\limits_{0}^\infty\int\limits_{0}^{2\nu} \sin(x\xi)\cos(v \nu)[\partial_\xi \vphi_s](\xi,\nu) e^{-\vphi_s(\xi,\nu)}\d \xi \d \nu\\
		&= x^{1+2s}\int\limits_{1}^\infty\int\limits_{0}^\infty \sin(\frac{2x\nu}{t})\cos(v \nu)[\partial_\xi \vphi_s](\frac{2\nu}{t},\nu) \frac{2\nu}{t^2} e^{-\vphi_s(2\nu/t,\nu)} \d \nu\d t\\
		&= \frac{1}{s}\int\limits_{1}^\infty\int\limits_{0}^\infty \sin(\frac{2x\nu}{t})\cos(v \nu)[\partial_\xi \vphi_s](2,t)\big[ 2s\frac{x^{2s}}{t^{2s}}\nu^{2s-1} \big]\frac{x\nu}{t} e^{-\frac{\nu^{2s}}{t^{2s}}\vphi_s(2,t)} \d \nu\d t.
	\end{align*}
	We proceed by changing the variable $(x\nu)^{2s}=t^{2s}r$. This yields
	\begin{align*}
		\I{}&= \frac{1}{s}\int\limits_{1}^\infty\int\limits_{0}^\infty \sin(2r^{1/2s})\cos(\frac{v}{x} t r^{1/2s})r^{1/2s} e^{-\frac{r}{x^{2s}}\vphi_s(2,t)} \d r\d t.
	\end{align*}
	\textbf{Part 2: ($2\nu\le \xi$)} We scale $\nu$ by $\xi/2$ which yields
	\begin{align*}
		\II{}&:= x^{1+2s}\int\limits_{0}^\infty\int\limits_{0}^{\xi/2} \sin(x\xi)\cos(v \nu)[\partial_\xi \vphi_s](\xi,\nu) e^{-\vphi_s(\xi,\nu)}\d \nu\d \xi \\
		&= x^{1+2s}\int\limits_{0}^\infty\int\limits_{0}^{1} \sin(x\xi)\cos(v/2 \xi t)[\partial_\xi \vphi_s](\xi,\xi/2 t) \frac{\xi}{2} e^{-\vphi_s(\xi,\xi t /2)}\d t\d \xi.\\
		\intertext{Now, we scale $\xi$ by $2$ to find}
		\II{}&=\frac{1}{s}\int\limits_{0}^\infty\int\limits_{0}^{1} \sin(2x\xi)\cos(v \xi t)[\partial_\xi \vphi_s](2,t) \big[2s x^{2s}\xi^{2s-1}\big] x\xi e^{-\xi^{2s}\vphi_s(2,t)}\d t\d \xi.
	\end{align*}
	Finally, we use the change of variables $(x\xi)^{2s}=r$. This yields
	\begin{align*}
		\II{}&= \frac{1}{s}\int\limits_{0}^\infty\int\limits_{0}^{1} \sin(2r^{1/2s})\cos(\frac{v}{x} t r^{1/2s})[\partial_\xi \vphi_s](2,t) r^{1/2s} e^{-\frac{r}{x^{2s}}\vphi_s(2,t)}\d t\d r.
	\end{align*}
	Therefore, we have proven that
	\begin{equation*}
		\pi^2 x^{2+2s}k_s(x,v)=\frac{1}{s}\int\limits_{0}^\infty\int\limits_{0}^{\infty} \sin(2r^{1/2s})\cos(\frac{v}{x} t r^{1/2s})[\partial_\xi \vphi_s](2,t) r^{1/2s} e^{-\frac{r}{x^{2s}}\vphi_s(2,t)}\d t\d r.
	\end{equation*}
	Now, we scale $r$ by $1/t^{2s}$ and, thereafter, use the change of variables $t\mapsto 1/t$. This, together with the trigonometric identities, yields
	\begin{align*}
		\pi^2 x^{2+2s}k_s(x,v)&=\frac{1}{2s}\Im\int\limits_{0}^\infty\int\limits_{0}^{\infty} \Big(e^{i(2t+\frac{v}{x})r^{1/2s}}+e^{i(2t-\frac{v}{x})r^{1/2s}}\Big)[\partial_\xi \vphi_s](2t,1) r^{1/2s} e^{-\frac{r}{x^{2s}}\vphi_s(2t,1)}\d t\d r\\
		&=\frac{1}{2s}\Im\int\limits_{0}^\infty\int\limits_{\R} e^{i2tr^{1/2s}}[\partial_\xi \vphi_s](2t+\frac{v}{x},1) r^{1/2s} e^{-\frac{r}{x^{2s}}\vphi_s(2t+\frac{v}{x},1)}\d t\d r\\
		&=\frac{1}{2s}\Im\int\limits_{0}^\infty\int\limits_{\R} \sgn(t)e^{i2r^{1/2s}}[\partial_\xi \vphi_s](2t+\frac{v}{x},1) \frac{r^{1/2s}}{\abs{t}^{1+2s}} e^{-\frac{r}{(x\abs{t})^{2s}}\vphi_s(2t+\frac{v}{x},1)}\d t\d r.
	\end{align*}
	Now, we rotate the path of integration to $\gamma_{\theta}$ with the same arguments as in the proof of step 2 in \autoref{lem:representation_k_s_v}.
\end{proof}

\subsection{Some specific path integrals}

The next lemma provides the final calculation in many results in \autoref{sec:analysis_at_infinity}. This identity is well known and was proved in some special cases e.g.\ in \cite[Equation (6)]{Pol23}. To keep the manuscript self contained, we provide the proof here.

\begin{lemma}\label{lem:r_integral_gamma_function}
	The following identity holds for any $s\in (0,1)$, $\alpha>-1-1/2s$, and $\theta \in (0, 2s)$ 
	\begin{align*}
		\Im\frac{1}{2s}\int\limits_{\gamma_{\theta}}e^{ir^{1/2s}}r^{1/2s+\alpha}\d r= \cos(\pi s(1+\alpha))\Gamma(2s(1+\alpha)+1).
	\end{align*}
\end{lemma}
\begin{proof}
	Firstly, note that the function $\C\ni r \mapsto e^{ir^{1/2s}}r^{1/2s+\alpha}$ is holomorphic in $\C\setminus (-\infty, 0]$. Now, using Cauchy's integral theorem, we rotate the path of integration to $\gamma_{s}$. This step is legal after an appropriate approximation scheme as done in the second step in the proof of \autoref{lem:representation_k_s_v}. Now, we write
	\begin{equation*}
		\Im\frac{1}{2s}\int\limits_{\gamma_{\theta}}e^{ir^{1/2s}}r^{1/2s+\alpha}\d r= \frac{1}{2s}\Re \int\limits_{0}^\infty e^{i\pi s(1+\alpha)}e^{-r^{1/2s}}r^{1/2s+\alpha}\d r= \cos(\pi s(1+\alpha))\int\limits_{0}^\infty e^{-y}y^{2s-1}y^{1+2s\alpha}\d y.
	\end{equation*}
	In the last equality, we used the change of variables $r^{1/2s}=y$.
\end{proof}

The following lemma is an essential tool in the proofs of \autoref{prop:fasymp_d1_s_x/v_le12} and \autoref{prop:fasymp_d1_s_x_s_ge_12}. The cancellation in the case $k>1/(2s)$ gained from \autoref{lem:additional_cancelation} enables us to calculate the behavior of the fundamental solution $p_s(1,x+v/2,v)$ at infinity in many cases.
\begin{lemma}\label{lem:additional_cancelation}
	For any $s\in (0,1)$, $k\ge 1/2s$, $c>0$, and $\theta \in (0,\min\{1/2,2s\})$ the following identity holds:
	\begin{align*}
		\Im \int\limits_{0}^\infty \frac{1}{t^{2sk}} \int\limits_{\gamma_{\theta}}e^{ir^{1/2s}}r^{1/2s+k-1} e^{-\frac{cr}{t^{2s}}} \d r \d t= \begin{cases}
			0 &, k>1/2s\\
			\frac{\pi s}{2} &, k=1/2s
		\end{cases}.
	\end{align*}
\end{lemma}
\begin{proof}
	With arguments similar to those in \autoref{lem:representation_k_s_v}, it suffices to consider the case $\theta = s/2$. Otherwise we rotate the path of integration of the holomorphic function to $\gamma_{s/2}$ using Cauchy's integral theorem. We scale the variable $r$ by $t^{2s}$ 
	\begin{align*}
		\Im \int\limits_{0}^\infty \int\limits_{\gamma_{\theta}}e^{ir^{1/2s}}r^{1/2s+k-1} \frac{1}{t^{2ks}} e^{-\frac{cr}{t^{2s}}}  \d r \d t &=	\Im\int\limits_{\gamma_{\theta}}r^{1/2s+k-1} e^{-cr}\int\limits_{0}^\infty t e^{itr^{1/2s}}  \d t\d r\\
		= \Im\int\limits_{\gamma_{\theta}}r^{1/2s+k-1} e^{-cr}\bigg[ 0+ \frac{1}{i^2r^{1/s}} \bigg]\d r
		&= -c^{1/2s-k}\Im\int\limits_{\gamma_{\theta}}r^{-1/2s+k-1} e^{-r} \d r.
	\end{align*}
	In the last inequality, we scaled the variable $r$ by $1/c$. Note that 
	\begin{align*}
		\I{}:=-\Im\int\limits_{\gamma_{\theta}}r^{-1/2s+k-1} e^{-r} \d r&= -\Im\int\limits_{0}^\infty e^{i\pi\theta(k-1/2s)}r^{-1/2s+k-1} e^{-e^{i\pi \theta}r} \d r\\
		&= \Im \frac{i-1}{\sqrt{2}}e^{i\pi s k/2} \cos(\pi s /2)^{1/2s-k} \int\limits_{0}^\infty e^{-r}r^{-1/2s+k-1} e^{-i\tan(\frac{\pi s}{2})r} \d r.
	\end{align*}
	Here, we scaled the variable $r$ by $1/\cos(\pi s/2)$. Now, we distinguish two cases. 
	
	\textbf{Case ($k>1/2s$):} Let's denote for $t\ge 0$ 
	\begin{equation*}
		h(t):= \int\limits_{0}^\infty e^{-r}r^{-1/2s+k-1}e^{-itr}\d r.
	\end{equation*}
	Differentiation with respect to $t$ reveals that this function satisfies the differential equation
	\begin{equation*}
		(1+it)h'(t)=-i(k-1/2s)h(t)
	\end{equation*}
	with $h(0)= \int_0^\infty e^{-r}r^{-1/2s+k-1}\d r >0$. This equation is uniquely solved by
	\begin{align*}
		h(t)= h(0)\sqrt{1+t^2}^{-k+1/2s}e^{-i(k-1/2s)\arctan(t)}. 
	\end{align*}
	We conclude that 
	\begin{align*}
		\I{}&= h(0) \Im \frac{i-1}{\sqrt{2}}e^{i\pi s k/2} \cos(\pi s /2)^{1/2s-k} \sqrt{1+\tan(\pi s/2)^2}^{-k+1/2s}e^{-i(k-1/2s)\frac{\pi s}{2}}=0.
	\end{align*}
	\textbf{Case $k=1/2s$:} In this case, the term $\I{}$ simplifies to $-\Im \int_{0}^\infty \exp(-r) r^{-1} \exp(-i\tan(\frac{\pi s}{2})r) \d r$. Let's define for any $t>0$
	\begin{equation*}
		h(t)= -\Im \int\limits_{0}^\infty e^{-r}r^{-1} e^{-itr} \d r.
	\end{equation*}
	Differentiation with respect to $t$ reveals
	\begin{equation*}
		h'(t)= \Re \int\limits_{0}^\infty e^{-r} e^{-itr} \d r= \Re \frac{-1}{-1-it}= \frac{1}{1+t^2}.
	\end{equation*}
	Thus, integrating the function $h'$ from $0$ to $\tan(\pi s /2)$ yields
	\begin{equation*}
		h(\tan(\pi s/2))= \int\limits_{0}^{\tan(\pi s /2)} \frac{1}{1+t^2}\d t = \frac{\pi s}{2}.
	\end{equation*}
\end{proof}

\section{Analysis at infinity along special paths}\label{sec:analysis_at_infinity}
For any $s\in (0,1)$, we define the function $j_s:\R\times \R\to (0,\infty)$ by
\begin{equation}\label{eq:def_j_s}
	j_s(x,v):= \big(1+\abs{x}^{2+2s}+\abs{v}^{2+2s}\big)\Big( 1+(2\abs{x}-\abs{v})_+\Big)^{2s}.
\end{equation} 
Note that $1/j_s(x,v)$ is the proposed decay of the kernel $k_s(x,v)$, see \autoref{th:bounds_time_1}.\smallskip

In this section, we analyze the behavior of $j_s(x,v) k_s(x,v)$ at infinity. Since $k_s$ is not radial, we need to distinguish different paths $\{(x(t),v(t))\mid t>0\}\subset \R^2$ that tend to infinity.

\subsection{The velocity dominant regime}\label{sec:analysis_at_infinity_velo_dom}
In this subsection, we consider paths $\{(x(t),v(t))\mid t>0 \}\subset \R^2$ where the velocity variable $v(t)$ dominates the spatial variable $x(t)$, more precisely $2\abs{x(t)}\le \abs{v(t)}$ near infinity. 

\begin{proposition}\label{prop:fasymp_d1_s_x/v_le12}
	Let $s\in (0,1)$, $\kappa\in[0,\frac{1}{2}]$, and $\iota\in[-\infty,\infty]$. For any family of vectors $\{\big(x(t),v(t)\big)\mid  t\in \R_+ \}\subset \R^{2}$ such that $\abs{v(t)}\to \infty$ as $t\to \infty$ and 
	\begin{equation}\label{eq:d_1_s_limit_x/v_le_12}
		\lim\limits_{t\to \infty}\frac{\abs{x(t)}}{\abs{v(t)}}= \kappa\quad \text{and}\quad \lim\limits_{t\to \infty} \frac{\abs{v(t)}}{2}-\abs{x(t)}= \iota, 
	\end{equation}
	we have
	\begin{align*}
		\lim\limits_{t\to \infty} j_s(x(t),v(t))\, k_s(x(t),v(t))=C_{s,1}(\kappa,\iota),
	\end{align*}
	where
	\begin{equation*}
		C_{s,1}(\kappa,\iota):=\big(1+\kappa^{2+2s}\big)2^{2s}\frac{ \sin(\pi s)\Gamma(2s+1)}{\pi} \begin{cases}
			(\frac{1}{2}+\iota_-)^{2s} \int\limits_{-\infty}^{(1+2s)^{\frac{1}{2s}}\iota}q_s(b)\d b&\text{ if }\iota>-\infty,\\
			\frac{\Gamma(2s)\sin(\pi s)}{(2s+1)\pi} &\text{ if }\iota=-\infty.
		\end{cases}
	\end{equation*}
	Here, $q_s$ is the fundamental solution to the fractional heat equation \eqref{eq:frac_heat} in one spatial dimension at time $t=1$. 
\end{proposition}
\begin{proof}
	By symmetry of $k_s$ in \eqref{eq:k_symmetry}, we assume without loss of generality that $v>0$ and $x\ge 0$. If $\iota \ne 0$, then we also assume $\sgn(v-2x)=\sgn(\iota)$. We simply write $x\to \infty$ whenever we refer to the limit $t\to \infty$ and $x=x(t), v=v(t)$. We distinguish two cases.

	\textbf{Case ($\iota>-\infty$):} Firstly, we may simplify the limit by noticing that 
	\begin{align*}
		\frac{j_s(x,v)}{v^{2+2s}}\to \big(1+\kappa^{2+2s}\big) \Big(\frac{1+2\abs{\iota}}{1+2\iota_+}\Big)^{2s}
	\end{align*}
	as $v\to \infty$. Thus, it is sufficient to consider
	\begin{align*}
		v^{2+2s}\pi^2 k_s(x,v)&= \Im\frac{1}{2s}\int\limits_{0}^\infty e^{ir^{1/2s}}r^{1/2s}\int\limits_{\R}  \frac{\sgn(t)}{\abs{t}^{1+2s}} [\partial_\nu\vphi_s](2, t+2x/v)e^{-\frac{r}{\abs{t}^{2s}v^{2s}}\vphi_s(2, t+2x/v)}\d t\d r.
	\end{align*}
	which follows from \autoref{lem:representation_k_s_v} with $\theta=0$.
	We fix some $r>0$. We define
	\begin{align*}
		\I{}&:=\int\limits_{\R}  \frac{\sgn(t)}{\abs{t}^{1+2s}} [\partial_\nu\vphi_s](2, t+2x/v)e^{-\frac{r}{\abs{t}^{2s}v^{2s}}\vphi_s(2, t+2x/v)}\d t\\
		&= \int\limits_{0}^\infty  \frac{\abs{2x/v+1+t}^{2s}e^{-\frac{r}{t^{2s}v^{2s}}\vphi_s(2, t+2x/v)}-\abs{2x/v+1-t}^{2s}e^{-\frac{r}{t^{2s}v^{2s}}\vphi_s(2, -t+2x/v)}}{2t^{1+2s}} \d t\\
		&\quad + \int\limits_{0}^\infty  \frac{\abs{2x/v-1-t}^{2s}e^{-\frac{r}{t^{2s}v^{2s}}\vphi_s(2, -t+2x/v)}-\abs{2x/v-1+t}^{2s}e^{-\frac{r}{t^{2s}v^{2s}}\vphi_s(2, t+2x/v)}}{2t^{1+2s}} \d t\\
		&=: \II{}+\III{}.
	\end{align*}
	Here, we used \eqref{eq:d_nu_phi}. \medskip 
	
	\textit{Analysis of the term $\II{}$:} In the term $\II{}$ we scale the variable $t$ by $2x/v+1$. Thereafter, we integrate $\II{}$ together with $e^{ir^{1/2s}}r^{1/2s}$ over $(0,\infty)$ and scale $r$ by $t^{2s}$. This yields
	\begin{align*}
		\Im \int\limits_{0}^\infty e^{ir^{1/2s}}r^{1/2s}\II{}\d r 
		&=\frac{1}{2} \Im \int\limits_{0}^\infty \int\limits_{0}^\infty e^{itr^{1/2s}}r^{1/2s} \Big(\abs{1+t}^{2s}e^{-\frac{r\,\vphi_s(2,2x/v+(1+2x/v)t)}{\abs{2x+v}^{2s}}}\\
		&\qquad \qquad-\abs{1-t}^{2s}e^{-\frac{r\,\vphi_s(2,2x/v-(1+2x/v)t)}{\abs{2x+v}^{2s}}}\Big) \d r\d t.
	\end{align*}
	We use the change of variables $t+1=y$ respectively $t-1=y$. Thereafter, we write 
	\begin{equation*}
		\Im \int\limits_{0}^\infty e^{ir^{1/2s}}r^{1/2s}\II{}\d r =:\II{1}+\II{2}+\II{3},
	\end{equation*}
	where
	\begin{align*}
		2\II{1}&= \Im \int\limits_{0}^\infty \int\limits_{0}^\infty e^{iyr^{1/2s}}r^{1/2s} y^{2s}\Big(e^{-\frac{r\,\vphi_s(2,(1+2x/v)y-1)}{\abs{2x+v}^{2s}}}e^{-ir^{1/2s}}-e^{-\frac{r\,\vphi_s(2,-(1+2x/v)y-1)}{\abs{2x+v}^{2s}}}e^{ir^{1/2s}}\Big)\d r\d y,\\
		2\II{2}&= -\Im \int\limits_{-1}^0 \int\limits_{0}^\infty e^{i(y+1)r^{1/2s}}r^{1/2s} \abs{y}^{2s}e^{-\frac{r\,\vphi_s(2,-(1+2x/v)y-1)}{\abs{2x+v}^{2s}}} \d r\d y,\\
		2\II{3}&= -\Im \int\limits_{0}^1 \int\limits_{0}^\infty e^{i(y-1)r^{1/2s}}r^{1/2s} \abs{y}^{2s}e^{-\frac{r\,\vphi_s(2,(1+2x/v)y-1)}{\abs{2x+v}^{2s}}}\d r\d y.
	\end{align*}
	By symmetry, we find that $\II{2}+\II{3}=0$. To treat the term $\II{1}$, we would like apply arguments similar to those used in the second step of the proof of the \autoref{lem:representation_k_s_v} and rotate the path of integration with respect to $y$ from $\R_+$ to $\gamma_{\theta}$ for some very small $\theta>0$. Since we would apply Cauchy's integral theorem, we need the corresponding functions to be holomorphic in the respective region of $\C$. But this is problematic as the map $j:z\mapsto \vphi_s(2c_3,c_1+c_2z)$ is not holomorphic in $\R_+ + i\R\subset \C$. We solve this problem as follows. Notice that $j_{c_1,c_2,c_3}$ is a holomorphic function on $\C\setminus((-(c_1+c_3)/c_2+i\R)\cup((c_3-c_1)/2+i\R))$ if written as 
	\begin{align*}
		\frac{(c_1+c_2z+c_3)\big( (c_1+c_2z+c_3)^2 \big)^s-(c_1+c_2z-c_3)\big( (c_1+c_2z-c_3)^2 \big)^s}{(2s+1)2c_3}.
	\end{align*}
	Note that we use the principal value of the complex logarithm to make sense of out this term. Now, we consider the following two integrals with respect to $y$. For the moment, we fix $r>0$ and define
	\begin{align*}
		\II{1,1}:=\int\limits_{0}^\infty e^{iyr^{1/2s}} y^{2s}e^{-\frac{r\,\vphi_s(2,(1+2x/v)y-1)}{\abs{2x+v}^{2s}}}\d y,\quad
		\II{1,2}:=\int\limits_{0}^\infty e^{iyr^{1/2s}} y^{2s}e^{-\frac{r\,\vphi_s(2,(1+2x/v)y+1)}{\abs{2x+v}^{2s}}}\d y.
	\end{align*}
	By the above considerations, we know that $y\mapsto \vphi_s(2,(1+2x/v)y+1)$ is holomorphic in $\R_++i\R\subset \C$. Thus, with arguments very close to those used in the proof of \autoref{lem:representation_k_s_v} step 2, the term $\II{1,2}$ equals
	\begin{align*}
		\int\limits_{\gamma_{\theta}} e^{iyr^{1/2s}} y^{2s}e^{-\frac{r\,\vphi_s(2,(1+2x/v)y+1)}{\abs{2x+v}^{2s}}}\d y
	\end{align*} 
	for a small angle $\theta>0$. We turn our attention to the term $\II{1,1}$. \smallskip
	
	\textbf{Claim A.} The term $\II{1,1}$ equals
	\begin{equation*}
		\Big(\,\int\limits_{0}^{2}+\int\limits_{2+\gamma_{\theta}} e^{iyr^{1/2s}} y^{2s}e^{-\frac{r\,\vphi_s(2,(1+2x/v)y-1)}{\abs{2x+v}^{2s}}}\d y \Big).
	\end{equation*}

	Note that the function $y\mapsto \vphi_s(2,(1+2x/v)y-1)$ is holomorphic in $\C_{>b}:=\{ a_1i+a_2\in \C\mid a_1,a_2\in \R\land a_2>b \}$ with $b= 2/(1+2x/v)$. Thus, it is holomorphic in $C_{>2}$. By Cauchy's integral theorem, the integral over the closed loop $\tilde{\gamma}_{\theta,n,2}:=2+\gamma_{\theta,n}$ is zero. Notice further that 
	\begin{align*}
		\big|\int\limits_{2+\gamma_{\theta,n}^{(2)}} e^{iyr^{1/2s}} y^{2s}e^{-\frac{r\,\vphi_s(2,(1+2x/v)y-1)}{\abs{2x+v}^{2s}}}\d y\big|&\le \frac{\pi \theta }{n}\big|\int\limits_{0}^1 e^{i(2+e^{i\pi\theta t}/n)\,r^{1/2s}} y^{2s}e^{-\frac{r\,\vphi_s(2,(1+2x/v)(2+e^{i\pi \theta t}/n)-1)}{\abs{2x+v}^{2s}}}\d t\big|
	\end{align*}
	converges to zero as $n\to \infty$. We analyze the following term closer:
	\begin{equation*}
		\vphi_s(2,(1+2x/v)(2+e^{i\pi \theta(1-t)}n)-1)= \frac{(1+2x/v)^{1+2s}}{(2s+1)2}\Big( r_{1,n}^{1+2s}e^{i(1+2s)\phi_{1,n}}-r_{2,n}^{1+2s}e^{i(1+2s)\phi_{2,n}} \Big),
	\end{equation*}
	where
	\begin{align*}
		r_{1,n}&:=\sqrt{4+4\cos(\pi \theta(1-t))n+n^2}, &&
		r_{2,n}:=\sqrt{4+4(1-\frac{v}{1+2x})\cos(\pi \theta(1-t))n+n^2},\\
		\phi_{1,n}&:=\arctan\Big(\frac{\sin(\pi \theta(1-t))n}{\cos(\pi \theta(1-t))n+2}\Big), &&
		\phi_{2,n}:=\arctan\Big(\frac{\sin(\pi \theta(1-t))n}{\cos(\pi \theta(1-t))n+2-2v/(1+2x)}\Big).
	\end{align*}
	Here, $\theta$ is sufficiently small such that $\Re((1+2x/v)(2+e^{i\pi \theta(1-t)}n)-2)>0$ for all $t\in[0,1]$. Thereby, there exists a constant $c_1>0$ such that
	\begin{align*}
		\Re\big(\vphi_s(2,(1+2x/v)(2+e^{i\pi \theta(1-t)}n)-1)\big)\ge c_1n^{2s}
	\end{align*}
	for all $t\in [0,1]$ and all $v$.
	Thus, we estimate
	\begin{align*}
		\big|\int\limits_{2+\gamma_{\theta,n}^{(4)}} e^{iyr^{1/2s}} y^{2s}e^{-\frac{r\,\vphi_s(2,(1+2x/v)y-1)}{\abs{2x+v}^{2s}}}\d y\big|\le c_1 n^{1+2s}\int\limits_{0}^1 e^{-\sin(\pi \theta(1-t))nr^{1/2s}}  e^{-\frac{c_1r\,n^{2s}}{\abs{2x+v}^{2s}}}\d y
	\end{align*}
	which converges to zero as $n\to \infty$. Thus, the claim follows. \medskip

	With similar arguments, we rotate the path of integration with respect to $r$ from $\gamma_0$ to $\gamma_{\beta}$ resp.\ $\gamma_{-\beta}$ for an angle $\beta>0$ sufficiently small depending on $\theta$. This yields
	\begin{align*}
		2\II{1}&=\Im \int\limits_{\gamma_{-\beta}}r^{1/2s}e^{-ir^{1/2s}}\II{1,1}\d r-\Im \int\limits_{\gamma_{\beta}}r^{1/2s}e^{ir^{1/2s}}\II{1,2}\d r.
	\end{align*}	
	 Now, note that the exponential $\exp(-\frac{r\,\vphi_s(2,(1+2x/v)y\pm 1)}{\abs{2x+v}^{2s}})$ converges to $1$ as $v\to \infty$. Thus, in the limit $v\to \infty$, the term $\II{1}$ equals
	\begin{align*}
		\Im &\int\limits_{\gamma_{\beta}} \Big(\,\int\limits_{0}^2+\int\limits_{2+\gamma_{\theta}}\Big)  e^{iy\overline{r}^{1/2s}}e^{-2i\pi \beta}\overline{r}^{1/2s} y^{2s}\frac{e^{-i\overline{r}^{1/2s}}}{2}\d y - \int\limits_{\gamma_{\theta}}  e^{iyr^{1/2s}}r^{1/2s} y^{2s}\frac{e^{ir^{1/2s}}}{2}\d y \d r\\
		&=	\Im\int\limits_{\gamma_{\beta}} \int\limits_{\gamma_{\theta_1}}  e^{iy\overline{r}^{1/2s}}e^{-2i\pi \beta}\overline{r}^{1/2s} y^{2s}\frac{e^{-i\overline{r}^{1/2s}}}{2}\d y - \int\limits_{\gamma_{\theta_2}}  e^{iyr^{1/2s}}r^{1/2s} y^{2s}\frac{e^{ir^{1/2s}}}{2}\d y \d r.
	\end{align*}
	In the last equation, we rotated the path of integration to $\gamma_{\theta_1}$ respectively $\gamma_{\theta_2}$ with $\theta_1=1/2+\beta/2s$ and $\theta_2=1/2-\beta/2s$. This is now possible as the term $\vphi_s(2,(1+2x/v)y-1)$, which was not holomorphic in $\R_++i\R$, is gone. Therefore, the term $\lim_{v\to\infty} \II{1}$ equals
	\begin{align*}
		\Re&e^{i\pi s}\int\limits_{0}^\infty \int\limits_{0}^\infty  e^{-yr^{1/2s}}r^{1/2s} y^{2s}\frac{e^{-ie^{-i\pi \beta/2s}r^{1/2s}}-e^{ie^{i\pi\beta/2s}r^{1/2s}}}{2} \d r\\
		&=\Gamma(2s+1)\Re e^{i\pi s}\int\limits_{0}^\infty \frac{e^{-ie^{-i\pi \beta/2s}r^{1/2s}}-e^{ie^{i\pi\beta/2s}r^{1/2s}}}{2r} \d r=\Gamma(2s+1)\Re e^{i\pi s}\int\limits_{0}^\infty \frac{-2i\Im\big(e^{ie^{i\pi\beta/2s}r^{1/2s}}\big)}{2r} \d r\\
		&=\Gamma(2s+1)\sin(\pi s) \Im\int\limits_{\gamma_{\beta}}\frac{e^{ir^{1/2s}}}{r} \d r.
	\end{align*}
	A rotation of the path $\gamma_{\beta}$ to $\gamma_0$ yields
	\begin{align*}
		\lim\limits_{v\to \infty} \II{1}&= \Gamma(2s+1)\sin(\pi s)2s\int\limits_0^\infty \frac{\sin(r)}{r}\d r = \Gamma(2s+1)\sin(\pi s)2s\frac{\pi}{2}.
	\end{align*}	
	
	\textit{Analysis of the term $\III{}$:} For the moment, we assume $2x\ne v$. We scale the variable $t$ by $\abs{2x/v-1}$ and write the expression $\III{}$ as
	\begin{align*}
		\int\limits_{0}^\infty  \frac{\abs{\sgn(\iota)+t}^{2s}e^{-\frac{r}{t^{2s}}\frac{\vphi_s(2, -\abs{2x/v-1}t+2x/v)}{\abs{2x-v}^{2s}}}-\abs{\sgn(\iota)-t}^{2s}e^{-\frac{r}{t^{2s}}\frac{\vphi_s(2, \abs{2x/v-1}t+2x/v)}{\abs{2x-v}^{2s}}}}{2\abs{t}^{1+2s}} \d t.
	\end{align*}
	Now, we follow the same calculations as in the \textit{analysis of the term $\II{}$}. Thereafter, we find
	\begin{align*}
		\III{1}&:= \Im \int\limits_{0}^\infty e^{ir^{1/2s}}r^{1/2s}\III{}\d r= \frac{1}{2}\Im \int\limits_{0}^\infty r^{1/2s} \Big(\III{1,1}e^{-\sgn(\iota)ir^{1/2s}}-\III{1,2}e^{\sgn(\iota)ir^{1/2s}}\Big)\d r,\\
		\intertext{where}
		\III{1,1}&:=\int\limits_{0}^\infty e^{iyr^{1/2s}} y^{2s}e^{-\frac{r\,\vphi_s(2,1-\abs{2x/v-1}y)}{\abs{2x-v}^{2s}}}\d y,\qquad		\III{1,2}:= \int\limits_{0}^\infty e^{iyr^{1/2s}}y^{2s}e^{-\frac{r\,\vphi_s(2,1+\abs{2x/v-1}y)}{\abs{2x-v}^{2s}}}\d y.
	\end{align*}
	One easily sees using Cauchy's integral theorem that $\III{1,2}$ equals
	\begin{align*}
		\int\limits_{\gamma_{\theta}} e^{iyr^{1/2s}}y^{2s}e^{-\frac{r\,\vphi_s(2,1+\abs{2x/v-1}y)}{\abs{2x-v}^{2s}}}\d y.
	\end{align*}
	As in the analysis of $\II{}$, the term $\III{1,1}$ is more involved. Note that $y\mapsto \vphi_s(2,1-\abs{2x/v-1}y)$ is holomorphic in $\C_{+,b}:=\{ z\in \C\mid  \Re z \in \R_+\setminus\{ b\}\}$, where $b:=2/\abs{2x/v-1}$. In contrast to the analysis of $\II{}$, we stress that $b\to \infty$ as $v\to \infty$. We define the following path for $a_1,a_2\in \R$
	\begin{align*}
		\gamma_{a_1,a_2}^{(5)}:(0,1)\to \C,\qquad \gamma_{a_1,a_2}^{(5)}(t):=a_1+i\tan(\pi \theta)a_2t.
	\end{align*}
	Now, we split the term $\III{1,1}$ into two parts $\III{1,1}=\IV{1}+\IV{2}$, where
	\begin{align*}
		\IV{1}&:= \int\limits_{0}^{b} e^{iyr^{1/2s}} y^{2s}e^{-\frac{r\,\vphi_s(2,\abs{2x/v-1}y-1)}{\abs{2x-v}^{2s}}}\d y,\quad		\IV{2}:=\int\limits_{b}^{\infty} e^{iyr^{1/2s}} y^{2s}e^{-\frac{r\,\vphi_s(2,\abs{2x/v-1}y-1)}{\abs{2x-v}^{2s}}}\d y.
	\end{align*}
	Since the paths $\gamma_{\theta,b,-}:=\gamma_{\theta}|_{(0,\arctan(\pi \theta)b)}$, $\gamma_{b,b}^{(5)}$, and the straight path from $0$ to $b$ form a closed loop, we find using appropriate approximations that $\IV{1}$ equals
	\begin{align*}
		\IV{1,1}-\IV{1,2}:= \int\limits_{\gamma_{\theta,b,-}} e^{iyr^{1/2s}} y^{2s}e^{-\frac{r\,\vphi_s(2,\abs{2x/v-1}y-1)}{\abs{2x-v}^{2s}}}\d y-\lim\limits_{\eps\to 0+}\int\limits_{\gamma_{b-\eps,b}^{(5)}} e^{iyr^{1/2s}} y^{2s}e^{-\frac{r\,\vphi_s(2,\abs{2x/v-1}y-1)}{\abs{2x-v}^{2s}}}\d y,
	\end{align*}
	With similar arguments, the term $\IV{2}$ equals
	\begin{align*}
		\IV{2,1}+\IV{2,2}:=\int\limits_{\gamma_{\theta,b,+}} e^{iyr^{1/2s}} y^{2s}e^{-\frac{r\,\vphi_s(2,\abs{2x/v-1}y-1)}{\abs{2x-v}^{2s}}}\d y+\lim\limits_{\eps\to 0+}\int\limits_{\gamma_{b+\eps,b}^{(5)}} e^{iyr^{1/2s}} y^{2s}e^{-\frac{r\,\vphi_s(2,\abs{2x/v-1}y-1)}{\abs{2x-v}^{2s}}}\d y
	\end{align*}
	where $\gamma_{\theta,b,+}:=\gamma_{\theta}|_{(\arctan(\pi \theta)b,\infty)}$. Note that the term $\IV{1,2}$ does not equal $\IV{2,2}$! Rather, we calculate
	\begin{align*}
		\IV{1,2}&=\int\limits_{0}^1 \frac{i2^{1+2s}\tan(\pi \theta)}{\abs{2x/v-1}^{1+2s}}e^{i2\frac{1+i\tan(\pi \theta)t}{\abs{2x/v-1}}r^{1/2s}} \big(1+i\tan(\pi\theta)t\big)^{2s}e^{-\frac{r\,\big((1+i\tan(\pi \theta)t)^{1+2s}-ie^{-i\pi s}(\tan(\pi \theta)t)^{1+2s}\big)}{2^{-2s}(2s+1)\abs{2x-v}^{2s}}}\d y\\
		\intertext{and}
		\IV{2,2}&=\int\limits_{0}^1 \frac{i2^{1+2s}\tan(\pi \theta)}{\abs{2x/v-1}^{1+2s}}e^{i2\frac{1+i\tan(\pi \theta)t}{\abs{2x/v-1}}r^{1/2s}} \big(1+i\tan(\pi\theta)t\big)^{2s}e^{-\frac{r\,\big((1+i\tan(\pi \theta)t)^{1+2s}-ie^{i\pi s}(\tan(\pi \theta)t)^{1+2s}\big)}{2^{-2s}(2s+1)\abs{2x-v}^{2s}}}\d y.
	\end{align*}
	If $\kappa < 1/2$, then both exponentials converge to $1$ and, thus, we readily see that these two terms equal each other in the limit $v\to \infty$. In the case $\kappa=1/2$, we integrate $\IV{1,2}$ with $r^{1/2s}e^{-\sgn(\iota)ir^{1/2s}}$ over $\R$, rotate the path of integration to $\gamma_{\beta}$ with a very small angle $\beta>0$, and scale $r$ by $\abs{2x/v-1}^{2s}$. This yields
	\begin{align*}
		\Im\int\limits_{0}^\infty r^{1/2s}&e^{-\sgn(\iota)i r^{1/2s}}\IV{1,2}\d r= \Re2^{1+2s}\tan(\pi \theta)\int\limits_{\gamma_{\beta}} r^{1/2s}e^{-\sgn(\iota)i\abs{2x/v-1} r^{1/2s}}\\
		&\times \int\limits_{0}^1 e^{i2(1+i\tan(\pi \theta)t)r^{1/2s}} \big(1+i\tan(\pi\theta)t\big)^{2s}e^{-r\frac{(1+i\tan(\pi \theta)t)^{1+2s}-ie^{-i\pi s}(\tan(\pi \theta)t)^{1+2s}}{2^{-2s}(2s+1)v^{2s}}}\d y \d r
	\end{align*}
	which converges to 
	\begin{align*}
		\Re2^{1+2s}\tan(\pi \theta)\int\limits_{\gamma_{\beta}} r^{1/2s}\int\limits_{0}^1 e^{i2(1+i\tan(\pi \theta)t)r^{1/2s}} \big(1+i\tan(\pi\theta)t\big)^{2s}\d y\d r
	\end{align*}
	as $v\to \infty$. With similar means one recognizes that 
	\begin{equation*}
		\Im\int\limits_{0}^\infty r^{1/2s}e^{-\sgn(\iota)i r^{1/2s}}\IV{2,2}\d r
	\end{equation*}
	converges to the same term as $v\to \infty$. Thus, the term $-\IV{1,2}+\IV{2,2}$ yields no contribution in the limit $v\to \infty$ in all cases. \smallskip
	
	Since 
	\begin{equation*}
		\IV{1,1}+\IV{2,1}=\int\limits_{\gamma_{\theta}} e^{iyr^{1/2s}} y^{2s}e^{-\frac{r\,\vphi_s(2,\abs{2x/v-1}y-1)}{\abs{2x-v}^{2s}}}\d y
	\end{equation*}
	and the exponential converges to $\exp(-r\vphi_s(2,1)/\iota^{2s})$ in the limit $v\to \infty$, we can apply similar arguments as in the analysis of the term $\II{}$ to find that
	\begin{align*}
		\Im \int\limits_{0}^\infty r^{1/2s}\big(e^{-\sgn(\iota)i r^{1/2s}}\big(\IV{1,1}+\IV{2,1}\big)-e^{\sgn(\iota)ir^{1/2s}}\III{1,2}\big)\d r
	\end{align*}
	converges to 
	\begin{align*}
		\frac{\sgn(\iota)}{2}\Im \int\limits_{\gamma_\beta}\int\limits_{\gamma_{\theta_1}} e^{iy\overline{r}^{1/2s}}e^{-2i\pi\beta}\overline{r}^{1/2s}y^{2s}e^{-i\overline{r}^{1/2s}}e^{-\overline{r}\frac{\vphi_s(2,1)}{\abs{2\iota}^{2s}}}\d y - \int\limits_{\gamma_{\theta_2}} e^{iyr^{1/2s}}r^{1/2s}y^{2s}e^{ir^{1/2s}}e^{-r\frac{\vphi_s(2,1)}{\abs{2\iota}^{2s}}}\d y \d r,
	\end{align*}
	where again $\theta_1=1/2+\beta/2s$ and $\theta_2=1/2-\beta/2s$. If $\iota =\infty$, then this term is exactly the one derived in the analysis of term $\II{}$. It remains to consider the case $\iota<\infty$. In this case, we can repeat the above calculations with $\beta=0$. This yields
	\begin{align*}
		\lim\limits_{v\to \infty} \III{1}&= \frac{\sgn(\iota)}{2}\Re e^{i\pi s} \int\limits_{0}^\infty \int\limits_{0}^\infty  e^{-yr^{1/2s}}r^{1/2s}y^{2s}e^{-ir^{1/2s}}e^{-r\frac{\vphi_s(2,1)}{\abs{2\iota}^{2s}}}\d y\\
		& \qquad- \int\limits_{0}^\infty e^{-yr^{1/2s}}r^{1/2s}y^{2s}e^{ir^{1/2s}}e^{-r\frac{\vphi_s(2,1)}{\abs{2\iota}^{2s}}}\d y \d r\\
		&= \sgn(\iota)\Gamma(2s+1)\Re e^{i\pi s} \int\limits_{0}^\infty  e^{-r\frac{\vphi_s(2,1)}{\abs{2\iota}^{2s}}}\frac{e^{-ir^{1/2s}}-e^{ir^{1/2s}}}{2r} \d r\\
		&=\sgn(2\iota)\Gamma(2s+1)\sin(\pi s)\int\limits_{0}^\infty  e^{-r\frac{\vphi_s(2,1)}{\abs{2\iota}^{2s}}}\frac{\sin(r^{1/2s})}{r} \d r\\
		&=2s\sgn(\iota)\Gamma(2s+1)\sin(\pi s)\Im\int\limits_{0}^\infty  \sin\big(\frac{r\abs{2\iota}}{\vphi_s(2,1)^{1/2s}}\big)e^{-r^{2s}}\frac{1}{r} \d r.
	\end{align*}
	Now, we define for $b\ge 0$
	\begin{align}
		f(b):= \int\limits_{0}^\infty \frac{ e^{-z^{2s}}\sin(bz)}{z}\d z.
	\end{align}
	Note that $f(0)=0$ and 
	\begin{align*}
		f'(b)= \int\limits_{0}^\infty  e^{-z^{2s}}\cos(bz)\d z=\pi\sF^{-1}[e^{-\abs{\cdot}^{2s}}](b)= \pi q_s(b).
	\end{align*}	
	This proves the result in the case $\iota>-\infty$ and $2x\ne v$. If $\iota =0$ and $2x=v$, then it is immediate to see that $\lim_{v\to \infty} \III{1}=0$. This proves the result in the case $\iota>-\infty$. \smallskip
	
	\textbf{Case ($\iota=-\infty$):} If $\iota=-\infty$, then $\kappa=1/2$. Firstly, we simplify the limit using
	\begin{align*}
		\frac{j_s(x,v)}{v^{2+2s}(2x-v)^{2s}}\to 1+(1/2)^{2+2s}.
	\end{align*}
	We use the representation of $v^{2+2s}\pi^2 k_s(x,v)$ derived in the case $\iota>-\infty$. This yields
	\begin{equation}\label{eq:proof_fasymp_iota_minus_infinity}
		\begin{split}
		2s &v^{2+2s}(2x-v)^{2s}\pi^2 k_s(x,v)= (2x-v)^{2s}\Big(\II{1}+ \III{1}\Big)=\V{1}+\V{2} \\
		&:=\Im \int\limits_{0}^\infty \int\limits_{0}^\infty e^{iyr^{1/2s}}r^{1/2s} y^{2s}e^{-ir^{1/2s}}\Big(\frac{e^{-\frac{r\,\vphi_s(2,(1+2x/v)y-1)}{\abs{2x+v}^{2s}}}-1}{2(2x-v)^{-2s}}+\frac{1-e^{-\frac{r\,\vphi_s(2,1+\abs{2x/v-1}y)}{\abs{2x-v}^{2s}}}}{2(2x-v)^{-2s}}\Big)\d r\d y\\
		&\qquad- \Im \int\limits_{0}^\infty \int\limits_{0}^\infty e^{iyr^{1/2s}}r^{1/2s} y^{2s} e^{ir^{1/2s}}\Big(\frac{e^{-\frac{r\,\vphi_s(2,-(1+2x/v)y-1)}{\abs{2x+v}^{2s}}}-1}{2(2x-v)^{-2s}}+ \frac{1-e^{-\frac{r\,\vphi_s(2,1-\abs{2x/v-1}y)}{\abs{2x-v}^{2s}}}}{2(2x-v)^{-2s}}\Big)\d r\d y.
	\end{split}
	\end{equation}
	We apply the same arguments as those in the proof in the case $\iota>-\infty$ and rotate the path of integration with respect to $r$ and $y$ to legitimize taking the limit $v\to \infty$. Since $(2x-v)/(2x+v)$ converges to zero as $v \to \infty$, the quotients in the previous terms $\VI{1}$, $\VI{2}$ become
	\begin{equation*}
		\frac{e^{-\frac{r\,\vphi_s(2,\pm(1+2x/v)y-1)}{\abs{2x+v}^{2s}}}-1}{2(2x-v)^{-2s}}\to 0\quad \text{and}\quad
		\frac{1-e^{-\frac{r\,\vphi_s(2,1\pm\abs{2x/v-1}y)}{\abs{2x-v}^{2s}}}}{2(2x-v)^{-2s}}\to \frac{r\,\vphi_s(2,1)}{2} \text{ as }v\to \infty.
	\end{equation*}
	Therefore, the term $\VI{1}$ in the limit $v\to \infty$ equals
	\begin{align*}
		\frac{\vphi_s(2,1)}{2}\Im \int\limits_{\gamma_{\theta_1}} \int\limits_{\gamma_{-\beta}} e^{iyr^{1/2s}}r^{1/2s+1} y^{2s} e^{-ir^{1/2s}} \d r\d y.
	\end{align*}
	After a minor calculation, this term equals
	\begin{align*}
		\frac{\vphi_s(2,1)}{2}&\Im \int\limits_{0}^\infty e^{i\pi(1/2+\beta/2s)(1+2s)} \int\limits_{\gamma_{-\beta}} e^{-y}\frac{r^{1/2s+1}}{\abs{r}^{1/2s+1}} y^{2s} e^{-ir^{1/2s}} \d r\d y\\
		&=\frac{\vphi_s(2,1)}{2}\Gamma(2s+1)\Im e^{i\pi(1/2+\beta/2s)(1+2s)} \int\limits_{\gamma_{-\beta}} \frac{r^{1/2s+1}}{\abs{r}^{1/2s+1}} e^{-ir^{1/2s}} \d r\\
		&=\frac{\vphi_s(2,1)}{2}\Gamma(2s+1)\Re e^{i\pi s}  \int\limits_{0}^\infty e^{-i\pi\beta} e^{-i e^{-i\pi \beta/2s}r^{1/2s}} \d r.
	\end{align*}
	Similarly, the term $\VI{2}$ in the same limit equals
	\begin{align*}
		-\frac{\vphi_s(2,1)}{2}\Im \int\limits_{\gamma_{\theta_2}} \int\limits_{\gamma_{\beta}} e^{iyr^{1/2s}}r^{1/2s+1} y^{2s} e^{ir^{1/2s}} \d r\d y
		= -\frac{\vphi_s(2,1)}{2}\Gamma(2s+1)\Re e^{i\pi s} \int\limits_{0}^\infty e^{i\pi\beta} e^{ie^{i\pi \beta/2s}r^{1/2s}} \d r.
	\end{align*}
	Thereby, the sum of both limits equals
	\begin{align*}
		\lim\limits_{v\to \infty} \VI{1}+\VI{2}&= -\frac{\vphi_s(2,1)}{2}\Gamma(2s+1)\Re e^{i\pi s}2i  \Im\int\limits_{0}^\infty e^{i\pi\beta} e^{ie^{i\pi \beta/2s}r^{1/2s}} \d r\\
		&= \vphi_s(2,1)\Gamma(2s+1)\sin(\pi s)  \Im\int\limits_{\gamma_{\beta}}e^{ir^{1/2s}} \d r.
	\end{align*}
	After a rotation of the path of integration to $\gamma_{s}$, this term equals
	\begin{align*}
		\vphi_s(2,1)\Gamma(2s+1)\sin(\pi s)  \Im\int\limits_{\gamma_{s}}e^{ir^{1/2s}} \d r&=\vphi_s(2,1)\Gamma(2s+1)\sin(\pi s)  \Im\int\limits_{0}^\infty e^{i\pi s}e^{-r^{1/2s}} \d r\\
		&=\frac{2^{2s}}{2s+1}\Big(\Gamma(2s+1)\sin(\pi s)\Big)^2.
	\end{align*}
\end{proof}

\subsection{The spatial dominant regime}\label{sec:analysis_at_infinity_space_dom}
In this section, we complete the analysis of $j_s(x,v)k_s(x,v)$ for $\abs{(x,v)}$ near infinity. We consider the case when the spatial variable dominates the velocity variable. This case turns out to be more delicate than the one considered in \autoref{sec:analysis_at_infinity_velo_dom}. Notice that the kernel $k_s$ decays much faster in this regime, see \eqref{eq:def_j_s} or \eqref{th:bounds_time_1}.

\begin{lemma}\label{lem:cancelation_more_1}
	Let $1/4<s<1$ and $a\in \R$, then 
	\begin{align*}
		a(x):=\Im \int\limits_{\R}\int\limits_{\gamma_{\theta}} e^{2ir^{1/2s}}r^{1/2s+1}\frac{1}{\abs{t}^{4s}}e^{-\frac{r}{(x\abs{t})^{2s}}\vphi_s(2t+a,1)}\d r\d t
	\end{align*}
	converges to zero as $x\to \infty$. 
\end{lemma}
\begin{proof}
	We scale $t$ by $1/x$. This yields
	\begin{align*}
		a(x)= x^{4s-3}\Im \int\limits_{\R}\int\limits_{\gamma_{\theta}} e^{2ir^{1/2s}}r^{1/2s+1}\frac{x^{2}}{\abs{t}^{4s}}e^{-\frac{r}{\abs{t}^{2s}}\vphi_s(2t/x+\frac{v}{x},1)}\d r \d t.
	\end{align*}
	In addition, by \autoref{lem:additional_cancelation}, we note that 
	\begin{align*}
		\Im \int\limits_{\R}\int\limits_{\gamma_{\theta}} e^{2ir^{1/2s}}r^{1/2s+1}\frac{1}{\abs{t}^{4s}}e^{-\frac{r}{\abs{t}^{2s}}\vphi_s(a,1)}\d r \d t=0.
	\end{align*}
	Here, we use $k=2>1/(2s)$, i.e.\ $s>1/4$, is essential. Thus, we find that $a(x)$ equals
	\begin{align*}
		x^{4s-3}\Im \int\limits_{\R}\int\limits_{\gamma_{\theta}} e^{2ir^{1/2s}}r^{1/2s+1}4\abs{t}^{2-4s}\frac{e^{-\frac{r}{\abs{t}^{2s}}\vphi_s(2t/x+\frac{v}{x},1)}+e^{-\frac{r}{\abs{t}^{2s}}\vphi_s(-2t/x+\frac{v}{x},1)}-2e^{-\frac{r}{\abs{t}^{2s}}\vphi_s(\frac{v}{x},1)}}{(2t/x)^{2}}\d r \d t.
	\end{align*}
	Note that the quotient in this integral converges to 
	\begin{align*}
		-\frac{r}{\abs{t}^{2s}}[\partial_{\xi}^2\vphi_s](a,1)e^{-\frac{r}{\abs{t}^{2s}}\vphi_s(a,1)}+\frac{r^2}{\abs{t}^{4s}}[\partial_\xi \vphi_s](a,1)^2e^{-\frac{r}{\abs{t}^{2s}}\vphi_s(a,1)} \text{ as }x\to \infty.
	\end{align*}
	By the additional factor $x^{4s-3}$, this proves the result in the case $s<3/4$. From now on, we assume $s\ge 3/4$. 
	Now, we calculate
	\begin{align*}
		&\Im \int\limits_{\R}\int\limits_{\gamma_{\theta}} e^{2ir^{1/2s}}r^{1/2s+2}\abs{t}^{2-6s} e^{-\frac{cr}{\abs{t}^{2s}}} \d r \d t= \Im \int\limits_{\R}\int\limits_{\gamma_{\theta}} e^{2i\abs{t}r^{1/2s}}r^{1/2s+2}\abs{t}^{3} e^{-cr} \d r \d t\\
		&\quad =2\Im \int\limits_{\gamma_{\theta}} \frac{6}{(2r^{1/2s})^4}r^{1/2s+2} e^{-cr} \d r.
	\end{align*}
	The previous term is zero by the same calculations used in the proof of \autoref{lem:additional_cancelation}. Furthermore, we calculate
	\begin{align*}
		&\Im \int\limits_{\R}\int\limits_{\gamma_{\theta}} e^{2ir^{1/2s}}r^{1/2s+3}\abs{t}^{2-8s} e^{-\frac{cr}{\abs{t}^{2s}}} \d r \d t= 2\Im \int\limits_{\gamma_{\theta}} \frac{6}{(2r^{1/2s})^4}r^{1/2s+3} e^{-cr} \d r.
	\end{align*}
	Again by the same calculations as in the proof of \autoref{lem:additional_cancelation}, this is zero. This proves the result for $s=3/4$. To prove the result in the case $s>3/4$, we write $a(x)$ as
	\begin{align*}
		&x^{4s-4}\Im \int\limits_{\R}\int\limits_{\gamma_{\theta}} e^{2ir^{1/2s}}r^{1/2s+1}8\abs{t}^{3-4s}\Big[e^{-\frac{r}{\abs{t}^{2s}}\vphi_s(2t/x+\frac{v}{x},1)}+e^{-\frac{r}{\abs{t}^{2s}}\vphi_s(-2t/x+\frac{v}{x},1)}-2e^{-\frac{r}{\abs{t}^{2s}}\vphi_s(\frac{v}{x},1)}\\
		&-\Big(-\frac{r}{\abs{t}^{2s}}[\partial_{\xi}^2\vphi_s](a,1)e^{-\frac{r}{\abs{t}^{2s}}\vphi_s(a,1)}+\frac{r^2}{\abs{t}^{4s}}[\partial_\xi \vphi_s](a,1)^2e^{-\frac{r}{\abs{t}^{2s}}\vphi_s(a,1)}\Big)\big(\frac{2t}{x}\big)^2\Big]\frac{1}{(2t/x)^{3}}\d r \d t.
	\end{align*}
	The term in the squared brackets converges to zero of order $\cO(\abs{t/x}^4)$. Thus, the whole term converges to zero. 
\end{proof}
\begin{lemma}\label{lem:cancelation_more_2}
	If $3/4<s<1$ and $a\in \R$, then 
	\begin{align*}
		b(x):=\Im \int\limits_{\R}\int\limits_{\gamma_{\theta}} e^{2ir^{1/2s}}r^{1/2s+1}\frac{1}{\abs{t}^{4s-2}}e^{-\frac{r}{(x\abs{t})^{2s}}\vphi_s(2t+a,1)}\d r \d t
	\end{align*}
	converges to zero as $x\to \infty$. 
\end{lemma}
\begin{proof}
	We proceed just as in \autoref{lem:cancelation_more_1}. We scale $t$ by $1/x$ and find
	\begin{align*}
		b(x)=\Im \int\limits_{\R}\int\limits_{\gamma_{\theta}} e^{2ir^{1/2s}}r^{1/2s+1}\frac{x^{4s-3}}{\abs{t}^{4s-2}}e^{-\frac{r}{\abs{t}^{2s}}\vphi_s(2t/x+a,1)}\d r\d t.
	\end{align*}
	Furthermore, the term
	\begin{align*}
		\Im \int\limits_{\R}\int\limits_{\gamma_{\theta}} e^{2ir^{1/2s}}r^{1/2s+1}\frac{1}{\abs{t}^{4s-2}}e^{-\frac{r}{\abs{t}^{2s}}\vphi_s(a,1)}\d r\d t= \Im \int\limits_{\R}\int\limits_{\gamma_{\theta}} e^{2i\abs{t}r^{1/2s}}r^{1/2s+1}\abs{t}^{3}e^{-r\vphi_s(a,1)}\d r\d t
	\end{align*}
	equals zero as we have shown in the proof of \autoref{lem:cancelation_more_1}. Thus, we may write $b(x)$ as 
	\begin{align*}
		x^{4s-4}\Im \int\limits_{\R}\int\limits_{\gamma_{\theta}} e^{2ir^{1/2s}}r^{1/2s+1}\frac{2}{\abs{t}^{4s-3}}\frac{e^{-\frac{r}{\abs{t}^{2s}}\vphi_s(2t/x+a,1)}+e^{-\frac{r}{\abs{t}^{2s}}\vphi_s(-2t/x+a,1)}-2e^{-\frac{r}{\abs{t}^{2s}}\vphi_s(a,1)}}{(2\abs{t}/x)}\d r\d t.
	\end{align*}
	The quotient in the previous integral converges to zero and, thus, the whole term does. 
\end{proof}

The next proposition is the key result in this subsection. It is the counterpart to \autoref{prop:fasymp_d1_s_x/v_le12}. The result in \autoref{prop:fasymp_d1_s_x_s_ge_12} in the special case $v(t)=0$ is already contained in \cite{ScMa10}, see equations (34) and (37) therein. 

\begin{proposition}\label{prop:fasymp_d1_s_x_s_ge_12}
	Let $s\in (0,1)$, $\kappa \in [0,2)$. We fix an arbitrary family of vectors $\{\big(x(t),v(t)\big)\mid  t\in \R_+ \}\subset \R^{2}$. If $\abs{x(t)}\to \infty$ as $t\to \infty$ and 
	\begin{equation}\label{eq:d_1_s_limit_x_s_ge_12}
		\lim\limits_{t\to \infty}\frac{\abs{v(t)}}{\abs{x(t)}}= \kappa,
	\end{equation}
	then
	\begin{align*}
		\lim\limits_{t\to\infty} j_s(x(t),v(t)) k_s(x(t),v(t))=  C_{s,3}(\kappa)
	\end{align*}
	as $t\to \infty$. Here, the constant $C_{s,3}(\kappa)$ is given by 
	\begin{equation*}
			\frac{1}{4\pi}(1+\kappa^{2+2s})\abs{2-\kappa}^{2s}(-\Delta_\kappa)^{\alpha}\vphi_s(\kappa,1)^2,
	\end{equation*}
	where $\alpha:= \frac{1+4s}{2}$ and
	\begin{align*}
		(-\Delta)^{\alpha}u(r)&= c_{\alpha} \pv \int\limits_{\R}\frac{Q_{\alpha}[u](r,t)}{\abs{t}^{1+2\alpha}}\d t,\\
		Q_{\alpha}[u](r,t)&:=\begin{cases}
			u(r)-u(r+t)&, \text{ if }\alpha\in (0,1),\\
			u(r)+\partial_r^2u(r)\frac{t^2}{2}-u(r+t)&,\text{ if } \alpha\in (1,2),\\
			u(r)+\partial_r^2u(r)\frac{t^2}{2}+\partial_r^4u(r)\frac{t^4}{4!}-u(r+t)&,\text{ if } \alpha\in (2,3),
		\end{cases}
	\end{align*}
	for $s\notin\{1/4,3/4\}$.
\end{proposition}
\begin{remark}
	In the cases $s=1/4$ and $s=3/4$, the constant $C_{s,3}(\kappa)$ involves second and fourth derivative in place of the fractional Laplacian $(-\Delta_\kappa)^\alpha\vphi_s(\kappa,1)^2$, i.e.\
	\begin{align*}
		(-\Delta_\kappa)^\alpha\vphi_s(\kappa,1)^2&= -\partial_\kappa^2\vphi_{1/4}(\kappa,1)^2 &&\text{ if }s=1/4, \\
		(-\Delta_\kappa)^\alpha\vphi_s(\kappa,1)^2&= \partial_\kappa^4\vphi_{3/4}(\kappa,1)^2 &&\text{ if }s=3/4.
	\end{align*}
\end{remark}
\begin{remark}
	Note that using the properties of the $\Gamma$-function, we find that
	\begin{align*}
		c_{\alpha}=\frac{2\alpha\sin(\pi \alpha)\Gamma(2\alpha)}{\pi }.
	\end{align*}
\end{remark}
\begin{proof}
	Firstly, we assume without loss of generality that $v=v(t)\ge 0$, $x=x(t)>0$, and $v/x<2$. We simply write $x(t)=x\to \infty$ for the limit $t\to \infty$. Since in the limit $t\to \infty$
	\begin{align*}
		\frac{1}{x^{2+4s}}\Big(1+x^{2+2s}+v^{2+2s}\Big)(1+|2x-v|)^{2s}\to (1+\kappa^{2+2s})\abs{2-\kappa}^{2s},
	\end{align*}
	it remains to consider the limiting behavior of $\pi^2 x^{2+4s}k_s(x,v)$. By \autoref{lem:representation_k_s_x}, we may write
	\begin{align*}
		\pi^2 x^{2+4s}k_s(x,v)&=:\frac{1}{2s}\Im\int\limits_{\gamma_{\theta}} e^{i2r^{1/2s}}r^{1/2s} f(r)\d r,
	\end{align*}
	where
	\begin{align*}
		f(r)= x^{2s}\int\limits_{\R} [\partial_\xi \vphi_s](2+\frac{v}{x}t,t)e^{-\frac{r}{x^{2s}}\vphi_s(\frac{v}{x}t+2,t)}\d t.
	\end{align*}
	Here, we used the change of variables $t\mapsto 1/t$ and the symmetry of $\vphi_s$ in the second variable, see \eqref{eq:k_symmetry}. Let's define for any $a\in \R$ the function
	\begin{align*}
		h_{a}(t):= \int\limits_{0}^t[\partial_{\xi}\vphi_s](2+ay,y)\d y.
	\end{align*}
	Now, integration by parts yields
	\begin{equation}\label{eq:x_limit_f_r}
		\begin{split}
			f(r)&= r\int\limits_{\R} h_{v/x}(t)\Big(\partial_t\vphi_s(2+\frac{v}{x}t,t) \Big)e^{-\frac{r}{x^{2s}}\vphi_{s}(2+\frac{v}{x}t,t) }\d t\\
			&= - r\int\limits_{\R} \Big([\partial_{\xi}\vphi_s](2+\frac{v}{x}t,t)\vphi_s(2+\frac{v}{x}t,t) \Big)e^{-\frac{r}{x^{2s}}\vphi_{s}(2+\frac{v}{x}t,t) }\d t\\
			&- r\int\limits_{\R} h_{v/x}(t)\vphi_s(2+\frac{v}{x}t,t)\Big(\partial_t e^{-\frac{r}{x^{2s}}\vphi_{s}(2+\frac{v}{x}t,t) }\Big)\d t=:\I{}+\II{}.
		\end{split}
	\end{equation}
	With the change of variables $t\mapsto 1/t$, the term $\I{}$ equals
	\begin{align*}
		- r\int\limits_{\R}\frac{1}{t^2} \Big(\frac{t}{\abs{t}^{2s}}[\partial_{\xi}\vphi_s](2t+\frac{v}{x},1)\frac{1}{\abs{t}^{2s}}\vphi_s(2t+\frac{v}{x},1) \Big)e^{-\frac{r}{(x\abs{t})^{2s}}\vphi_{s}(2t+\frac{v}{x},1) }\d t\\
		=  r\int\limits_{\R}\frac{\partial_t \big(\vphi_s(\frac{v}{x},1)^2-\vphi_s(2t+\frac{v}{x},1)^2\big)}{4t \, \abs{t}^{4s}} e^{-\frac{r}{(x\abs{t})^{2s}}\vphi_{s}(2t+\frac{v}{x},1) }\d t.
	\end{align*}
	Now, we integrate by parts again:
	\begin{align*}
		\I{}&= (1+4s) r\int\limits_{\R}\frac{ \vphi_s(\frac{v}{x},1)^2-\vphi_s(2t+\frac{v}{x},1)^2}{4\abs{t}^{4s+2}} e^{-\frac{r}{(x\abs{t})^{2s}}\vphi_{s}(2t+\frac{v}{x},1) }\d t\\
		&\quad - r\int\limits_{\R}\frac{ \vphi_s(\frac{v}{x},1)^2-\vphi_s(2t+\frac{v}{x},1)^2}{4t \, \abs{t}^{4s}} \Big(\partial_te^{-\frac{r}{(x\abs{t})^{2s}}\vphi_{s}(2t+\frac{v}{x},1) }\Big)\d t\\
		&=:\III{}+\IV{}.
	\end{align*}
	In the following, we will prove that both $\II{}$ and $\IV{}$ yield no contribution in the limit $x\to \infty$, i.e.\ their imaginary parts converge to zero when integrated with $r^{1/2s}e^{ir^{1/2s}}$ over $\gamma_{\theta}$. Thus, the only contributing term turns out to be $\III{}$. We start by analyzing $\III{}$ in the limit $x\to \infty$. For this, we need to distinguish a few cases depending on the size of $s$. 
	
	\textbf{Case $s<1/4$:} In this case, the singularity of the integrand is sufficiently small such that we can simply pull the limit into the integral in the principle value sense. Note that the exponential in the term $\III{}$ converges to $1$ as $x\to \infty$. Thus, we find that
	\begin{equation*}
		\lim\limits_{x\to \infty}\III{}=r\frac{1+4s}{4}\pv\int\limits_{\R}\frac{ \vphi_s(\kappa,1)^2-\vphi_s(2t+\kappa,1)^2}{\abs{t}^{4s+2}}\d t.
	\end{equation*}
	Finally, the \autoref{lem:r_integral_gamma_function} yields the desired claim in the regime $s<1/4$.  \medskip
	
	\textbf{Case $s=1/4$:} We define the auxiliary function
	\begin{equation}
		\tilde{h}_{1,x}(t):=2\Im \int\limits_{\gamma_{\theta}} e^{i2r^2}\frac{r^3 }{\abs{t}}e^{-\frac{r}{\sqrt{x\abs{t}}}\vphi_{1/4 }(2t+\frac{v}{x},1) }\d r\1_{(-1,1)}(t).
	\end{equation}
	We will show that this function is an approximate identity at $t= 0$. \smallskip
	
	\textbf{Claim A:} $\tilde{h}_{1,x}\rightharpoonup \pi/4^2 \big(\delta_{0-}+\delta_{0+}\big)$.\smallskip
	
	Since $\Im \int_{\gamma_{\theta}} e^{ir^2}r^3\d r =0$ by \autoref{lem:r_integral_gamma_function}, we know that for any $\eps>0$
	\begin{equation*}
		\int\limits_{(-\eps,\eps)^c} \tilde{h}_{1,x}(t)\d t \to 0 \text{ as }x\to \infty.
	\end{equation*}
	Furthermore, the integral $\int\limits_{0}^\infty \tilde{h}_{1,x}(t) \d t$ in the limit $x\to \infty$ equals
	\begin{align*}
		\lim\limits_{x\to \infty} 2	\int\limits_{0}^x\Im \int\limits_{\gamma_{\theta}} e^{i2r^2}\frac{r^3 }{\abs{t}}e^{-\frac{r}{\sqrt{\abs{t}}}\vphi_{1/4 }(2t/x+\frac{v}{x},1) }\d r\d t= 2	\int\limits_{0}^\infty\Im \int\limits_{\gamma_{\theta}} e^{i2r^2}\frac{r^3 }{t}e^{-\frac{r}{\sqrt{t}}\vphi_{1/4}(\kappa,1) }\d r\d t.
	\end{align*}
	Here, we scaled $t$ by $1/x$, used that the integrand is even. This equals $\pi/4^2$ by \autoref{lem:additional_cancelation}. By symmetry, the same is true for the integral over $(-\infty,0)$ which proves the claim.\smallskip
	
	Using this claim and \autoref{lem:r_integral_gamma_function}, we conclude that $2\Im \int_{\gamma_{\theta}} e^{ir^{2}}r^{2}\III{}\d r$ in the limit $x\to \infty$ converges to
	\begin{align*}
		\frac{1+4s}{4}\frac{\pi}{4^2}\lim\limits_{t\to 0} \frac{ 2\vphi_{1/4}(\kappa,1)^2-\vphi_{1/4}(2t+\kappa,1)^2-\vphi_{1/4}(-2t+\kappa,1)^2}{\abs{t}^{2}}= -\frac{\pi}{8}\partial_\kappa^2\vphi_{1/4}(\kappa,1)^2.
	\end{align*} 

	\textbf{Case $1/4<s<3/4$:} The \autoref{lem:cancelation_more_1} allows us to add a term to $\III{}$ which does not change the result in the limit $x\to \infty$. Thus, we receive
	\begin{align*}
		&\lim\limits_{x\to \infty}\frac{1}{2s}\Im\int\limits_{\gamma_{\theta}} e^{2ir^{1/2s}} r^{1/2s}\int\limits_{\R} \III{}\d t \d r\\
		&\quad =\frac{1+4s}{8s}\lim\limits_{x\to \infty}\Im\int\limits_{\gamma_{\theta}} e^{2ir^{1/2s}} r^{1/2s+1}\int\limits_{\R}  \frac{ \vphi_s(\frac{v}{x},1)^2+2[\partial_{\kappa}^2\vphi_s(\kappa,1)^2]t^2-\vphi_s(2t+\frac{v}{x},1)^2}{\abs{t}^{4s+2}e^{\frac{r}{(x\abs{t})^{2s}}\vphi_{s}(2t+\frac{v}{x},1) }} \d t\d r \\
		&\quad= \frac{1+4s}{8s}\Im\int\limits_{\gamma_{\theta}} e^{ir^{1/2s}} r^{1/2s+1} \d r\pv\int\limits_{\R}\frac{ \vphi_s(\kappa,1)^2+\frac{1}{2}[\partial_{\kappa}^2\vphi_s(\kappa,1)^2]t^2-\vphi_s(t+\kappa,1)^2}{\abs{t}^{4s+2}} \d t.
	\end{align*}
	This is the desired constant in the regime $1/4<s<3/4$ after an application of \autoref{lem:r_integral_gamma_function}. \smallskip
	
	\textbf{Case ($s=3/4$):} We define auxiliary functions
	\begin{align*}
		\tilde{h}_{2,x}(t):= \Im \int\limits_{\gamma_{\theta}} e^{i2r^{2/3}}\frac{r^{2/3+1}}{\abs{t}}e^{-\frac{r}{(\abs{t}x)^{3/2}} \vphi_{3/4}(2 t+\frac{v}{x},1) }\d r\1_{(-1,1)}(t).
	\end{align*}
	\textbf{Claim B:} The family $\{\tilde{h}_{2,x}\mid x>0\}$ converges weakly to $-3^2\pi/8^2\big(\delta_{0+}+\delta_{0-}\big)$ in the limit $x\to \infty$. \smallskip
	
	Since $ \Im \int_{\gamma_{\theta}}e^{i2r^{2/3}} r^{2/3+1}\d r$ equals zero by \autoref{lem:r_integral_gamma_function}, for any $\eps>0$ the integral $\int_{(-\eps,\eps)^c} \tilde{h}_{2,x}(t )\d t$ converges to zero as $x\to \infty$. We scale the variable $t$ by $x$ after which $\int_{0}^\infty \tilde{h}_{2,x}(t)\d t$ equals
	\begin{align*}
		 &\int\limits_{0}^x \Im \int\limits_{\gamma_{\theta}} e^{i2r^{2/3}}r^{2/3+1}t^{-1}e^{-\frac{r}{\abs{t}^{3/2}} \vphi_{3/4}(2t/x+\frac{v}{x},1) }\d r\d t \to \int\limits_{0}^\infty \Im \int\limits_{\gamma_{\theta}} e^{i2r^{2/3}}r^{2/3+1}t^{-1}e^{-\frac{r }{t^{3/2}}\vphi_{3/4}(\kappa,1) }\d r\d t \\
		 &= \int\limits_{0}^\infty \Im \int\limits_{\gamma_{\theta}} e^{i2tr^{2/3}}r^{2/3+1}t^{3}e^{-r  }\d r\d t= \frac{3}{8}\Im \int\limits_{\gamma_{\theta}} r^{-1}e^{-r  }\d r= -\frac{\pi 3}{8}\frac{3}{8}, \text{ in the limit }x\to \infty.
	\end{align*}
	Here, we scaled $r$ by $t^{3/2}$ and used the same calculation as for the term $\I{}$ in the proof of \autoref{lem:additional_cancelation}. Since the same calculation remains true for the integral of $\tilde{h}_{2,x}$ over $(-\infty,0)$, this proves the claim. \medskip 
	
	Using this claim, the \autoref{lem:r_integral_gamma_function}, and the \autoref{lem:cancelation_more_1}, we find that $2/3\Im \int_{\gamma_{\theta}} e^{i2r^{2/3}}r^{2/3}\III{}\d r$ converges to
	\begin{align*}
		&\frac{2}{3}\lim\limits_{x\to \infty}\int\limits_{\R} \frac{ \vphi_s(\kappa,1)^2+2[\partial_{\kappa}^2\vphi_s(\kappa,1)^2]t^2-\vphi_s(2t+\kappa,1)^2}{\abs{t}^{4}} \tilde{h}_{2,x}(t)\d t\\
		&\quad = -\frac{ 3\pi}{ 2^5}\lim\limits_{t\to 0+}\frac{ 2\vphi_s(\kappa,1)^2+4[\partial_{\kappa}^2\vphi_s(\kappa,1)^2]t^2-\vphi_s(2t+\kappa,1)^2-\vphi_s(-2t+\kappa,1)^2}{t^{4}}\\
		&\quad = \frac{ \pi}{ 2^3} \partial_{\kappa}^4\vphi_{3/4}(\kappa,1)^2.
	\end{align*}
	in the limit $x\to \infty$. \medskip 

	\textbf{Case $3/4<s<1$:} Just as in the case $1/4<s<3/4$, we may add a term to $\III{}$ which does not change the result in the limit $x\to \infty$, see \autoref{lem:cancelation_more_1}. Furthermore, we may add another term using \autoref{lem:cancelation_more_2}. This yields
	\begin{align*}
		&\lim\limits_{x\to \infty}\frac{1}{2s}\Im\int\limits_{\gamma_{\theta}} e^{2ir^{1/2s}} r^{1/2s}\int\limits_{\R} \III{}\d t \d r\\
		&\quad =\frac{1+4s}{8s}\lim\limits_{x\to \infty}\Im\int\limits_{\gamma_{\theta}} e^{2ir^{1/2s}} r^{1/2s+1}\\
		&\qquad \times\int\limits_{\R}  \frac{ \vphi_s(\frac{v}{x},1)^2+2[\partial_{\kappa}^2\vphi_s(\kappa,1)^2]t^2+\frac{2^4}{4!}[\partial_{\kappa}^4\vphi_s(\kappa,1)^4]t^4-\vphi_s(2t+\frac{v}{x},1)^2}{\abs{t}^{4s+2}e^{\frac{r}{(x\abs{t})^{2s}}\vphi_{s}(2t+\frac{v}{x},1) }} \d t\d r \\
		&\quad= \frac{1+4s}{8s}\Im\int\limits_{\gamma_{\theta}} e^{ir^{1/2s}} r^{1/2s+1} \d r\\
		&\qquad\times\pv\int\limits_{\R}\frac{ \vphi_s(\kappa,1)^2+\frac{1}{2}[\partial_{\kappa}^2\vphi_s(\kappa,1)^2]t^2+\frac{1}{4!}[\partial_{\kappa}^4\vphi_s(\kappa,1)^4]t^4-\vphi_s(t+\kappa,1)^2}{\abs{t}^{4s+2}} \d t.
	\end{align*}
	Finally, the \autoref{lem:r_integral_gamma_function} yields the desired result in the case $3/4<s<1$.\medskip
	
	It remains to prove that both the term $\II{}$ and $\IV{}$ yield to no contribution in the limit. We prove this in the following two claims.\medskip
	
	\textbf{Claim C:} The term $\Im \int\limits_{\gamma_{\theta}} e^{2ir^{1/2s}}r^{1/2s}\IV{}\d r$ converges to zero as $x\to \infty$. \smallskip
	
	 We note that the function
	 \begin{equation}\label{eq:def_rho_x}
	 	\rho_x:t\mapsto \sgn(t)\Im \int\limits_{\gamma_{\theta}}e^{ir^{1/2s}}r^{1/2s+1}\Big(\partial_te^{-\frac{r}{(x\abs{t})^{2s}}\vphi_{s}(2t+\frac{v}{x},1) }\Big)\d r
	 \end{equation}
	 converges to $c_1(s)(\delta_{0+}+\delta_{0-})$ for some constant $c_1(s)\in \R$. Here $c_1(1/4)=c_1(3/4)=0$ by \autoref{lem:r_integral_gamma_function}. This immediately implies the claim in the case $s\le 1/4$ since the function
	 \begin{equation*}
	 	\frac{2\vphi_s(\frac{v}{x},1)^2-\vphi_s(2t+\frac{v}{x},1)^2-\vphi_s(-2t+\frac{v}{x},1)^2}{\abs{t}^{1+4s}}
	 \end{equation*}
	 is zero at $t=0$, if $s<1/4$, and bounded in a neighborhood of $t=0$. In the case $s\in (1/4,1)$, we note that 
	 \begin{align*}
	 	 \Im \int\limits_{\R} \int\limits_{\gamma_{\theta}}e^{ir^{1/2s}}r^{1/2s+1}\frac{t}{\abs{t}^{4s}}\Big(\partial_te^{-\frac{r}{(x\abs{t})^{2s}}\vphi_{s}(2t+\frac{v}{x},1) }\Big)\d r\d t
	 \end{align*}
 	converges to zero by \autoref{lem:cancelation_more_1}. Thus, we may add this term with the factor $2[\partial_\xi\vphi_s^2](v/x,1)$ to $\Im\int_{\gamma_{\theta}} e^{2ir^{1/2s}}r^{1/2s}\IV{}\d r$ without changing the result in the limit $x\to \infty$. Now, we use again the above approximation of the identity. Since 
 	\begin{align*}
 		\frac{2\vphi_s(v/x,1)^2-4[\partial_{\xi}\vphi_s^2](v/x,1)t^2-\vphi_s(2t+v/x,1)^2-\vphi_s(-2t+v/x,1)^2}{\abs{t}^{1+4s}}
 	\end{align*}
 	is zero at $t=0$, if $s<3/4$, and bounded in a neighborhood of $t=0$, the claim follows in the case $1/4<s\le 3/4$. In the last regime $3/4<s<1$, we prove the claim with the same arguments as before but with the additional observation that 
 	\begin{align*}
 		\Im \int\limits_{\R} \int\limits_{\gamma_{\theta}}e^{ir^{1/2s}}r^{1/2s+1}\frac{t^3}{\abs{t}^{4s}}\Big(\partial_te^{-\frac{r}{(x\abs{t})^{2s}}\vphi_{s}(2t+\frac{v}{x},1) }\Big)\d r\d t
 	\end{align*}
 	converges to zero as $x\to \infty$ if $3/4<s<1$ by \autoref{lem:cancelation_more_2}.\medskip
 	
 	\textbf{Claim D:} The term $\Im \int\limits_{\gamma_{\theta}} e^{2ir^{1/2s}}r^{1/2s}\II{}\d r$ converges to zero as $x\to \infty$. \smallskip
 	
 	First, we write
 	\begin{align*}
 		h_a(t)= \abs{t}^{2s}\int\limits_{0}^1 [\partial_{\xi}\vphi_s](2/t+ay,y)\d y=:\abs{t}^{2s}\tilde{h}_a(1/t)
 	\end{align*}
 	and note that $\tilde{h}_a(t)$ is bounded and differentiable in a neighborhood of $0$. Second, the use the change of variables $t \mapsto 1/t$ to find that $\II{}$ equals
 	\begin{align*}
 		 r\int\limits_{\R} \frac{\tilde{h}_a(t)\vphi_s(2t+\frac{v}{x},1)}{\abs{t}^{4s}}\Big(\partial_t e^{-\frac{r}{(x\abs{t})^{2s}}\vphi_{s}(2t+\frac{v}{x},1) }\Big)\d t.
 	\end{align*}
 	We integrate $\II{}$ together with $e^{ir^{1/2s}} r^{1/2s}$ over $r$ and take the imaginary part. Using the definition of $\rho_x$, see \eqref{eq:def_rho_x}, we write
 	\begin{align*}
 		\Im \int\limits_{\gamma_{\theta}} e^{ir^{1/2s}}r^{1/2s}\II{}\d r &= \int\limits_{\R} \sgn(t)\frac{\tilde{h}_{v/x}(t)\vphi_s(2t+v/x,1)}{\abs{t}^{4s}} \rho_x(t)\d t. 
 	\end{align*}
 	Recall that $\rho_x\to c_1(s)(\delta_{0+}+\delta_{0-})$, see the proof of claim C. Thus, the claim follows in the case $0<s\le 1/4$ since 
 	\begin{align*}
 		\frac{\tilde{h}_{v/x}(t)\vphi_s(2t+v/x,1)-\tilde{h}_{v/x}(-t)\vphi_s(-2t+v/x,1)}{\abs{t}^{4s}}
 	\end{align*}
 	is zero at $t=0$, if $s<1/4$, and bounded in a neighborhood of $0$. The remaining cases follow using the same arguments as in the proof of claim C, i.e.\ we use \autoref{lem:cancelation_more_1} for $1/4<s<1$ and \autoref{lem:cancelation_more_2} for $3/4<s<1$. 
\end{proof}

\begin{lemma}\label{lem:fasymp_d1_s_x_s_ge_12_positivity_s_14}
	The constant $C_{1/4,3}(\kappa)$ from \autoref{prop:fasymp_d1_s_x_s_ge_12} is uniformly positive and bounded for $\kappa\in [0,2)$.
\end{lemma}
\begin{proof}
	Since $\vphi_{1/4}(\cdot, 1)^2$ is continuous and sufficiently differentiable in $[0,2)$, the map $\kappa \mapsto C_{1/4,3}(\kappa)$ is continuous in $[0,2)$. Thus, it suffices to prove that $C_{1/4,3}(\kappa)$ is positive for any $\kappa\in (0,2)$, positive in the limit $\kappa \to 0+$, and positive and bounded in the limit $\kappa\to 2-$. So fix $\kappa\in (0,2)$. We can simply calculate the constant $C_{1/4,3}(\kappa)$ explicitly. It suffices to consider
	\begin{equation*}
		-\abs{2-\kappa}^{1/2}\partial_{\kappa}^2\vphi_{1/4}(\kappa,1)^2.
	\end{equation*}
	We start by rewriting this second derivative by a change of variables. Note that
	\begin{align*}
		\vphi_s(t,1)^2= \frac{t^{4s}}{\kappa^{4s}}\vphi_s(\kappa,\frac{\kappa}{t})^2.
	\end{align*}
	Therefore, by chain rule, we find that for any $a\in \R$
	\begin{equation}\label{eq:second_derivative_vphis_change_of_variables}
		-\partial_{\kappa}^2 \vphi_s(\kappa,a)^2= -\frac{a^{2}}{\kappa^{2}}\partial_a^2\vphi_s(\kappa,a)^2+2\big(4s-1\big)\frac{a^1}{\kappa^{2}}\partial_a\vphi_s(\kappa,a)^2-4s(4s-1)\frac{1}{\kappa^{2}}\vphi_s(\kappa,a)^2
	\end{equation}
	This allows us to write
	\begin{align*}
		-\partial_{\kappa}^2 \vphi_{1/4}(\kappa,1)^2=-\frac{1}{\kappa^{2}}[\partial_{\nu}^2\vphi_{1/4}^2](\kappa,1).
	\end{align*}

	We start of calculating the second derivative
	\begin{align*}
		2(2s+1)\kappa^2\partial_{a}^2 \vphi_s(\kappa,a)^2&= \partial_{a}^2 \frac{\Big( (a+\frac{\kappa}{2})^{1+2s}-(a-\frac{\kappa}{2})^{1+2s}  \Big)^2}{(2s+1)2}\\
		&= \partial_{a} \big( (a+\frac{\kappa}{2})^{1+2s}-(a-\frac{\kappa}{2})^{1+2s}  \big)\big( (a+\frac{\kappa}{2})^{2s}-(a-\frac{\kappa}{2})^{2s} \big)\\
		&= (1+2s)\big( (a+\frac{\kappa}{2})^{2s}-(a-\frac{\kappa}{2})^{2s} \big)^2 \\
		&\qquad+2s\big( (a+\frac{\kappa}{2})^{1+2s}-(a-\frac{\kappa}{2})^{1+2s}  \big)\big( (a+\frac{\kappa}{2})^{2s-1}-(a-\frac{\kappa}{2})^{2s-1} \big).
	\end{align*}
	In the present case, i.e.\ $a=1$ and $s=1/4$, this equals
	\begin{align*}
		4-\frac{3\big(1-\frac{\kappa^2}{4}\big)
			+\Big(1+\frac{\kappa^2}{4}\Big)}{\big(1-\frac{\kappa^2}{4}\big)^{1/2}}=4-\frac{4-\frac{\kappa^2}{2}}{\big(1-\frac{\kappa^2}{4}\big)^{1/2}}=2\frac{2\big(1-\frac{\kappa^2}{4}\big)^{1/2}-2+\frac{\kappa^2}{4}}{\big(1-\frac{\kappa^2}{4}\big)^{1/2}}.
	\end{align*}
	Thus, the constant $C_{1/4,3}(\kappa)$ equals
	\begin{align*}
		\frac{1+\kappa^{5/2}}{\sqrt{2}\pi}\frac{2-\frac{\kappa^2}{4}-2\big(1-\frac{\kappa^2}{4}\big)^{1/2}}{\big(1+\frac{\kappa}{2}\big)^{1/2}\kappa^4}.
	\end{align*}
	This term is continuous on $[0,2]$ and strictly positive for all $\kappa\in [0,2]$ since $2-\frac{\kappa^2}{4}-2\big(1-\frac{\kappa^2}{4}\big)^{1/2}\ge \kappa^4/2^6$.
\end{proof}

\begin{lemma}\label{lem:fasymp_d1_s_x_s_ge_12_positivity_s_34}
	The constant $C_{3/4,3}(\kappa)$ from \autoref{prop:fasymp_d1_s_x_s_ge_12} is uniformly positive and bounded for $\kappa\in [0,2)$.
\end{lemma}
\begin{proof}
	The proof follows the same arguments as the proof of \autoref{lem:fasymp_d1_s_x_s_ge_12_positivity_s_14}. Here, it remains to consider the term 
	\begin{align*}
		\abs{2-\kappa}^{3/2}\partial_{\kappa}^4\vphi_{3/4}(\kappa,1)^2.
	\end{align*}
	 Notice that $\partial_{\kappa}^4\vphi_s(\kappa,1)^2 = \kappa^{-4}\partial_{a}^4\vphi_s(\kappa,a)^2\mid_{a=1}$. Now, we calculate the derivative for $a>\kappa/2$:
	\begin{align*}
		\kappa^24(2s+1)^2\partial_{a}^4\vphi_s(\kappa,a)^2\mid_{a=1}&= \partial_a^4 \Big(\big(a+\frac{\kappa}{2}\big)^{1+3/2}-\big(a-\frac{\kappa}{2}\big)^{1+3/2}\Big)^2\mid_{a=1}\\
		&=15\frac{64\cdot(\sqrt{4-\kappa^2}-2)+16\kappa^2(3-\sqrt{4-\kappa^2})-3\kappa^4}{(4-\kappa^2)^{3/2}}
	\end{align*}
	Therefore, 
	\begin{align*}
		\frac{5}{3}\abs{1-\kappa/2}^{3/2}\partial_\kappa^4 \vphi_{3/4}(\kappa,1)^2	&= \frac{64(\sqrt{4-\kappa^2}-2)+16\kappa^2(3-\sqrt{4-\kappa^2})-3 \kappa^4}{\kappa^6(1+\kappa/2)^{3/2}}\\
		&=\frac{16\cdot (4-\kappa^2)(\sqrt{4-\kappa^2}-2)+16\kappa^2-3 \kappa^4}{\kappa^6(1+\kappa/2)^{3/2}}.
	\end{align*}
	It is easy to see that this expression is positive for all $\kappa\in (0,2)$. Furthermore, using the bound
	\begin{align*}
		\sqrt{4-\kappa^2}\ge 2-\frac{\kappa^2}{4}-\frac{\kappa^4}{2^6}-\frac{3}{2^8}\kappa^6,
	\end{align*}
	one readily sees that $C_{s,3}(\kappa)$ is positive in both the limits $\kappa\to 0+$ and $\kappa\to 2-$. 
\end{proof}

\begin{lemma}\label{lem:fasymp_d1_s_x_s_ge_12_boundedness}
	For any $s\in (0,1)\setminus\{ 1/4,3/4 \}$ the constant $C_{s,3}(\kappa)$ from \autoref{prop:fasymp_d1_s_x_s_ge_12} is continuous on $[0,2]$ and positive in the limit $\kappa \to 2-$.
\end{lemma}
\begin{proof}
	Since $\vphi_{s}(\cdot, 1)^2$ is continuous, sufficiently differentiable in a neighborhood of $[0,2)$ and decays like $\abs{\cdot}^{4s}$, the map $\kappa\to C_{s,3}(\kappa)$ is continuous in $[0,2)$. It remains to consider the limit $\kappa \to 2-$. For this, we distinguish the cases $0<s<1/4$, $1/4<s<3/4$, and $3/4<s<1$. \medskip
	
	\textbf{Case $0<s<1/4$.} The change of variables $y= t/(\kappa+t)$ reveals that $	\abs{2-\kappa}^{2s}(-\Delta_{\kappa})^\alpha \vphi_s(\kappa,1)^2$ equals
	\begin{align*}
		& c_{\alpha}\pv \int\limits_{\R} \frac{\vphi_s(\kappa,1)^2-\vphi_s(\kappa+t,1)^2}{\abs{t}^{2+4s}}\d t =  \frac{c_\alpha}{\kappa^{1+4s}} \pv \int\limits_{\R} \frac{\vphi_s(\kappa,1)^2\abs{1-y}^{4s}-\vphi_s(\kappa,1-y)^2}{\abs{y}^{2+4s}}\d y.
	\end{align*}
	Since $\abs{\cdot}^{4s}$, see e.g.\ \cite{Lan72} in particular equation (1.1.1) and section 6 therein, is the fundamental solution for the fractional Laplace operator $(-\Delta)^\alpha$, the above may be written as 
	\begin{align*}
		\frac{\abs{2-\kappa}^{2s}}{\kappa^{1+4s}} (-\Delta)^\alpha\vphi_s(\kappa,\cdot)^2(1).
	\end{align*}
	Now, we use the Leibniz rule for nonlocal operators, after which $	\abs{2-\kappa}^{2s}(-\Delta_{\kappa})^\alpha \vphi_s(\kappa,1)^2$ equals
	\begin{align*}
		\frac{\abs{2-\kappa}^{2s}}{\kappa^{1+4s}}\Big( 2\vphi_s(\kappa,1)(-\Delta)^\alpha\vphi_s(\kappa,\cdot)(1)-\Gamma_\alpha\big(\vphi_s(\kappa,\cdot)\big)(1)  \Big). 
	\end{align*}
	Here, $\Gamma_\alpha$ is the carr{\'e} du champ for the fractional Laplace operator of order $\alpha$. Since $\vphi_s(\kappa,t)\le c_1 (1+\abs{t})^{2s}$ for all $t\in \R$ and some constant $c_1>0$ independent of $t$ and $\kappa$ and 
	\begin{align*}
		\kappa\abs{\vphi_s(\kappa,1)-\vphi_s(\kappa,t)}\le \big(  \abs{1+\kappa/2}^{2s}+1\big) \abs{1-t}
	\end{align*}
	for all $t\in (-1,1)$, the carr{\'e} du champ $\Gamma_\alpha(\vphi_s(\kappa,\cdot))(1)$ is bounded in the limit $\kappa\to 2-$. Therefore, $\abs{2-\kappa}^{2s}\Gamma_{\alpha}(\vphi_s(\kappa,\cdot))(1)$ converges to zero as $\kappa \to 2-$. It remains to consider 
	\begin{align*}
		\abs{2-\kappa}^{2s}(-\Delta)^\alpha\vphi_s(\kappa,\cdot)(1)&=\abs{2-\kappa}^{2s}c_\alpha \int\limits_{0}^\infty \frac{2\vphi_s(\kappa,1)-\vphi_s(\kappa,1+t)-\vphi_s(\kappa,1-t)}{t^{2+4s}}\d t\\
		&= \frac{\abs{2-\kappa}^{2s}}{1+4s}c_\alpha \int\limits_{0}^\infty \frac{[\partial_{\nu}\vphi_s](\kappa,1-t)-[\partial_{\nu}\vphi_s](\kappa,1+t)}{t^{1+4s}}\d t\\
		&= \frac{\abs{2-\kappa}^{2s}}{(1+4s)(1+\kappa/2)^{2s}}c_\alpha \int\limits_{0}^\infty 	\frac{\abs{1-t}^{2s}-\abs{1+t}^{2s}}{t^{1+4s}\kappa}\d t\\
		&\quad + \frac{2^{2s}}{1+4s}c_\alpha \int\limits_{0}^\infty 	\frac{\abs{1+t}^{2s}-\abs{1-t}^{2s}}{t^{1+4s}\kappa}\d t
	\end{align*}
	In the first equality, we integrated by parts and, in the second, we scaled $t$ by $1+\kappa/2$ respectively $1-\kappa/2$. Thus, in the limit $\kappa \to 2-$, the term $\abs{2-\kappa}^{2s}(-\Delta)^\alpha\vphi_s(\kappa,\cdot)(1)$ equals
	\begin{align*}
		\frac{2^{2s-1}}{1+4s}c_\alpha \int\limits_{0}^\infty 	\frac{\abs{1+t}^{2s}-\abs{1-t}^{2s}}{t^{1+4s}}\d t.
	\end{align*}
	This term is positive since the integrand is.\medskip
	
	\textbf{Case $1/4<s<3/4$.} We use integration by parts twice to find that $(1+4s)4sc_{\alpha}^{-1}(-\Delta_{\kappa})^\alpha \vphi_s(\kappa, 1)$ equals
	\begin{align*}
		\pv \int\limits_{\R} \frac{[\partial_{\xi}^2\vphi_s^2](\kappa,1)-[\partial_{\xi}^2\vphi_s^2](\kappa+t,1)}{\abs{t}^{4s}}\d t=\I{}+\II{}.
	\end{align*}
	Here, the terms $\I{}$ and $\II{}$ are defined by 
	\begin{align*}
		\I{}&:=2\pv \int\limits_{\R} \frac{[\partial_{\xi}\vphi_s]^2(\kappa,1)-[\partial_{\xi}\vphi_s]^2(\kappa+t,1)}{\abs{t}^{4s}}\d t,\\
		\II{}&:=2\pv \int\limits_{\R} \frac{\vphi_s(\kappa,1)[\partial_{\xi}^2\vphi_s](\kappa,1)-\vphi_s(\kappa+t,1)[\partial_{\xi}^2\vphi_s](\kappa+t,1)}{\abs{t}^{4s}}\d t.
	\end{align*}
	Since $[\partial_{\xi}\vphi_s](\kappa+t,1)$ is $C^{4s-1+\eps}$ is a neighborhood of $t=0$ with the bound
	\begin{align*}
		\norm{[\partial_{\xi}\vphi_s](\kappa+\cdot,1)}_{C^{4s-1+\eps}((-\kappa/4,\kappa/4))}\le c_1 (1+\abs{2-\kappa}^{-2s+1-\eps}),
	\end{align*}
	the term $\abs{1-\kappa/2}^{2s}\I{}$ converges to zero as $\kappa\to 2-$. Using Leibniz rule, we may write $2^{-1}\II{}$ as $\II{1}+\II{2}$, where
	\begin{align*}
		\II{1}&:=\vphi_s(\kappa,1)\pv \int\limits_{\R} \frac{[\partial_{\xi}^2\vphi_s](\kappa,1)-[\partial_{\xi}^2\vphi_s](\kappa+t,1)}{\abs{t}^{4s}}\d t\\
		\II{2}&:=\pv \int\limits_{\R}\Bigg[ [\partial_{\xi}^2\vphi_s](\kappa,1)\Big(\vphi_s(\kappa,1)-\vphi_s(\kappa+t,1)\Big)\\ 
		&\qquad -  \Big([\partial_{\xi}^2\vphi_s](\kappa,1)-[\partial_{\xi}^2\vphi_s](\kappa+t,1)\Big)\Big( \vphi_s(\kappa,1)-\vphi_s(\kappa+t,1) \Big)\Bigg]\frac{\d t}{\abs{t}^{4s}}.
	\end{align*}
	Just as for the term $\I{}$, the term $\abs{2-\kappa}^{2s}\II{2}$ converges to zero as $\kappa \to 2-$. This follows from the following observations: The map $t\mapsto \vphi_s(\kappa+t,1)$ is $C^{1+2s-\eps}$ in a neighborhood of the origin uniformly in $\kappa \in [0,2]$ and 
	\begin{align*}
		\norm{[\partial_{\xi}^2\vphi_s](\kappa+\cdot,1)}_{C^{4s-2+\eps}((-\kappa/4,\kappa/4))}\le c_1 (1+\abs{2-\kappa}^{-2s+2-\eps}),
	\end{align*}
	if $s\ge 1/2$, and bounded by a constant multiple of $1+\abs{2-\kappa}^{2s-1}$, if $s<1/2$, for all $\kappa \in (1,2]$. It remains to consider the term $\II{1}$. We use the change of variables $y= t/(\kappa+t)$. Thereafter, $\vphi_s(\kappa,1)^{-1}\II{2}$ equals
	\begin{align*}
		\III{}&:= \pv \int\limits_{\R} \frac{[\partial_{\xi}^2\vphi_s](\kappa,1)-[\partial_{\xi}^2\vphi_s](\frac{\kappa }{1-y},1)}{\abs{\kappa y / (1-y)}^{4s}}\frac{\kappa}{(1-y)^2}\d y\\
		&= \kappa^{1-4s} \pv \int\limits_{\R} \frac{[\partial_{\xi}^2\vphi_s](\kappa,1)\abs{1-y}^{4s-2}-\abs{1-y}^{2s}[\partial_{\xi}^2\vphi_s](\kappa ,1-y)}{\abs{y }^{4s}}\d y.
	\end{align*}
	We continue with a minor side calculation. Note that $\vphi_s(\kappa,1-y) = \kappa^{2s}\vphi_s(1,(1-y)/\kappa)$ and, thus,
	\begin{align*}
		\partial_\kappa \vphi_s(\kappa, 1-y)&= 2s \kappa^{2s-1}\vphi_s(1,(1-y)/\kappa)- \kappa^{2s-2}(1-y)[\partial_{\nu}\vphi_s](\kappa,(1-y)/\kappa) \\
		\partial_\kappa^2 \vphi_s(\kappa, 1-y)&= 2s(2s-1) \kappa^{2s-2}\vphi_s(1,(1-y)/\kappa)- 2s \kappa^{2s-3}(1-y)[\partial_{\nu}\vphi_s](1,(1-y)/\kappa)\\
		&\qquad - (2s-2)\kappa^{2s-3}(1-y)[\partial_{\nu}\vphi_s](1,(1-y)/\kappa)+ \kappa^{2s-4}(1-y)^2[\partial_{\nu}^2\vphi_s](1,(1-y)/\kappa).
	\end{align*}
	By the scaling properties of $\vphi_s$, we may write
	\begin{align*}
		\partial_\kappa^2 \vphi_s(\kappa, 1-y)&= \frac{2s(2s-1) }{\kappa^{2}}\vphi_s(\kappa,1-y)- \frac{2(2s-1)}{\kappa^{2}} (1-y)[\partial_{\nu}\vphi_s](\kappa,1-y)\\
		&\qquad + \frac{(1-y)^2}{\kappa^{2}}[\partial_{\nu}^2\vphi_s](\kappa,1-y).
	\end{align*}
	Since $\abs{\cdot}^{4s-2}$ is the fundamental solution to $(-\Delta)^{2s-1/2}$ and using the previous calculations, the term $\III{}$ may be written as  $\III{1}+\III{2}+\III{3}$, where
	\begin{align*}
		\III{1}&:=\frac{(2s-1)}{\kappa^{1+4s}}\\
		&\quad\times\pv \int\limits_{\R} \frac{2s\vphi_s(\kappa,1)\abs{1-y}^{2s-2}-\vphi_s(\kappa,1-y)-2y[\partial_{\nu}\vphi_s](\kappa,1-y)- \frac{y^2-2y}{2s-1}[\partial_{\nu}^2\vphi_s](\kappa,1-y)}{\abs{y}^{4s}\abs{1-y}^{-2s}}\d y,\\
		 \III{2}&:= -\frac{2(2s-1)}{\kappa^{1+4s}}\pv \int\limits_{\R} \frac{[\partial_{\nu}\vphi_s](\kappa,1)\abs{1-y}^{2s-2}-[\partial_{\nu}\vphi_s](\kappa,1-y)}{\abs{1-y}^{-2s}\abs{y}^{4s}}\d y,\\
		 \III{3}&:= \frac{1}{\kappa^{1+4s}}\pv \int\limits_{\R} \frac{[\partial_{\nu}^2\vphi_s](\kappa,1)-\abs{1-y}^{2s}[\partial_{\nu}^2\vphi_s](\kappa,1-y)}{\abs{y}^{4s}}\d y.
	\end{align*}
	Note that in $\III{1}$ and $\III{2}$ the principle value is to be taken both in the origin and at infinity. Just as for the term $\I{}$, it is easy to see that $(2-\kappa)^{2s}\III{1}$ and $(2-\kappa)^{2s}\III{2}$ converge to zero as $\kappa \to 2-$. Finally, we see that 
	\begin{equation*}
		\frac{(1-\kappa/2)^{2s}\kappa^{1+4s}}{2s}\III{3}
	\end{equation*}
	equals
	\begin{align*}
		&\int\limits_{0}^\infty \frac{2(1+\kappa/2)^{2s-1}-\abs{1+y}^{2s}(1+\kappa/2+y)^{2s-1}-\abs{1-y}^{2s}(1+\kappa/2-y)\abs{1+\kappa/2-y}^{2s-2}}{(1-\kappa/2)^{-2s}y^{4s}}\d y \\
		&\,\, - \int\limits_{0}^\infty \frac{2(1-\kappa/2)^{2s-1}-\abs{1+y}^{2s}(1-\kappa/2+y)^{2s-1}-\abs{1-y}^{2s}(1-\kappa/2-y)\abs{1+\kappa/2-y}^{2s-2}}{(1-\kappa/2)^{-2s}y^{4s}}\d y.
	\end{align*}
	Now, we scale $y$ by $(1+\kappa/2)$ in the first term and by $(1-\kappa/2)$ in the second one. Thus, the previous term equals 
	\begin{align*}
		&\frac{(1-\kappa/2)^{2s}}{(1+\kappa/2)^{2s}}\int\limits_{0}^\infty \frac{2-\abs{1+(1+\kappa/2)y}^{2s}(1+y)^{2s-1}-\abs{1-(1+\kappa/2)y}^{2s}(1-y)\abs{1-y}^{2s-2}}{y^{4s}}\d y \\
		&\qquad - \int\limits_{0}^\infty \frac{2-\abs{1+(1-\kappa/2)y}^{2s}(1+y)^{2s-1}-\abs{1-(1-\kappa/2)y}^{2s}(1-y)\abs{1-y}^{2s-2}}{y^{4s}}\d y
	\end{align*}
	which converges to 
	\begin{equation}\label{eq:fasymp_limit_2_14_s_34}
		-\int\limits_{0}^\infty \frac{2-(1+y)^{2s-1}-(1-y)\abs{1-y}^{2s-2}}{y^{4s}}\d y 
	\end{equation}
	in the limit $\kappa \to 2-$. \smallskip
	
	This proves that the constant $\kappa\mapsto C_{s,3}(\kappa)$ is continuous on $[0,2]$. Now, we check that the limit $\kappa\to 2-$ is positive. Note that the normalization constant of $(-\Delta)^\alpha$ is negative in the case $1/4<s<3/4$. Thus, it remains to prove that the expression \eqref{eq:fasymp_limit_2_14_s_34} is negative. In the case $s=1/2$, this is obvious. Now, we distinguish the cases $1/4<s<1/2$ and $1/2<s<3/4$. \smallskip
	
	First, we assume $1/4<s<1/2$. We split the integral in \eqref{eq:fasymp_limit_2_14_s_34} into the part over $(0,1)$ and $(1,\infty)$. On $(1,\infty)$, the use the change of variables $t\mapsto 1/t$. This yields
	\begin{align*}
		&\int\limits_{0}^\infty \frac{2-(1+y)^{2s-1}-(1-y)\abs{1-y}^{2s-2}}{y^{4s}}\d y\\
		&= \int\limits_{0}^1 \frac{2(1+y^{8s-2})-(1+y)^{2s-1}(1+y^{6s-1})-(1-y^{6s-1})(1-y)\abs{1-y}^{2s-2}}{y^{4s}}\d y.
	\end{align*}
	We distinguish two cases. If $s<1/2$, then note that
	\begin{align*}
		1+y^{8s-2}-(1+y)^{2s-1}(1+y^{6s-1})> 1+y^{8s-2}-(1+y^{6s-1})= y^{6s-1}(y^{2s-1}-1)>0.
	\end{align*}
	Furthermore, since $(1-y^{6s-1})\le 1-y^2$ for any $y\in (0,1)$
	\begin{equation*}
		(1+y^{8s-2})-(1-y^{6s-1})(1-y)\abs{1-y}^{2s-2}\ge 1+ y^{8s-2}-(1+y)\abs{1-y}^{2s}
	\end{equation*}
	which is strictly positive. Thus, we have proven that the expression \eqref{eq:fasymp_limit_2_14_s_34} is negative. Now, we consider the case $1/2<s<3/4$. Since the numerator $2-(1+y)^{2s-1}-(1-y)\abs{1-y}^{2s-2}$ is zero at $y=0$ and its derivative $(2s-1)(\abs{1-y}^{2s-2}-(1+y)^{2s-2})$ is strictly positive for all $0<y$, we know that $2-(1+y)^{2s-1}-(1-y)\abs{1-y}^{2s-2}$ is positive. This proves the positivity of $\lim_{\kappa \to 2-} C_{s,3}(\kappa)$.\medskip
	
	\textbf{Case $3/4<s<1$.} We proceed similar as in the proof in the case $1/4<s<3/4$. Using integration by parts three times reveals
	\begin{align*}
		\IV{}&:=(4s+1)4s(4s-1)c_{\alpha}^{-1}(-\Delta_\kappa)^\alpha\vphi_s(\kappa,1)^2= \pv \int\limits_{\R} \frac{[\partial_{\kappa}^4 \vphi_s^2](\kappa,1)t-[\partial_{\xi}^3 \vphi_s^2](\kappa+t,1)}{t\abs{t}^{4s-2}}\d t.
	\end{align*}
	Just as in the previous case, we notice that the only contributing term to $(2-\kappa)^{2s}\IV{}$ in the limit $\kappa \to 2-$ is 
	\begin{align*}
		\V{}&:=\vphi_s(\kappa,1)\pv \int\limits_{\R} \frac{[\partial_{\xi}^4 \vphi_s](\kappa,1)t-[\partial_{\xi}^3 \vphi_s](\kappa+t,1)}{t\abs{t}^{4s-2}}\d t.
	\end{align*}
	Now, we scale $t$ by $(1-\kappa/2)$. Thereafter, $(1-\kappa/2)^{2s}\,\vphi_s(\kappa,1)^{-1}\V{}$ equals
	\begin{align*}
		(1-\kappa/2)^{-2s+2}\int\limits_{0}^\infty \frac{2[\partial_{\xi}^4 \vphi_s](\kappa,1)(1-\kappa/2)t-[\partial_{\xi}^3 \vphi_s](\kappa+(1-\kappa/2)t,1)+[\partial_{\xi}^3 \vphi_s](\kappa-(1-\kappa/2)t,1)}{t\abs{t}^{4s-2}}\d t.
	\end{align*}
	With arguments just as in the previous cases, it is obvious that only the terms for which all derivatives with respect to $\xi$ fall onto $-(\nu-\xi/2)\abs{\nu-\xi/2}^{2s}$ in the function $\vphi_s$, see \eqref{eq:phi_s}, contribute in the limit $\kappa \to 2-$. Thus, $\lim_{\kappa \to 2-} (2-\kappa)^{2s}\V{}$ equals
	\begin{align*}
		2^{2s-4}2s(2s-1)\int\limits_{0}^\infty \frac{-(2s-2)t+\abs{1-t/2}^{2s-2}-\abs{1+t/2}^{2s-2}}{t^{4s-1}}\d t
	\end{align*}
	Since $2s-2<0$ and for any $t>0$ we know that $\abs{1-t/2}<\abs{1+t/2}$, the integrand in the previous integral is positive for any $t\in (0,\infty)$. Thus, the limit $\lim_{\kappa \to 2-}C_{s,3}(\kappa)$ exists and is positive. 
\end{proof}

\begin{lemma}\label{lem:fasymp_d1_s_x_s_ge_12_positivity_kappa_0}
	For any $s\in (0,1)\setminus\{ 1/4,3/4 \}$ the constant $C_{s,3}(0)$ from \autoref{prop:fasymp_d1_s_x_s_ge_12} is positive.
\end{lemma}
\begin{proof}
	We divide the proof into the cases $0<s<1/4$, $1/4<s<3/4$, and $3/4<s<1$. \smallskip
	
	\textbf{Case $0<s<1/4$.} By the definition of the constant $C_{s,3}(0)$ it suffices to prove that $(-\Delta_\kappa)^{\alpha}\vphi_s(0,1)^2$ is positive. Since $0<\alpha<1$, the normalization constant $c_{\alpha}$ of the operator $(-\Delta)^\alpha$ is positive. Thus, it suffices to prove the positivity of 
	\begin{align*}
		\I{}:=\pv\frac{1}{2}\int\limits_{\R} \frac{\vphi_s(0,1)^2-\vphi_s(t,1)^2}{\abs{t}^{1+2\alpha}}\d t&= \int\limits_{0}^\infty \frac{\vphi_s(0,1)^2-\vphi_s(t,1)^2}{\abs{t}^{1+2\alpha}}\d t\\
		=\int\limits_{0}^\infty \frac{\vphi_s(0,1)^2-\vphi_s(2t,1)^2}{4^\alpha\abs{t}^{1+2\alpha}}\d t&= \int\limits_{0}^\infty \frac{(2s+1)^2 4 t^2-\Big((1+t)^{1+2s}-(1-t)\abs{1-t}^{2s}\Big)^2}{4^{1+\alpha}\abs{t}^{4+4s}(2s+1)^2}\d t.
	\end{align*}
	Here, we used the symmetry of $\vphi_s(\cdot, 1)$, $\vphi_s(0,1)=1$, and scaled $t$ by $2$, see \eqref{eq:k_symmetry}. Now, we use integration by parts twice, after which the previous term equals
	\begin{align*}
		\I{}=\int\limits_{0}^\infty &\frac{8(2s+1)^2-\partial_t^2\Big((1+t)^{1+2s}-(1-t)\abs{1-t}^{2s}\Big)^2}{\abs{t}^{2+4s}(2s+1)^2(3+4s)(2+4s)4^{1+\alpha}}\d t\\
		&= \int\limits_{0}^\infty \frac{4(2s+1)^2\Big(\abs{1+t}^{4s}+\abs{1-t}^{4s}\Big)-\partial_t^2\Big((1+t)^{1+2s}-(1-t)\abs{1-t}^{2s}\Big)^2}{\abs{t}^{2+4s}(2s+1)^2(3+4s)(2+4s)4^{1+\alpha}}\d t.
	\end{align*}
	The previous equality is justified since $\abs{\cdot}^{4s}$ is the fundamental solution to $(-\Delta)^\alpha$. Now, let's analyze the numerator. For $t>0$, we calculate
	\begin{align*}
		4&(2s+1)^2\Big(\abs{1+t}^{4s}+\abs{1-t}^{4s}\Big)-\partial_t^2\Big((1+t)^{1+2s}-(1-t)\abs{1-t}^{2s}\Big)^2\\
		&= 4(2s+1)^2\Big(\abs{1+t}^{4s}+\abs{1-t}^{4s}\Big)-2(1+2s)^2\Big((1+t)^{2s}+\abs{1-t}^{2s}\Big)^2\\
		&\qquad -2(1+2s)(2s)\Big((1+t)^{1+2s}-(1-t)\abs{1-t}^{2s}\Big)\Big((1+t)^{2s-1}-(1-t)\abs{1-t}^{2s-2}\Big)\\
		&= 2(2s+1)\Bigg(\Big(\abs{1+t}^{2s}-\abs{1-t}^{2s}\Big)^2 +8st^2(1-t^2)\abs{1-t^2}^{2s-2}\Bigg).
	\end{align*}
	Now, the claim follows from
	\begin{align*}
		&\int\limits_0^\infty \frac{\Big(\abs{1+t}^{2s}-\abs{1-t}^{2s}\Big)^2 +8st^2(1-t^2)\abs{1-t^2}^{2s-2}}{t^{2+4s}}\d t \\
		&\quad= \Bigg(	\int\limits_0^1+	\int\limits_1^\infty\Bigg) \frac{\Big(\abs{1+t}^{2s}-\abs{1-t}^{2s}\Big)^2 +8st^2(1-t^2)\abs{1-t^2}^{2s-2}}{t^{2+4s}}\d t\\
		&\quad=\int\limits_0^1\frac{\Big(\abs{1+t}^{2s}-\abs{1-t}^{2s}\Big)^2(1+t^{2+4s}) +(1-t^{2+4s})8st^2(1-t^2)\abs{1-t^2}^{2s-2}}{t^{2+4s}}\d t.
	\end{align*}
	Here, we used the change of variables $t\mapsto 1/t$ in the integral over $(1,\infty)$. Note that the integrand is strictly positive outside of zero. \smallskip
	
	\textbf{Case $1/4<s<3/4$.} Since the normalization constant of the fractional Laplacian $(-\Delta)^\alpha$ is negative in this case, it remains to prove that 
	\begin{equation*}
		\II{}:=\int\limits_{0}^\infty \frac{1+[\partial_{\xi}\vphi_s^2](0,1)\frac{t^2}{2}-\vphi_s(t,1)^2}{t^{2+4s}}\d t
	\end{equation*}
	is strictly negative. We will use the same calculations as in the proof in the case $0<s<1/4$. It is easy to see that $\abs{\cdot}^{4s}$ still is the fundamental solution to the operator $(-\Delta)^{\alpha}$. This yields
	\begin{align*}
		\II{}&= c_1\int\limits_0^1 \Bigg(\big(\abs{1+t}^{2s}-\abs{1-t}^{2s}\big)^2(1+t^{4s+2})+8st^2(1-t^{2+4s})(1-t)\abs{1-t}^{2s-2}(1+t)^{2s-1}\\
		&\qquad\qquad -8s(1+2s)(t^2+t^{8s})\Bigg) \frac{\d t}{t^{2+4s}}=: c_1\int\limits_0^1 \III{}\frac{\d t}{t^{2+4s}}
	\end{align*}
	for some positive constant $c_1$. It suffices to show that the integrand is nonpositive and nontrivial. For this, notice that for any $t\in (0,1)$
	\begin{equation*}
		\abs{1+t}^{2s}-\abs{1-t}^{2s}\le 4s t + (2^{2s}-4s)_+ t^2. 
	\end{equation*}
	We split the proof into two further cases. If $s\ge 1/2$, then using the previous bound yields
	\begin{align*}
		\III{} &\le 4^2s^2 t^2(1+t^{4s+2})+8s t^2(1-t^{2+4s})-8s(1+2s)(1+t^{8s-2})t^2\\
		&= -8st^2 \Big( 2s(t^{8s-2}-t^{4s+2})+t^{2+4s}+t^{8s-2} \Big)<0.
	\end{align*}
	Here, we used $t\in (0,1)$ and $s\in(0,1)$. If $s<1/2$, then using $2^{2s}-4s\le 2(\sqrt{2}-1)(1-2s)$ we find that
	\begin{align*}
		\III{}&\le \big(4s t+(2^{2s}-4s)t^2\big)^2(1+t^{4s+2})+8s t^2(1-t^{2+4s})-8s(1+2s)(1+t^{8s-2})t^2\\
		&\le -(4s)^2t^2\big(t^{8s-2}-t^{4s+2}\big) -(t^2+t^{8s})\Big( 16s^2-16s(\sqrt{2}-1)(1-2s)-4(\sqrt{2}-1)^2(1-2s)^2 \Big)\\
		&\le -(4s)^2t^2\big(t^{8s-2}-t^{4s+2}\big)<0.
	\end{align*}
	It is easy to see that the previous expression is negative for any $1/4<s<1/2$ which proves the claim.	\smallskip
	
	\textbf{Case $3/4<s<1$.} We proceed just as in the previous two cases. Since the normalization constant of the fractional Laplacian $(-\Delta)^\alpha$ is positive in this case, it remains to prove the positivity of 
	\begin{align*}
		\IV{}&:= \int\limits_{0}^\infty \frac{1+[\partial_\xi^2 \vphi_s^2](0,1) \frac{t^2}{2}+[\partial_\xi^4 \vphi_s^2](0,1) \frac{t^4}{4!}- \vphi_s(t,1)^2}{t^{2+4s}}\d t\\
		&= \int\limits_{0}^\infty \frac{1+[\partial_\xi^2 \vphi_s^2](0,1) 2t^2+[\partial_\xi^4 \vphi_s^2](0,1) \frac{2t^4}{3}- \vphi_s(2t,1)^2}{2^{1+4s}t^{2+4s}}\d t.
	\end{align*}
	Multiplying both the numerator and the denominator by $t^2$ and using integration by parts twice reveals
	\begin{align*}
		&(2^{3+4s}(3+4s)(2+4s)(1+2s)^2)\IV{}\\
		&= \int\limits_{0}^\infty \Bigg(8(2s+1)^2+2^5\, 3(2s+1)^2[\partial_\xi^2 \vphi_s^2](0,1) t^2
		+  2^4\,5(2s+1)^2[\partial_\xi^4 \vphi_s^2](0,1)  t^4\\
		&\qquad -\partial_t^2\Big( (1+t)^{1+2s}-(1-t/2)\abs{1-t}^{2s}  \Big)^2\Bigg) \frac{\d t}{t^{2+4s}}.
	\end{align*}
	Note that $\abs{\cdot}^{4s}$ is the fundamental solution to $(-\Delta)^\alpha$. This allows us, just as in the case $s<1/4$, to rewrite the above term as follows:
	\begin{align*}
		&\\
		&= \int\limits_{0}^\infty \Bigg(\big(\abs{1+t}^{2s}-\abs{1-t}^{2s}\big)^2+8st^2(1-t)(1+t)\abs{1-t}^{2s-2}\abs{1+t}^{2s-2} \\
		&\qquad-2s(2s+1)t^2+2(2s+1)2s(2s-1)\frac{16s^2-26s+15}{3}t^4\Bigg) \frac{\d t}{t^{2+4s}}.
	\end{align*}
	Now, we split the integral at $t=1$ and use the change of variables $t\mapsto 1/t$ on $(1,\infty)$. Thereafter, $(2^{2+4s}(3+4s)(2+4s)(1+2s))\IV{}$ equals
	\begin{align*}
		&\int\limits_{0}^1 \Bigg(\big(\abs{1+t}^{2s}-\abs{1-t}^{2s}\big)^2(1+t^{2+4s})+8st^2(1-t^{4s})(1-t)(1+t)\abs{1-t}^{2s-2}\abs{1+t}^{2s-2} \\
		&\qquad-2s(2s+1)(t^2+t^{8s})+2(2s+1)2s(2s-1)\frac{16s^2-26s+15}{3}(t^4+t^{8s-2})\Bigg) \frac{\d t}{t^{2+4s}}=:\int\limits_{0}^1 \V{}\frac{\d t}{t^{2+4s}}.
	\end{align*}
	Surely, the second term in $\V{}$ is positive. Furthermore, the last term is bigger than $2s(2s+1)t^{8s}$ for all $t\in(0,1)$ and $s\in (3/4,1)$. Lastly, we notice that 
	\begin{equation*}
		\abs{1+t}^{2s}-\abs{1-t}^{2s}\ge 4st>\sqrt{2s(2s+1)}t
	\end{equation*}
	which proves that $\V{}$ is positive for all $t\in (0,1)$. 
\end{proof}

I learned the proof of the next lemma from Mateusz Kwa{\'s}nicki. 
\begin{lemma}\label{lem:fasymp_d1_s_x_s_ge_12_positivity}
	For any $s\in (0,1)\setminus \{1/4,3/4\}$ and any $\kappa\in (0,2)$, the constant $C_{s,3}(\kappa)$ from \autoref{prop:fasymp_d1_s_x_s_ge_12} is positive.
\end{lemma}
\begin{proof}
	By the definition of $C_{s,3}(\kappa)$, it suffices to prove the positivity of 
	\begin{equation*}
		4(2s+1)^2(-\Delta_{\kappa})^\alpha\vphi_{s}^2(\kappa,1).
	\end{equation*}
	We write
	\begin{align*}
		4(2s+1)^2\vphi_s^2(2t,1)&=\frac{\abs{1+t}^{2\alpha+1}+\abs{1-t}^{2\alpha+1}-2(1-t^2)\abs{1-t^2}^{\alpha-1/2}}{t^2} =f_1(t)+f_2(t), \\
		\intertext{where}
		f_1(t)&:=\frac{\abs{1+t}^{2\alpha+1}+\abs{1-t}^{2\alpha+1}-2}{t^2},\quad
		f_2(t):= 2\frac{1-(1-t^2)\abs{1-t^2}^{\alpha-1/2}}{t^2}.
	\end{align*}
	Now, we proceed in two steps. \medskip
	
	\textit{Step 1.} Firstly, we prove that $(-\Delta)^\alpha f_1(t)=0$ for all $t\in (-1,1)$. \smallskip
	
	We rewrite $f_1$ as
	\begin{equation*}
		f_1(t)= 2\alpha(2\alpha-1)\int\limits_{0}^1 \int\limits_{-1}^{1} r\abs{1-rth}^{2\alpha-1}\d h\d r.
	\end{equation*}
	Recall that $\abs{\cdot}^{2\alpha-1}$ is the fundamental solution to $(-\Delta)^{\alpha}$, see e.g.\ \cite{Lan72} in particular equation (1.1.1) and section 6 therein. Next, we prove $(-\Delta)^\alpha f_1=0$ in $C_c^\infty((-1,1))'$. Let $\psi\in C_c^\infty((-1,1))$. An application of Fubini's theorem yields
	\begin{equation*}
		\text{(I)}:=\int\limits_{\R} f_1(x) (-\Delta)^\alpha \psi(x)\d x= 2\alpha(2\alpha-1) \int\limits_{0}^1 \int\limits_{-1}^1 r \int\limits_{\R} |1-rth|^{2\alpha-1} (-\Delta)^\alpha \psi(t)\d t \d h \d r.
	\end{equation*}
	Note that $|\cdot|^{2\alpha-1}$ is the fundamental solution to $(-\Delta)^{\alpha}$ in $\mathbb{R}$. Thus, $\text{(I)}$ equals
	\begin{equation*}
		2\alpha(2\alpha-1) \int\limits_{0}^1 \int\limits_{-1}^1 r  |rt|^{2\alpha}\psi(\frac{1}{rh}) \d h \d r.
	\end{equation*}
	This equals zero since $\psi$ is supported in $(-1,1)$. This proves $(-\Delta)^\alpha f_1=0$ in $C_c^\infty((-1,1))'$. 
	
	Finally, due to $f_1\in L^1(\R; (1+|\cdot|)^{-1-2\alpha})$ and the regularity of $f_1$ in $(-1,1)$, $(-\Delta)^\alpha f_1(t)=0$ holds for every $t\in (-1,1)$. \medskip
	
	\textit{Step 2.} In the second step, we consider the remaining term $f_2$ and prove that $(-\Delta)^\alpha f_2>0$ in $(0,1)$. \smallskip
	
	This will be a direct consequence of the following claim. \smallskip
	
	\textit{Claim.} For any $t\in (-1,1)$ we have
	\begin{equation}
		(-\Delta)^\alpha f_2(t)=2^{\alpha+1} \frac{\Gamma(3/2+\alpha)\Gamma(1/2+\alpha)}{\pi(1+\alpha)}\big(1- \sin(\pi \alpha) \big)\pfqnr{2}{1}{ 1/2+\alpha & 1+\alpha}{2+\alpha}{t^2}.
	\end{equation}
	
	In order to prove the claim, we rewrite $f_2$ as follows 
	\begin{equation*}
		\frac{1}{2}f_2(t)= (\alpha+1/2)t^{-2}\int\limits_{0}^{t^2}\abs{1-r}^{\alpha-1/2}\d r=(\alpha+1/2)\int\limits_{0}^{1}(1-rt^2)_+^{\alpha-1/2}+(rt^2-1)_+^{\alpha-1/2}\d r.
	\end{equation*}
	Now, we apply \cite[Corollary 3]{DKK17} which yields
	\begin{align*}
		(-\Delta_t)^\alpha (1-rt^2)_+^{\alpha-1/2} &= r^\alpha (-\Delta)^\alpha[(1-(\cdot)^2)_+^{\alpha-1/2}](\sqrt{r}\,t)\\
		&= (2r)^{\alpha}\Gamma(1/2+\alpha)\meijer{2,1}{3,3}{1/2-\alpha & 1/2& -\alpha}{0 & -\alpha & 1/2}{r\,t^2}
	\end{align*}
	and 
	\begin{align*}
		(-\Delta_t)^{\alpha}(rt^2-1)_+^{\alpha-1/2}&= r^\alpha (-\Delta)^{\alpha}[((\cdot)^2-1)_+^{\alpha-1/2}](\sqrt{r}\,t)\\
		&= (2r)^\alpha \Gamma(1/2+\alpha) \meijer{1,2}{3,3}{1/2-\alpha & 1/2 & -\alpha}{0 & -\alpha & 1/2}{r\,t^2}.
	\end{align*}
	These calculations from \cite{DKK17} together with the identities \cite[5.4 (1)-(2)]{Luk69i} yield
	\begin{align*}
		(-\Delta_t)^\alpha (1-rt^2)_+^{\alpha-1/2} &=(2r)^{\alpha}\Gamma(1/2+\alpha)\meijer{1,1}{2,2}{1/2-\alpha &  1/2}{0 & 1/2}{r\,t^2}\\
		\intertext{and}
			(-\Delta_t)^{\alpha}(rt^2-1)_+^{\alpha-1/2}&=(2r)^\alpha \Gamma(1/2+\alpha) \meijer{1,1}{2,2}{1/2-\alpha & -\alpha}{0 & -\alpha }{r\,t^2}.
	\end{align*}
	Together with the representation of $f_2$, we find that $(-\Delta)^\alpha f_2(t)$ equals
	\begin{align*}
		&(\alpha+1/2)2^{\alpha+1} \Gamma(1/2+\alpha)\int\limits_{0}^{1}r^{\alpha}\big(\meijer{1,1}{2,2}{1/2-\alpha &  1/2}{0 & 1/2}{r\,t^2}+ \meijer{1,1}{2,2}{1/2-\alpha & -\alpha}{0 & -\alpha }{r\,t^2}\big)\d r\\
		&\quad=2^{\alpha+1} \Gamma(3/2+\alpha)\Big(\meijer{1,2}{3,3}{-\alpha & 1/2-\alpha &  1/2}{0 & 1/2 & -1-\alpha}{t^2}+ \meijer{1,2}{3,3}{-\alpha & 1/2-\alpha & -\alpha}{0 & -\alpha & -1-\alpha}{t^2}\Big).
	\end{align*}
	Here, we used \cite[5.4 (4)]{Luk69i} and \cite[5.6 (10)]{Luk69i}. We may simplify this expression further. By \cite[5.4 (1)]{Luk69i} and the above calculations, the term $(-\Delta)^\alpha f_2(t)$ equals
	\begin{align*}
		2^{\alpha+1} \Gamma(3/2+\alpha)\Big(\meijer{1,2}{3,3}{-\alpha & 1/2-\alpha &  1/2}{0 & 1/2 & -1-\alpha}{t^2}+ \meijer{1,1}{2,2}{ 1/2-\alpha & -\alpha}{0 &  -1-\alpha}{t^2}\Big).
	\end{align*}
	We rewrite the Meijer G-function in terms of generalized hypergeometric functions, see e.g.\ \cite[6.5 (1)]{Luk69i} or \cite{PBM90}. This yields
	\begin{align*}
		(-\Delta)^\alpha f_2(t)&= 2^{\alpha+1} \Gamma(3/2+\alpha)\Big(\frac{\Gamma(1+\alpha)\Gamma(1/2+\alpha)}{\Gamma(1/2)\Gamma(1/2)\Gamma(2+\alpha)}  \pfqnr{3}{2}{1+\alpha & 1/2+\alpha &  1/2}{1/2 & 2+\alpha}{t^2}\\
		&\qquad + \frac{\Gamma(1/2+\alpha)}{\Gamma(-\alpha)\Gamma(2+\alpha)} \pfqnr{2}{1}{ 1/2+\alpha & 1+\alpha}{2+\alpha}{t^2}\Big)\\
		&= 2^{\alpha+1} \frac{\Gamma(3/2+\alpha)\Gamma(1/2+\alpha)}{\Gamma(2+\alpha)}\Big(\frac{\Gamma(1+\alpha)}{\Gamma(1/2)\Gamma(1/2)} + \frac{1}{\Gamma(-\alpha)} \Big)\pfqnr{2}{1}{ 1/2+\alpha & 1+\alpha}{2+\alpha}{t^2}\\
		&= 2^{\alpha+1} \frac{\Gamma(3/2+\alpha)\Gamma(1/2+\alpha)}{\pi(1+\alpha)}\big(1- \sin(\pi \alpha) \big)\pfqnr{2}{1}{ 1/2+\alpha & 1+\alpha}{2+\alpha}{t^2}.
	\end{align*}
	Here, we used the definition of the generalized hypergeometric function to reduce its order and the properties of the gamma function. This proves the claim. \smallskip
	
	The positivity of the constant follows from the claim in step 2. More precisely, the factor 
	\begin{equation*}
		2^{\alpha+1} \frac{\Gamma(3/2+\alpha)\Gamma(1/2+\alpha)}{\pi(1+\alpha)}\big(1- \sin(\pi \alpha) \big)
	\end{equation*} 
	is obviously positive. Additionally, using the definition of the hypergeometric function in terms of a power series, i.e.\ 
	\begin{equation*}
		\pfqnr{2}{1}{ 1/2+\alpha & 1+\alpha}{2+\alpha}{t^2}= \sum\limits_{k=0}^\infty \frac{(1/2+\alpha)_k\,(1+\alpha)_k}{(2+\alpha)_k} \frac{t^{2k}}{k!},
	\end{equation*}
	where $(q)_k$ is the rising Pocherhammer symbol, it is clear that $\pfqnr{2}{1}{ 1/2+\alpha & 1+\alpha}{2+\alpha}{t^2}$ is positive since all terms in the sum are positive. 	
\end{proof}

\section{Proof of the main theorem}\label{sec:proof_main_th}

\begin{proof}[Proof of \autoref{th:bounds_time_1}]
	We define the auxiliary function 
	\begin{equation*}
		g_s(x,v):= p_s(1,x+\frac{v}{2},v)\,\big(1+\abs{x}+\abs{v}\big)^{2+2s}\Big( 1+(2\abs{x}-\abs{v})_+\Big)^{2s}.
	\end{equation*}  
	It remains to prove the existence of a constant $C\ge 1$ such that $C^{-1}\le g_s\le C$ in $\R\times \R$. We divide the proof into two steps. In the first step, we prove this bound outside of a compact set. \smallskip
	
	\textbf{Claim A:} There exists a ball $B\subset \R \times \R$ and a constant $c_1\ge 1$ such that $c_1^{-1}\le g_s \le c_1$ on $B^c$. \smallskip
	
	We prove this claim via contradiction. Assume the contrary, i.e.\ for any natural number $n$ there exists a vector $(x_n, v_n)\in \R^2$ such that $x_n^2+v_n^2\ge n^2$ and 
	\begin{equation*}
		g_s(x_n,v_n)\le \frac{1}{n} \text{ or } g_s(x_n,v_n)\ge n.
	\end{equation*}
	Now, we define two new sequences 
	\begin{equation*}
		a_n:= \frac{\abs{x_n}}{\abs{v_n}},\quad b_n:= \frac{\abs{v_n}}{2}-\abs{x_n}.
	\end{equation*}
	Since the space of the extended real numbers $\overline{\R}$ is compact, we find a subsequence $\{n_k\mid k\in \N\}$ such that both $a_{n_k}$ and $b_{n_k}$ converge in $\overline{\R}$. By \autoref{prop:fasymp_d1_s_x/v_le12}, \autoref{prop:fasymp_d1_s_x_s_ge_12}, \autoref{lem:fasymp_d1_s_x_s_ge_12_positivity_s_14}, \autoref{lem:fasymp_d1_s_x_s_ge_12_positivity_s_34}, \autoref{lem:fasymp_d1_s_x_s_ge_12_boundedness}, \autoref{lem:fasymp_d1_s_x_s_ge_12_positivity_kappa_0}, and  \autoref{lem:fasymp_d1_s_x_s_ge_12_positivity}  we find a constant $c_2\ge 1$ such that 
	\begin{align*}
		\frac{1}{c_2} \le \liminf\limits_{k\to \infty} g_s(x_{n_k}, v_{n_k}), \quad \limsup_{k\to \infty} g_s(x_{n_k}, v_{n_k})\le c_2.
	\end{align*}
	Picking $k$ sufficiently big such that $n_k> c_2$ yields a contradiction. Thus, the claim follows with $c_1= 2c_2$. \medskip
	
	Now, it remains to prove the bound on the ball $B$. Note that it suffices to prove that $p_s(1,x,v)$ is uniformly positive and bounded on a twice as large ball $2B$. Here, the boundedness follows directly from $p_s(1,x,v)\le p_s(1,0,0)<\infty$. This brings us to the lower bound.\medskip
	
	\textbf{Claim B:} There exists a constant $c_3>0$ such that $p_s(1,x,v)\ge c_3$ on $2\overline{B}$. \smallskip
	
	By \autoref{lem:basic_properties}, we know that $p_s(1,\cdot, \cdot)$ is continuous and nonnegative. To prove the claim, we assume the contrary. Assume that there exists $(x_0,v_0)\in 2B$ such that $p_s(1,x_0,v_0)=0$. From \eqref{eq:scaling} and \eqref{eq:frac_kol}, we can derive an equation for the function $p_s(1,x,v)$. This reads
	\begin{equation}\label{eq:time_id_equation}
		-(1+\frac{1}{s}) p_s(1,x,v)-\frac{1}{2s} v\cdot \nabla_v p_s(1,x,v)-\big((1+\frac{1}{2s})x-v\big)\cdot \nabla_x p_s(1,x,v)+(-\Delta_v)^s p_s(1,x,v)=0.
	\end{equation}
	Since $(x_0,v_0)\in 2B$ is a global and local minimum, the gradient of $p_s(1,\cdot,\cdot)$ with respect to both space and velocity vanishes at $(x_0,v_0)$. Thus, the equation \eqref{eq:time_id_equation} at $(x,v)=(x_0,v_0)$ reads
	\begin{equation*}
		(-\Delta_v)^s p_s(1,x_0,v_0)=0.
	\end{equation*} 
	By claim A we know that $p_s(1,x_0,w)$ is strictly positive for any $w\in \R$ satisfying $(x_0,w)\notin 2B$. By the nonnegativity of the fundamental solution, we find a contradiction in 
	\begin{align*}
			0&=(-\Delta_v)^s p_s(1,x_0,v_0)= c_{1,s}\pv \int\limits_{\R} \frac{p_s(1,x_0,v_0)-p_s(1,x_0,w)}{\abs{v_0-w}^{1+2s}}\d w \\
			&\le c_{1,s} \int\limits_{\R}\1_{2B^c}((x_0,w)) \frac{0-p_s(1,x_0,w)}{\abs{v_0-w}^{1+2s}}\d w<0.
	\end{align*}
	Thus, the claim follows. 
\end{proof}

\section{An explicit representation}\label{sec:half_laplace}
The fundamental solution to the problem \eqref{eq:frac_kol} in the case $s=1/2$ has an explicit representation in terms of elementary functions. In this section, we provide this representation, see \autoref{th:fsol_d1_s_1/2}, and its proof.
\begin{theorem}\label{th:fsol_d1_s_1/2}
	For any $x,v\in \R$, the function $\frac{\pi^2}{2} k_{1/2}(x,v)= \frac{\pi^2}{2} p_{1/2}(1,x+v/2,v)$ equals
	\begin{align*}
		&\frac{2}{4^2x^2+(1+v^2)^2} +\frac{ 4 \abs{x} \sqrt{\sqrt{ (1+v^2)^2+4^2x^2 } +1+v^2 } +(1+v^2)\sqrt{ \sqrt{ (1+v^2)^2+4^2x^2 }-1-v^2  }  }{\sqrt{2} \cdot 4 \sqrt{(1+v^2)^2+4^2x^2}^3}\\
		&\quad \times\ln \Bigg(1- \frac{4\sqrt{2}x}{2\sqrt{2}x+ \frac{  4x^2+\sqrt{(1+v^2)^2+4^2x^2} }{\sqrt{ \sqrt{(1+v^2)^2+4^2x^2}+1+v^2 }   }} \Bigg)\\
		&\, +\frac{  (1+v^2)\sqrt{\sqrt{(1+v^2)^2+4^2x^2}+1+v^2} -4\abs{x}\sqrt{\sqrt{(1+v^2)^2+4^2x^2}-1-v^2} }{2\sqrt{2}\sqrt{(1+v^2)^2+4^2x^2}^3} \\
		&\quad\times\Bigg[\arctan\Bigg( \frac{ \sqrt{2}-\sqrt{\sqrt{(1+v^2)^2+4^2x^2}-1-v^2}  }{\sqrt{\sqrt{(1+v^2)^2+4^2x^2}+1+v^2}+\sqrt{2}\,\abs{v} }\Bigg)\\
		&\quad\qquad+\arctan\Bigg( \frac{ \sqrt{2}+\sqrt{\sqrt{(1+v^2)^2+4^2x^2}-1-v^2}  }{\sqrt{\sqrt{(1+v^2)^2+4^2x^2}+1+v^2}+\sqrt{2}\,\abs{v} }\Bigg)\\
		&\quad\quad\qquad + \arctan\Bigg( \frac{\sqrt{2} + \sqrt{\sqrt{(1+v^2)^2+4^2x^2}-1-v^2}}{ \sqrt{\sqrt{(1+v^2)^2+4^2x^2}+1+v^2} -\sqrt{2}\,\abs{v}}\Bigg)\\
		&\quad\qquad\qquad+\arctan\Bigg( \frac{\sqrt{2}- \sqrt{\sqrt{(1+v^2)^2+4^2x^2}-1-v^2} }{ \sqrt{\sqrt{(1+v^2)^2+4^2x^2}+1+v^2} -\sqrt{2}\,\abs{v}}\Bigg)\Bigg].
	\end{align*}
\end{theorem}
In the special case $v=0$, a representation was already provided in \cite[Appendix C]{ScMa10}.
\begin{lemma}\label{prop:fsol_d_1_s_1/2}
	For any $x,v\in \R$ we have
	\begin{align*}
		k_{1/2}(x,v)&=  \frac{2}{\pi^2}  \frac{1+(2\abs{x})^2-\abs{v}^2}{\big(1+(2\abs{x}+\abs{v})^2\big)\big(1+(2\abs{x}-\abs{v})^2\big) }  + \frac{4}{\pi^2}\int\limits_{-1}^{1}  \frac{\big(w^2+1\big)^2-4(2\abs{x}+\abs{v}w)^2}{\big((w^2+1)^2+4(2\abs{x}+\abs{v}w)^2\big)^2}\d w.
	\end{align*}
\end{lemma}
\begin{proof}
	Recall from \eqref{eq:k_s} and the symmetries of $\vphi_{1/2}$ that $k_{1/2}$ is given by 
	\begin{align*}
		k_{1/2}(x,v)&= \frac{1}{\pi^{2}} \int\limits_{0}^\infty \int\limits_{0}^\infty \cos (x \xi) \cos (v \nu)e^{-\vphi_{1/2}(\xi,\nu)}  \d \nu \d \xi.
	\end{align*}
	Let $\xi, \nu > 0$. Note that 
	\begin{align*}
		\vphi_{1/2}(\xi, \nu)=\begin{cases}
			\nu &, \text{ if } 2\nu\ge \xi\\
			\frac{\nu^2}{\xi}+ \frac{\xi}{4}&, \text{ else}
		\end{cases}.
	\end{align*}
	We split $k_{1/2}(x,v)$ into two terms $\I{}+\II{}$, where
	\begin{align*}
		\I{}:= \frac{1}{\pi^{2}} \int\limits_{0}^\infty \int\limits_{\xi/2}^{\infty} \cos (\abs{x} \xi) \cos (\abs{v} \nu)e^{-\nu}  \d \nu \d \xi ,
		\II{}:= \frac{1}{\pi^{2}} \int\limits_{0}^\infty \int\limits_{0}^{\xi/2} \cos (\abs{x} \xi) \cos (\abs{v} \nu)e^{- \frac{\nu^2}{\xi}- \frac{\xi}{4}}  \d \nu \d \xi.
	\end{align*}
	Since $\int_{0}^{2\nu } \cos(\abs{x}\xi)= \frac{1}{\abs{x}}  \sin(2\abs{x}\nu)$, the trigonometric identities yield 
	\begin{align*}
		\pi^2(\text{I})= \frac{1}{\abs{x}}\int\limits_{0}^\infty \frac{\sin\big( (2\abs{x}+\abs{v})\nu \big)+\sin\big( (2\abs{x}-\abs{v})\nu \big)   }{2}e^{-\nu}  \d \nu.
	\end{align*}
	Additionally, we know that $\int_{0}^\infty \sin(at)e^{-t}\d t = a/(1+a^2)$ for any $a\in \R$. This is easily seen by integrating $A= \int_{0}^\infty \sin(ax)e^{-x}\d x$ partially twice and solving for $A$. Thus, we know that 
	\begin{align*}
		\pi^2(\text{I})=\frac{1}{2\abs{x}}\Bigg(  \frac{(2\abs{x}+\abs{v})}{1+(2\abs{x}+\abs{v})^2} + \frac{(2\abs{x}-\abs{v})}{1+(2\abs{x}-\abs{v})^2} \Bigg).
	\end{align*}
	Now, we turn our attention to the term $(\text{II})$. The change of variables $\nu = \xi w/2$ yields
	\begin{align*}
		\pi^2 \, (\text{II})&=\int\limits_{0}^\infty \frac{\abs{\xi}}{2}\int\limits_{0}^{1} \cos (\abs{x} \xi) \cos (\abs{v}\xi/2 w)e^{- \xi\big(\frac{w^2- 1}{4} \big) }  \d w \d \xi,\\
		\intertext{which equals, after applying trigonometric identities,}
		&= \frac{1}{4}\int\limits_{0}^{1}\int\limits_{0}^\infty \abs{\xi} \Big( \cos\big( (\abs{x}+\abs{v}w/2)\xi \big)+\cos\big( (\abs{x}-\abs{v}w/2)\xi \big)  \Big)e^{- \xi\big(\frac{w^2+1}{4} \big) }  \d \xi \d w.
	\end{align*}
	Now, let's fix $w\in (0,1)$ momentarily. We define
	\begin{align*}
		(\text{III})&:=\int\limits_{0}^\infty \abs{\xi} \Big( \cos\big( (\abs{x}+\abs{v}w/2)\xi \big)+\cos\big( (\abs{x}-\abs{v}w/2)\xi \big)  \Big) e^{- \xi\big( \frac{w^2+1}{4} \big) }  \d \xi\\
		&= \frac{1}{2}(-\Delta_y)^{1/2}\Big[ \sF_{\xi} \Big(e^{-\abs{\xi}\big( \frac{w^2+1}{4} \big)}\Big)(y) \Big]\mid_{y= (\abs{x}+\abs{v}w/2)}\\
		&\qquad+\frac{1}{2}(-\Delta_y)^{1/2}\Big[ \sF_{\xi} \Big(e^{-\abs{\xi}\big(\frac{w^2+1}{4}\big)}\Big)(y) \Big]\mid_{y= (\abs{x}-\abs{v}w/2)}.
	\end{align*}
	But, we know that 
	\begin{align*}
		\sF_{\xi} e^{-a\abs{\xi}}(y) = \frac{2a}{a^2+y^2} = 2\pi \tilde{p}_{1/2}(a,y)
	\end{align*}
	for any $a>0$, which is up to a multiplicative constant the solution to the fractional heat equation in dimension $d=1$, the order $s=1/2$, and at time $a$. Using this fact, we find that 
	\begin{align*}
		(-\Delta_y)^{1/2}\sF_{\xi} e^{-a\abs{\xi}}(y) &= 2\pi (-\Delta_y)^{1/2}p_{1/2}(a,y)= - 2\pi \partial_a \tilde{p}_{1/2}(a,y)\\
		&=-\Big( \frac{2}{a^2+y^2}  - \frac{(2a)^2}{(a^2+y^2)^2} \Big)= 2\frac{a^2-y^2}{(a^2+y^2)^2}.
	\end{align*}
	We plug this into the term $(\text{III})$ to find that
	\begin{align*}
		(\text{III}) &= \Bigg(  \frac{\big(\frac{w^2+1}{4}\big)^2-(\abs{x}+\abs{v}w/2)^2}{(\big(\frac{w^2+1}{4}\big)^2+(\abs{x}+\abs{v}w/2)^2)^2}+\frac{\big(\frac{w^2+1}{4}\big)^2-(\abs{x}-\abs{v}w/2)^2}{(\big(\frac{w^2+1}{4}\big)^2+(\abs{x}-\abs{v}w/2)^2)^2} \Bigg)\\
		&= 4^{2}\Bigg(  \frac{\big(w^2+1\big)^2-4(2\abs{x}+\abs{v}w)^2}{\Big(\big(w^2+1\big)^2+4(2\abs{x}+\abs{v}w)^2\Big)^2}+\frac{\big(w^2+1\big)^2-4(2\abs{x}-\abs{v}w)^2}{\Big(\big(w^2+1\big)^2+4(2\abs{x}-\abs{v}w)^2\Big)^2} \Bigg).
	\end{align*}
\end{proof}
To prove \autoref{th:fsol_d1_s_1/2}, it remains to calculate the integral in the representation given in \autoref{prop:fsol_d_1_s_1/2}. We provide this lengthy calculation in the appendix. 

\begin{proof}[{Proof of \autoref{th:fsol_d1_s_1/2}}]
	Combine \autoref{prop:fsol_d_1_s_1/2} and \autoref{lem:special_integral_d_1_s_1/2}.
\end{proof}

\appendix
\section{}
\begin{lemma}\label{lem:special_integral_d_1_s_1/2}
	Let $a,b\in \R_+$ and define the function 
	\begin{align*}
		F_{a,b}(x):= 2\frac{(x^2+1)^2-4(a+bx)^2}{((x^2+1)^2+4(a+bx)^2)^2}   .
	\end{align*}
	Then $\int_{-1}^{1} F_{a,b}(x)\d x$ equals
	\begin{align*}
		&\frac{1}{(1+b^2)^2+4a^2} \frac{(b^2+1-a^2)(1+b^2)^2-2a^4}{ a^4 - 2a^2(b^2-1) + (1 + b^2)^2 }\\
		&\quad+\frac{ 2 a \sqrt{\sqrt{ (1+b^2)^2+4a^2 } +1+b^2 } +(1+b^2)\sqrt{ \sqrt{ (1+b^2)^2+4a^2 }-1-b^2  }  }{\sqrt{2} \cdot 4 \sqrt{(1+b^2)^2+4a^2}^3}\\
		&\quad\quad \times\ln \Bigg(1- \frac{2\sqrt{2}}{\sqrt{2}+ \frac{  a^2+\sqrt{(1+b^2)^2+4a^2} }{a\sqrt{ \sqrt{(1+b^2)^2+4a^2}+1+b^2 }   }} \Bigg)\\
		&\quad+\frac{  (1+b^2)\sqrt{\sqrt{(1+b^2)^2+4a^2}+1+b^2} -2\abs{a}\sqrt{\sqrt{(1+b^2)^2+4a^2}-1-b^2} }{2\sqrt{2}\sqrt{(1+b^2)^2+4a^2}^3} \\
		&\quad\quad\times\Bigg[\arctan\Bigg( \frac{ \sqrt{2}-\sqrt{\sqrt{(1+b^2)^2+4a^2}-1-b^2}  }{\sqrt{2}\,b+ \sqrt{\sqrt{(1+b^2)^2+4a^2}+1+b^2} }\Bigg)\\
		&\quad\quad\qquad+\arctan\Bigg( \frac{ \sqrt{2}+\sqrt{\sqrt{(1+b^2)^2+4a^2}-1-b^2}  }{\sqrt{2}\,b+ \sqrt{\sqrt{(1+b^2)^2+4a^2}+1+b^2} }\Bigg)\\
		&\quad\qquad\qquad + \arctan\Bigg( \frac{\sqrt{2} + \sqrt{\sqrt{(1+b^2)^2+4a^2}-1-b^2}}{ \sqrt{\sqrt{(1+b^2)^2+4a^2}+1+b^2} -\sqrt{2}\,b}\Bigg)\\
		&\quad\quad\qquad\qquad+\arctan\Bigg( \frac{\sqrt{2}- \sqrt{\sqrt{(1+b^2)^2+4a^2}-1-b^2} }{ \sqrt{\sqrt{(1+b^2)^2+4a^2}+1+b^2} -\sqrt{2}\,b}\Bigg)\Bigg].
	\end{align*}
\end{lemma}
\begin{proof}
	We decompose $F_{a,b}$ using partial fractions:
	\begin{align*}
		F_{a,b}(x)= \frac{1}{(x^2+1+2i(a+bx))^2}+\frac{1}{(x^2+1-2i(a+bx))^2}.
	\end{align*}
	To decompose this further, we define
	\begin{align*}
		f_{\pm}(x)&:= \frac{1}{(x^2+1\pm 2i(a+bx))^2},\\
		c:=c(a,b)&:= \sqrt{ 1+b^2+2ai  }.
	\end{align*}
	Note that the denominator of $f_{\pm}$ is zero whenever $x= \mp ib+i c(\pm a, b)$ or $x= \mp ib-i c(\pm a, b)$. Note that 
	\begin{align*}
		c(a,b)=  \frac{1}{\sqrt{2}} \Bigg(\sqrt{\sqrt{ (1+b^2)^2 +4a^2 } + 1+b^2} + i \, \sgn(a) \sqrt{\sqrt{ (1+b^2)^2+4a^2} -1-b^2 } \Bigg).
	\end{align*}
	Thereby, $\overline{c(a,b)}= c(-a,b)$ and $c(a,-b)=c(a,b)$. We write $c:=c(a,b)$ and
	\begin{align*}
		f_+(x)&= \frac{1}{(x+bi+ic)^2}\frac{1}{(x+bi-ic)^2}\\
		&= \frac{i}{4c^3 (x+bi+ci)}- \frac{i}{4c^3 (x+bi-ci)}- \frac{1}{4c^2( x+bi+ci )^2}- \frac{1}{4c^2 (x+bi-ci)^2}\\
		&=: f_+^{(1)}(x)+f_+^{(2)}(x)+f_+^{(3)}(x)+f_+^{(4)}(x)\\
		\intertext{and respectively}
		f_{-}(x) &= \frac{i}{4\overline{c}^3 (x-bi+\overline{c}i)}- \frac{i}{4\overline{c}^3 (x-bi-\overline{c}i)}- \frac{1}{4\overline{c}^2( x-bi+\overline{c}i )^2}- \frac{1}{4\overline{c}^2 (x-bi-\overline{c}i)^2}\\
		&=: f_-^{(2)}(x)+f_-^{(1)}(x)+f_-^{(4)}(x)+f_-^{(3)}(x).
	\end{align*}
	Note that $\overline{f_-^{(j)}}=f_{+}^{(j)}$ and moreover
	\begin{align*}
		g_1(x):=f_{+}^{(1)}(x)+ f_{-}^{(1)}(x)&= \frac{1}{2\abs{c}^6} \frac{ \Im(c^3)\big( x-\Im(c) \big)+\Re(c^3)\big( b+\Re(c) \big)  }{(x-\Im(c))^2+(b+\Re(c))^2}\\
		g_2(x):=f_{+}^{(2)}(x)+ f_{-}^{(2)}(x)&=-\frac{1}{2\abs{c}^6} \frac{ \Im(c^3)\big( x+\Im(c) \big)+\Re(c^3)\big( b-\Re(c) \big)  }{(x+\Im(c))^2+(b-\Re(c))^2}\\
		g_3(x):=f_{+}^{(3)}(x)+f_{-}^{(3)}(x)&= \frac{1}{2\abs{c}^4}\frac{  \Re(c^2)\big( (b+\Re(c))^2-(x-\Im(c))^2 \big) }{\Big( (x-\Im(c))^2+(b+\Re(c))^2  \Big)^2 }\\
		&\qquad +  \frac{1}{\abs{c}^4}\frac{ \Im(c^2)\big( b+\Re(c) \big)\big( x-\Im(c) \big) }{\Big( (x-\Im(c))^2+(b+\Re(c))^2  \Big)^2 },\\
		g_4(x):=f_{+}^{(4)}(x)+f_{-}^{(4)}(x)&= \frac{1}{2\abs{c}^4}\frac{  \Re(c^2)\big( (b-\Re(c))^2-(x+\Im(c))^2 \big) }{\Big( (x+\Im(c))^2+(b-\Re(c))^2  \Big)^2 }\\
		&\qquad +  \frac{1}{\abs{c}^4}\frac{ \Im(c^2)\big( b-\Re(c) \big)\big( x+\Im(c) \big) }{\Big( (x+\Im(c))^2+(b-\Re(c))^2  \Big)^2 }.
	\end{align*}
	Let's reorder these terms in $g_1,\dots, g_4$. We write
	\begin{align*}
		g_1(x)+g_3(x)&= \Bigg( \frac{\Re(c^3)\big(  b+\Re(c) \big)}{2\abs{c}^6} + \frac{\Re(c^2)}{2\abs{c}^4} \Bigg)\frac{1}{(x-\Im(c))^2+(b+\Re(c))^2}\\
		&\quad + \Bigg( \frac{\Im(c^3)}{4\abs{c}^6} \Bigg)\frac{2(x-\Im(c))}{(x-\Im(c))^2+(b+\Re(c))^2}\\
		&\quad + \Bigg( \frac{\Im(c^2)\,\big( b+\Re(c)  \big)}{2\abs{c}^4} \Bigg)\frac{2(x-\Im(c))}{((x-\Im(c))^2+(b+\Re(c))^2)^2}\\
		&\quad - \Bigg( \frac{\Re(c^2)}{2\abs{c}^4} \Bigg)\frac{2(x-\Im(c))^2}{((x-\Im(c))^2+(b+\Re(c))^2)^2}=: h_1(x)+h_2(x)+h_3(x)+h_4(x).
	\end{align*}
	In a similar fashion we reorder the terms $g_2+g_4$ as $\tilde{h}_1(x)+\dots+ \tilde{h}_4(x)$, just replacing $c$ by $-c$ in the previous equation, i.e.\
	\begin{align*}
		\tilde{h}_1(x)&:= \Bigg( -\frac{\Re(c^3)\big(  b-\Re(c) \big)}{2\abs{c}^6} + \frac{\Re(c^2)}{2\abs{c}^4} \Bigg)\frac{1}{(x+\Im(c))^2+(b-\Re(c))^2}\\
		\tilde{h}_2(x)&:=  \Bigg( -\frac{\Im(c^3)}{4\abs{c}^6} \Bigg)\frac{2(x+\Im(c))}{(x+\Im(c))^2+(b-\Re(c))^2}\\
		\tilde{h}_3(x)&:=  \Bigg( \frac{\Im(c^2)\,\big( b-\Re(c)  \big)}{2\abs{c}^4} \Bigg)\frac{2(x+\Im(c))}{((x+\Im(c))^2+(b-\Re(c))^2)^2}\\
		\tilde{h}_4(x)&:=- \Bigg( \frac{\Re(c^2)}{2\abs{c}^4} \Bigg)\frac{2(x+\Im(c))^2}{((x+\Im(c))^2+(b-\Re(c))^2)^2}.
	\end{align*}

	Now, we are in the position to integrate. 
	\begin{align*}
		\int\limits_{-1}^1 h_4(x)\d x= \Bigg( \frac{\Re(c^2)}{2\abs{c}^4} \Bigg)\Bigg(\Big[ \frac{(x-\Im(c))}{(x-\Im(c))^2+(b+\Re(c))^2} \Big]_{x=-1}^{x=1}-\int\limits_{-1}^1 \frac{1}{(x-\Im(c))^2+(b+\Re(c))^2} \d x \Bigg).
	\end{align*}
	Thereby, 
	\begin{align*}
		\int\limits_{-1}^1 h_1(x)+h_4(x)\d x&= \Bigg( \frac{\Re(c^2)}{2\abs{c}^4} \Bigg)\Big[ \frac{(x-\Im(c))}{(x-\Im(c))^2+(b+\Re(c))^2} \Big]_{x=-1}^{x=1} \\
		&\quad +\Bigg(\frac{\Re(c^3)\big(  b+\Re(c) \big)}{2\abs{c}^6} \Bigg)\int\limits_{-1}^1 \frac{1}{(x-\Im(c))^2+(b+\Re(c))^2} \d x\\
		&= \Bigg( \frac{\Re(c^2)}{2\abs{c}^4} \Bigg)\Big[ \frac{(x-\Im(c))}{(x-\Im(c))^2+(b+\Re(c))^2} \Big]_{x=-1}^{x=1} \\
		&\quad +\Bigg(\frac{\Re(c^3)\big(  b+\Re(c) \big)}{2\abs{c}^6} \Bigg)\Big[  \frac{\arctan\big( \frac{x}{b+\Re(c)} \big)}{b+\Re(c)} \Big]_{x=-1}^{x=1}\\
		&=\Bigg( \frac{\Re(c^2)}{2\abs{c}^4} \Bigg)\Bigg( \frac{(1-\Im(c))}{(1-\Im(c))^2+(b+\Re(c))^2}+\frac{(1+\Im(c))}{(1+\Im(c))^2+(b+\Re(c))^2} \Bigg) \\
		&\quad +\frac{\Re(c^3)}{2\abs{c}^6}\Bigg( \arctan\big( \frac{1-\Im(c)}{b+\Re(c)} \big)-\arctan\big( \frac{-1-\Im(c)}{b+\Re(c)} \big)\Bigg).
	\end{align*}
	Since $\tilde{h}_1+\tilde{h}_4$ is just $h_1+h_4$ with $c$ replaced by $-c$, we have
	\begin{align*}
		\int\limits_{-1}^1 \tilde{h}_1(x)+\tilde{h}_4(x)\d x&=\Bigg( \frac{\Re(c^2)}{2\abs{c}^4} \Bigg)\Bigg( \frac{(1+\Im(c))}{(1+\Im(c))^2+(b-\Re(c))^2}+\frac{(1-\Im(c))}{(1-\Im(c))^2+(b-\Re(c))^2} \Bigg) \\
		&\quad -\frac{\Re(c^3)}{2\abs{c}^6}\Bigg( \arctan\big( \frac{1+\Im(c)}{b-\Re(c)} \big)-\arctan\big( \frac{-1+\Im(c)}{b-\Re(c)} \big)\Bigg).
	\end{align*}
	As for the terms $h_2$ and $\tilde{h}_2$, we have
	\begin{align*}
		\int\limits_{-1}^1 h_2(x)\d x&= \Bigg( \frac{\Im(c^3)}{4\abs{c}^6} \Bigg)  \Big[ \ln\big( (x-\Im(c))^2+(b+\Re(c))^2 \big) \Big]_{x=-1}^{x=1} \\
		&=\Bigg( \frac{\Im(c^3)}{4\abs{c}^6} \Bigg)   \ln\bigg( \frac{(1-\Im(c))^2+(b+\Re(c))^2}{(1+\Im(c))^2+(b+\Re(c))^2 } \bigg) \\
		\intertext{and, thus, }
		\int\limits_{-1}^1 \tilde{h}_2(x)\d x&= -\Bigg( \frac{\Im(c^3)}{4\abs{c}^6} \Bigg)   \ln\bigg( \frac{(1+\Im(c))^2+(b-\Re(c))^2}{(1-\Im(c))^2+(b-\Re(c))^2 } \bigg).
	\end{align*}
	Lastly, the terms $h_3$ and $\tilde{h}_3$ integrate to 
	\begin{align*}
		\int\limits_{-1}^1 h_3(x)\d x&=  \Bigg( \frac{\Im(c^2)\,\big( b+\Re(c)  \big)}{2\abs{c}^4} \Bigg)\Big[\frac{-1}{(x-\Im(c))^2+(b+\Re(c))^2}\Big]_{x=-1}^{x=1}\\
		&=-\Bigg( \frac{\Im(c^2)\,\big( b+\Re(c)  \big)}{2\abs{c}^4} \Bigg)\Bigg(\frac{1}{(1-\Im(c))^2+(b+\Re(c))^2}-\frac{1}{(1+\Im(c))^2+(b+\Re(c))^2}\Bigg)\\
		\intertext{and, similarly, }
		\int\limits_{-1}^1 \tilde{h}_3(x)\d x&=  \Bigg( \frac{\Im(c^2)\,\big( b-\Re(c)  \big)}{2\abs{c}^4} \Bigg)\Bigg(\frac{1}{(1-\Im(c))^2+(b-\Re(c))^2}-\frac{1}{(1+\Im(c))^2+(b-\Re(c))^2}\Bigg).
	\end{align*}
	Therefore, we find that 
	\begin{equation*}
		\int\limits_{-1}^1 F_{a,b}(x)\d x= \I{}+\II{}+\III{}+\IV{}+\V{}+\VI{},
	\end{equation*}
	where
	\begin{align*}
		\I{}&:=\frac{1}{(1-\Im(c))^2+(b-\Re(c))^2}\frac{1}{2\abs{c}^4}\Bigg(\Im(c^2)\,\big( b-\Re(c)  \big) +\Re(c^2)\,(1-\Im(c)) \Bigg) \\
		\II{}&:=\frac{1}{(1+\Im(c))^2+(b-\Re(c))^2}\frac{1}{2\abs{c}^4}\Bigg( -\Im(c^2)\,\big( b-\Re(c)  \big)+\Re(c^2)(1+\Im(c)) \Bigg) \\
		\III{}&:=\frac{1}{(1-\Im(c))^2+(b+\Re(c))^2}\frac{1}{2\abs{c}^4}\Bigg( \Re(c^2)(1-\Im(c)) -\Im(c^2)(b+\Re(c))\Bigg) \\
		\IV{}&:=\frac{1}{(1+\Im(c))^2+(b+\Re(c))^2}\frac{1}{2\abs{c}^4}\Bigg( \Im(c^2)(b+\Re(c))+\Re(c^2)(1+\Im(c)) \Bigg) \\
		\V{}&:=  \frac{\Im(c^3)}{4\abs{c}^6}    \ln\bigg( \frac{(1-\Im(c))^2+(b+\Re(c))^2}{(1+\Im(c))^2+(b+\Re(c))^2 }    \frac{(1-\Im(c))^2+(b-\Re(c))^2 }{(1+\Im(c))^2+(b-\Re(c))^2}\bigg)\\
		\VI{}&:= \frac{\Re(c^3)}{2\abs{c}^6}\Bigg( \arctan\big( \frac{1-\Im(c)}{b+\Re(c)} \big)+\arctan\big( \frac{1+\Im(c)}{b+\Re(c)} \big)\\
		&\qquad\qquad\quad-    \arctan\big( \frac{1+\Im(c)}{b-\Re(c)} \big)+\arctan\big( \frac{-1+\Im(c)}{b-\Re(c)} \big)  \Bigg).
	\end{align*}
	Note that 
	\begin{align*}
		\Re(c^3)&= \Re(c^2\cdot c)= \Re(c^2)\, \Re(c)- \Im(c^2)\,\Im(c)\\
		&=\frac{1+b^2}{\sqrt{2}}\sqrt{\sqrt{(1+b^2)^2+4a^2}+1+b^2} - \sqrt{2}\abs{a}\sqrt{\sqrt{(1+b^2)^2+4a^2}-1-b^2}\\
		\intertext{and}
		\Im(c^3)&= \Im(c^2)\,\Re(c)+ \Re(c^2)\,\Im(c)\\
		&= \sqrt{2} a \sqrt{\sqrt{ (1+b^2)^2+4a^2 } +1+b^2 } +\frac{1+b^2}{\sqrt{2}} \sgn(a)\sqrt{ \sqrt{ (1+b^2)^2+4a^2 }-1-b^2  }.
	\end{align*}
	Additionally, we simplify
	\begin{align*}
		\abs{c}^2= \sqrt{(1+b^2)^2+4a^2}.
	\end{align*}
	Therefore, plugging this into our results for each term $\I{}, \dots, \IV{}$ yields
	\begin{align*}
		\I{}&= \frac{1}{(1+b^2)^2+4a^2}\\
		&\quad\times\frac{-\sqrt{2}\,a\,\big( \sqrt{\sqrt{(1+b^2)^2+4a^2}+1+b^2} -\sqrt{2}\,b \big) +(\frac{1+b^2}{\sqrt{2}})\,(\sqrt{2}-\sqrt{\sqrt{(1+b^2)^2+4a^2}-1-b^2})  }{  \Big( \sqrt{2}-\sqrt{\sqrt{(1+b^2)^2+4a^2}-1-b^2}  \Big)^2+ \Big( \sqrt{ \sqrt{(1+b^2)^2+4a^2}+1+b^2 } -\sqrt{2}\,b \Big)^2 },\\
		\II{}&=\frac{1}{(1+b^2)^2+4a^2}\\
		&\quad\times\frac{\sqrt{2}\,a\,\big( \sqrt{\sqrt{(1+b^2)^2+4a^2}+1+b^2} -\sqrt{2}\,b \big) +(\frac{1+b^2}{\sqrt{2}})\,(\sqrt{2}+\sqrt{\sqrt{(1+b^2)^2+4a^2}-1-b^2})  }{  \Big( \sqrt{2}+\sqrt{\sqrt{(1+b^2)^2+4a^2}-1-b^2}  \Big)^2+ \Big( \sqrt{ \sqrt{(1+b^2)^2+4a^2}+1+b^2 } -\sqrt{2}\,b \Big)^2 },\\
		\III{}&= \frac{1}{(1+b^2)^2+4a^2}\\
		&\quad\times\frac{-\sqrt{2}\,a\,\big( \sqrt{\sqrt{(1+b^2)^2+4a^2}+1+b^2} +\sqrt{2}\,b \big) +(\frac{1+b^2}{\sqrt{2}})\,(\sqrt{2}-\sqrt{\sqrt{(1+b^2)^2+4a^2}-1-b^2})  }{  \Big( \sqrt{2}-\sqrt{\sqrt{(1+b^2)^2+4a^2}-1-b^2}  \Big)^2+ \Big( \sqrt{ \sqrt{(1+b^2)^2+4a^2}+1+b^2 } +\sqrt{2}\,b \Big)^2 },\\
		\IV{}&= \frac{1}{(1+b^2)^2+4a^2}\\
		&\quad\times\frac{\sqrt{2}\,a\,\big( \sqrt{\sqrt{(1+b^2)^2+4a^2}+1+b^2} +\sqrt{2}\,b \big) +(\frac{1+b^2}{\sqrt{2}})\,(\sqrt{2}+\sqrt{\sqrt{(1+b^2)^2+4a^2}-1-b^2})  }{  \Big( \sqrt{2}+\sqrt{\sqrt{(1+b^2)^2+4a^2}-1-b^2}  \Big)^2+ \Big( \sqrt{ \sqrt{(1+b^2)^2+4a^2}+1+b^2 } +\sqrt{2}\,b \Big)^2 }.
	\end{align*}
	Finally, using the definition of $c$ and the previous calculations, we find 
	\begin{align*}
		&\frac{(1-\Im(c))^2+(b+\Re(c))^2}{(1+\Im(c))^2+(b+\Re(c))^2 }    \frac{(1-\Im(c))^2+(b-\Re(c))^2 }{(1+\Im(c))^2+(b-\Re(c))^2}\\
		&\quad =\frac{(\sqrt{2}-\sqrt{\sqrt{ (1+b^2)^2+4a^2} -1-b^2 })^2+(\sqrt{2}b+\sqrt{\sqrt{ (1+b^2)^2 +4a^2 } + 1+b^2})^2}{(\sqrt{2}+\sqrt{\sqrt{ (1+b^2)^2+4a^2} -1-b^2 })^2+(\sqrt{2}b+\sqrt{\sqrt{ (1+b^2)^2 +4a^2 } + 1+b^2})^2 }  \\
		&\qquad\times\frac{(\sqrt{2}-\sqrt{\sqrt{ (1+b^2)^2+4a^2} -1-b^2 })^2+(\sqrt{2}b-\sqrt{\sqrt{ (1+b^2)^2 +4a^2 } + 1+b^2})^2 }{(\sqrt{2}+\sqrt{\sqrt{ (1+b^2)^2+4a^2} -1-b^2 })^2+(\sqrt{2}b-\sqrt{\sqrt{ (1+b^2)^2 +4a^2 } + 1+b^2})^2}.
	\end{align*}
	This term equals after expanding 
	\begin{align*}
		&\frac{ \Big( 2(1+b^2)+2\sqrt{(1+b^2)^2+4a^2}-2\sqrt{2}\sqrt{\sqrt{(1+b^2)^2+4a^2}-1-b^2} \Big)^2\hspace{-5pt}-8b^2\big(\sqrt{(1+b^2)^2+4a^2}+1+b^2\big) }{ \Big( 2(1+b^2)+2\sqrt{(1+b^2)^2+4a^2}+2\sqrt{2}\sqrt{\sqrt{(1+b^2)^2+4a^2}-1-b^2} \Big)^2\hspace{-5pt}-8b^2\big(\sqrt{(1+b^2)^2+4a^2}+1+b^2\big)}\\
		&\quad = 1- \frac{2}{1+ \frac{ \sqrt{2} a^2+\sqrt{2}\sqrt{(1+b^2)^2+4a^2} }{\sqrt{ \sqrt{(1+b^2)^2+4a^2}-1-b^2 } \big(\sqrt{(1+b^2)^2+4a^2}+1+b^2\big)  }}= 1- \frac{2\sqrt{2}}{\sqrt{2}+ \frac{  a^2+\sqrt{(1+b^2)^2+4a^2} }{a\sqrt{ \sqrt{(1+b^2)^2+4a^2}+1+b^2 }   }}.
	\end{align*}
	Using this calculation and the definition of $c$ yields the correct expression for the term $\V{}$. Furthermore, we plug the definition of $c$ and use the calculation for $\Re(c^3)$ into the term $\VI{}$. This gives the proposed term involving the $\arctan$ functions in the statement of the lemma. Finally, we combine the terms $\I{}, \dots, \IV{}$ to find
	\begin{align*}
		&\I{}+\II{}+\III{}+\IV{}= -\frac{1}{(1+b^2)^2+4a^2} \frac{1}{ a^4 - 2a^2(b^2-1) + (1 + b^2)^2 }\\
		&\times \Bigg[ (-(1 + b^2)^3 + a^2(b^4-1) + a^3\sqrt{\sqrt{(1 + b^2)^2+4a^2 }-1 - b^2 }  \sqrt{\sqrt{(1 + b^2)^2+4a^2}+1 + b^2} \\
		&\qquad+ a(1 + b^2)\,\sqrt{\sqrt{(1 + b^2)^2+4a^2}-1 - b^2 }\,\sqrt{\sqrt{(1 + b^2)^2+4a^2}+1 + b^2}) \Bigg]\\
		&= \frac{1}{(1+b^2)^2+4a^2} \frac{(b^2+1-a^2)(1+b^2)^2-2a^4}{ a^4 - 2a^2(b^2-1) + (1 + b^2)^2 }.
	\end{align*}
\end{proof}

%\bibliographystyle{alpha}
%\bibliography{literatur}
%	
\end{document}